\documentclass[12pt]{amsart}
\usepackage{amssymb, amscd, mathpazo}
\usepackage{fullpage}
\usepackage{enumerate} 
\usepackage{url}

\title[Equivariant degenerations of spherical modules in type $\ssA$]{Equivariant degenerations of spherical modules for groups of type $\ssA$ }

\author{Stavros Argyrios Papadakis}
\address{Centro de An\'{a}lise Matem\'atica, Geometria e Sistemas
Din\^{a}micos,
Departamento de Ma\-te\-m\'atica, Instituto Superior T\'ecnico,
Universidade T\'ecnica de Lisboa,
Av. Rovisco Pais, 1049-001 Lisboa,
Portugal}
\email{papadak@math.ist.utl.pt}

\author{Bart Van Steirteghem}
\address{Department of Mathematics, Medgar Evers College - City University of New York, 1650 Bedford Ave., Brooklyn, NY 11225, USA}
\email{bartvs@mec.cuny.edu}

\date{22 December 2011}

\input xy
\xyoption{all}

\numberwithin{equation}{section}

\newcommand{\inn}{\subseteq}
\newcommand{\isom}{\simeq}
\newcommand{\quot}{/\!\!/}
\newcommand{\onto}{\twoheadrightarrow}
\newcommand{\into}{\hookrightarrow}

\renewcommand{\>}{\rangle}
\newcommand{\<}{\langle}

\renewcommand{\epsilon}{\varepsilon}

\DeclareMathOperator{\im}{im}

\DeclareMathOperator{\rk}{rk}
\DeclareMathOperator{\Spec}{Spec}
\DeclareMathOperator{\ad}{ad}
\DeclareMathOperator{\Hom}{Hom}
\DeclareMathOperator{\Lie}{Lie}
\DeclareMathOperator{\Mor}{Mor}
\DeclareMathOperator{\Sym}{Sym}
\DeclareMathOperator{\Aut}{Aut}

\DeclareMathOperator{\Coind}{Coind}
\newcommand{\mf}{\mathfrak}
\newcommand{\fa}{\mf{a}}
\newcommand{\fb}{\mf{b}}

\newcommand{\fg}{\mf{g}}
\newcommand{\fh}{\mf{h}}

\newcommand{\fk}{\mf{k}}

\newcommand{\ft}{\mf{t}}
\newcommand{\fu}{\mf{u}}

\newcommand{\fz}{\mf{z}}

\renewcommand{\sl}{\mf{sl}}

\newcommand{\msf}{\mathsf}
\newcommand{\ssA}{\msf{A}}

\newcommand{\GL}{\mathrm{GL}}
\newcommand{\SL}{\mathrm{SL}}
\newcommand{\SO}{\mathrm{SO}}

\newcommand{\shE}{\mathcal{E}}
\newcommand{\shF}{\mathcal{F}}
\newcommand{\shI}{\mathcal{I}}

\newcommand{\shN}{\mathcal{N}}
\newcommand{\shM}{\mathcal{M}}
\newcommand{\shO}{\mathcal{O}}
\newcommand{\Oh}{\mathcal{O}}

\newcommand{\cS}{\mathcal{S}}

\renewcommand{\AA}{{\mathbb A}}
\newcommand{\CC}{{\mathbb C}}
\newcommand{\NN}{{\mathbb N}}
\newcommand{\PP}{{\mathbb P}}
\newcommand{\QQ}{{\mathbb Q}}

\newcommand{\ZZ}{{\mathbb Z}}
\newcommand{\GGm}{{\mathbb G}_m}

\newcommand{\om}{\omega}

\newcommand{\tHilb}{\mathrm{Hilb}}
\newcommand{\tM}{\mathrm{M}}
\newcommand{\st}{\mathrm{st}}
\newcommand{\wm}{\mathcal{S}}
\newcommand{\wmo}{\overline{\wm}}
\newcommand{\wg}{\Delta}
\newcommand{\wgo}{\overline{\wg}}
\newcommand{\Tad}{T_{\ad}}
\newcommand{\Vgg}{(V/\fg\cdot x_0)^{G_{x_0}}}
\newcommand{\Vggo}{(V/\fg\cdot x_0)^{\overline{G}_{x_0}}}
\newcommand{\Vggp}{(V/\fg\cdot x_0)^{G'_{x_0}}}
\newcommand{\Vg}{V/\fg \cdot x_0}
\DeclareMathOperator{\Mat}{Mat}

\newcommand{\ESo}{{ES1}}
\newcommand{\ESt}{{ES2}}
\newcommand{\ESth}{{ES3}}
\newcommand{\ESf}{{ES4}}

\newcommand{\SRo}{{SR1}}
\newcommand{\SRt}{{SR2}}
\newcommand{\SRth}{{SR3}}
\newcommand{\SRf}{{SR4}}

\newcommand{\Sch}{\mathrm{(Schemes)}}  
\newcommand{\Cat}{\mathrm{(Cat)}}
\newcommand{\Sets}{\mathrm{(Sets)}}

\theoremstyle{plain}
\newtheorem{theorem}{Theorem}[section]
\newtheorem{lemma}[theorem]{Lemma}
\newtheorem{prop}[theorem]{Proposition}
\newtheorem{cor}[theorem]{Corollary}

\theoremstyle{definition}
\newtheorem{defn}[theorem]{Definition}
\newtheorem{remark}[theorem]{Remark}

\newtheorem{example}[theorem]{Example}

\newtheorem{listabc}[theorem]{List}

\begin{document}

\begin{abstract}
Let $G$ be a complex reductive algebraic group. Fix a
Borel subgroup $B$ of $G$ and a maximal torus $T$ in $B$. Call the monoid of dominant weights $\Lambda^+$ and let $\wm$ be a finitely generated submonoid of $\Lambda^+$.
V.~Alexeev and M.~Brion introduced a moduli scheme $\tM_{\wm}$ which
classifies affine $G$-varieties $X$ equipped with a
$T$-equivariant isomorphism $\Spec \CC[X]^U\to \Spec \CC[\wm]$, where $U$ is
the unipotent radical of $B$.
Examples of $\tM_{\wm}$ have been obtained by S. Jansou, P. Bravi
and S. Cupit-Foutou. In this paper, we prove that $\tM_{\wm}$ is isomorphic to an affine space  when
$\wm$ is the weight monoid of a spherical $G$-module
with $G$ of type $\ssA$. Unlike the earlier examples,
this includes cases where $\wm$ does not satisfy the condition $\<\wm\>_{\ZZ} \cap
\Lambda^+ = \wm$.
\end{abstract}

\maketitle

\section{Introduction and statement of results}
As part of the classification of affine $G$-varieties $X$, where $G$ is a complex connected reductive group,  a natural question is to what extent the $G$-module structure of the ring $\CC[X]$ of regular functions on $X$ determines $X$. Put differently, to what extent does the $G$-module structure of $\CC[X]$ determine its algebra structure?

In the mid 1990s, F. Knop conjectured that the answer to this question is `completely' when  $X$ is a smooth affine \emph{spherical} variety. To be precise, \emph{Knop's Conjecture}, which has since been proved by I.~Losev~\cite{losev-knopconj}, says that if $X$ is a smooth affine $G$-variety such that the $G$-module $\CC[X]$ has no multiplicities, then this $G$-module uniquely determines the $G$-variety $X$ (up to $G$-equivariant isomorphism).  Knop also proved \cite{knop-autoHam} that the validity of his conjecture implies that of Delzant's Conjecture~\cite{delzant-rank2} about multiplicity-free symplectic manifolds. 

In~\cite{alexeev&brion-modaff}, V.~Alexeev and M.~Brion brought geometry to the general question. Given a maximal torus $T$ in $G$ and an affine $T$-variety $Y$ such that all $T$-weights in $\CC[Y]$ have finite multiplicity, they introduced a moduli scheme $\tM_{Y}$ which parametrizes (equivalence classes of) pairs $(X,\varphi)$, where $X$ is an affine $G$-variety and $\varphi\colon X\quot U \to Y$ is a $T$-equivariant isomorphism (here $U \inn G$ is a fixed maximal unipotent subgroup normalized by $T$ and $X\quot U := \Spec \CC[X]^U$ is the categorical quotient).  They also proved that $\tM_Y$ is an affine connected scheme, of finite type over $\CC$, and that the orbits of the natural action of $\Aut^T(Y)$ on $\tM_Y$ are in bijection with the isomorphism classes of affine $G$-varieties $X$ such that $X\quot U \isom Y$. See also \cite[Section 4.3]{brion-ihs-arxivv2} for more information on $\tM_Y$.

The first examples of $\tM_Y$ were obtained by S.~Jansou \cite{jansou-deformations}. He dealt with the following situation. Suppose $\Lambda^+$ is the set of dominant weights of $G$ (with respect to the Borel subgroup $B = TU$ of $G$) and let $\lambda \in \Lambda^+$. Jansou proved that if $Y = \CC$ with $T$ acting linearly with weight $-\lambda$, then $\tM_{\lambda}:=\tM_Y$ is a (reduced) point or an affine line. Moreover, he linked $\tM_Y$ to the theory of \emph{wonderful varieties} (see, e.g., \cite{bravi-luna-f4} or \cite{pezzini-cirmspher}) by showing that $\tM_{\lambda}$ is an affine line if and only if $\lambda$ is a spherical root for $G$.

P.~Bravi and S.~Cupit-Foutou~\cite{bravi&cupit} generalized Jansou's result as follows. 
Given a free submonoid $\wm$ of $\Lambda^+$ such that
\begin{equation} \label{eq:Psat}
\<\wm\>_{\ZZ} \cap \Lambda^+ =\wm,
\end{equation}
put $Y := \Spec \CC[\wm]$ and $\tM_{\wm} := \tM_Y$. Bravi and Cupit-Foutou proved that $\tM_{\wm}$ is isomorphic to an affine space. More precisely, the map $T \to \Aut^T(Y)$ coming from the action of $T$ on $Y$ induces an action of $T$ on $\tM_{\wm}$, and they proved that $\tM_{\wm}$ is (isomorphic to) a multiplicity-free representation of $T$ whose weight set is the set of spherical roots of a wonderful variety associated to $\wm$. 
The connections between the moduli schemes $\tM_Y$ and wonderful varieties have been studied further in \cite{cupit-ihswv-prep, cupit-wvgr-prep}.

In this paper we compute examples of $\tM_{\wm}$ where $\wm$ is a  free submonoid of $\Lambda^+$, but does not necessarily satisfy (\ref{eq:Psat}). To be more precise, we prove that $\tM_{Y}$ is (again) isomorphic to an affine space whenever $Y = W \quot U$ with $W$ a spherical $G$-module and $G$ of type $\ssA$ (see Theorem~\ref{thm:main} below for the precise statement). The reason we chose to work with spherical modules is that they have been classified (`up to central tori') and that many of their combinatorial invariants have been computed (see~\cite{knop-rmks}). We prove Theorem~\ref{thm:main} by reducing it to a case-by-case verification (Theorem~\ref{thm:cbc}). It turns out that in most of our cases, condition (\ref{eq:Psat}) is not satisfied. The fact that the classification of spherical modules is `up to central tori' means that this verification needs some care, see Section~\ref{sec:reduction} and Remark~\ref{rem:saturation}. In this paper we restrict ourselves to groups of type $\ssA$ because the work needed is already quite lengthy. The reduction of the proof of Theorem~\ref{thm:main} to the case-by-case analysis is independent of the type of $G$. 

The main consequence of the absence of condition (\ref{eq:Psat}) is that computing the tangent space to $\tM_{\wm}$ at its unique $T$-fixed point and unique closed $T$-orbit $X_0$, which is also the first step in the work of Jansou, and Bravi and Cupit-Foutou, becomes more involved (see Section~\ref{sec:Brionstrat} below). On the other hand, we know, by definition, that our moduli schemes $\tM_{\wm} = \tM_Y$ (where $Y=W\quot U$)  contain the closed point $(W,\pi)$ where $\pi\colon W\quot U \to Y$ is the identity map. By general results from \cite{alexeev&brion-modaff} this point has an (open) $T$-orbit of which we know the dimension $d_W$.  This implies that once we have determined that $\dim T_{X_0} \tM_{\wm} \le d_W$, our main result follows. Jansou and especially  Bravi--Cupit-Foutou have to do quite a bit more work (involving the existence of a certain wonderful variety depending on $\wm$) to prove that $\tM_{\wm}$ contains a $T$-orbit of the same dimension as $T_{X_0} \tM_{\wm}$.

\subsection{Notation and preliminaries}
We will consider algebraic groups and schemes over $\CC$. In addition, like in \cite{alexeev&brion-modaff}, all schemes will be assumed to be Noetherian. By a variety, we mean an integral separated scheme of
finite type over $\CC$. In particular, varieties are irreducible.

In this paper, unless stated otherwise, $G$ will be a connected reductive linear algebraic group over $\CC$ in which we have chosen a (fixed) maximal torus $T$ and a (fixed) Borel subgroup $B$ containing $T$.  We will use $U$ for the unipotent radical of $B$; it is a maximal unipotent subgroup of $G$. For an algebraic group $H$, we denote $X(H)$ the group of characters, that is, the set of all homomorphisms of algebraic groups $H \to \GGm$, where $\GGm$ denotes the
multiplicative group $\CC^{\times}$.
By a $G$-module or a representation of $G$ we will always mean a
(possibly infinite dimensional)
\emph{rational} $G$-module (sometimes also called a locally finite $G$-module).
For the definition, which applies to non-reductive groups too, see for example~\cite[p.86]{alexeev&brion-modaff}. Because $G$ is reductive, every $G$-module $E$ is the direct sum of irreducible (or \emph{simple}) $G$-submodules. We call $E$ \emph{multiplicity-free} if it is the direct sum of pairwise non-isomorphic simple $G$-modules. 

 We will use $\Lambda$ for the weight lattice $X(T)$ of $G$, which is naturally identified with $X(B)$, and $\Lambda^+$ for the submonoid of $X(T)$ of dominant weights (with respect to $B$). 
 Every $\lambda \in \Lambda^+$ corresponds to a unique irreducible representation of $G$, which we will denote $V(\lambda)$. It is specified by the property that $\lambda$ is its unique $B$-weight.
 Conversely every irreducible representation of $G$ is of the form $V(\lambda)$ for a unique $\lambda \in \Lambda^+$.
 Furthermore, we will use $v_{\lambda}$ for a highest weight vector in $V(\lambda)$. It is defined up to nonzero scalar: $V(\lambda)^U = \CC v_{\lambda}$. For $\lambda \in \Lambda^+$, we will use $\lambda^*$ for the highest weight of the dual $V(\lambda)^*$ of $V(\lambda)$. We then have that
 $\lambda^* = -w_0(\lambda)$, where $w_0$ is the longest element of the Weyl group $N_G(T)/T$ of $G$. For a $G$-module $M$ and $\lambda \in \Lambda^+$, we will use $M_{(\lambda)}$ for the isotypical component of $M$ of type $V(\lambda)$.

We denote the center of $G$ by $Z(G)$ and use $T_{\ad}$ for the adjoint torus $T/Z(G)$ of $G$. 
The set of simple roots of $G$ (with respect to $T$ and $B$) will be denoted $\Pi$, the set of positive roots $R^+$ and the root lattice $\Lambda_R$.  When $\alpha$ is a root,  $\alpha^{\vee} \in \Hom_{\ZZ}(\Lambda, \ZZ)$ will stand for its coroot. In particular, $\<\alpha, \alpha^{\vee}\> = 2$ where $\<\cdot,\cdot\>$ is the natural pairing between $\Lambda$ and its dual $\Hom_{\ZZ}(\Lambda, \ZZ)$ (which is naturally identified with the group of one-parameter subgroups of $T$).  

The Lie algebra of an algebraic group $G, H, T, B, U$ etc.\ will be denoted by the corresponding fraktur letter $\fg, \fh, \ft, \fb, \fu$, etc. At times, we will also use $\Lie(H)$ for the Lie algebra of $H$. For a reductive group $G$, we will use $G'$ for its derived group $(G,G)$. It is a semisimple group and its Lie algebra is $\fg' = [\fg,\fg]$. When $G$ acts on a set $X$ and $x \in X$, then $G_x$ stands for the isotropy group of $x$. We adopt the convention that $G'_x := (G')_x$ and analogous notations for $\fg$-actions. 
For every root $\alpha$ of $G$, we choose a non-zero element $X_{\alpha}$ of the (one-dimensional) root space $\fg^{\alpha} \inn \fg$. We call $X_{\alpha}$ a \emph{root operator}. 

A reductive group $G$ is said to be of \emph{type  $\ssA$} if $\fg'$ is $0$ or isomorphic to a direct sum
\[\sl(n_1) \oplus \sl(n_2) \oplus\ldots\oplus \sl(n_k)\]
for some positive integer $k$ and integers $n_i \ge 2$ ($1\le i \le k$).   

When $G=\SL(n)$ and $i\in \{1,\ldots,n-1\}$,  we denote $\omega_i$ the highest weight of the module $\bigwedge^i \CC^n$. In addition, for $\SL(n)$ we put $\omega_n = \omega_0 = 0$. 
Similarly, when $G = \GL(n)$ and $i \in \{1,\ldots, n\}$, the highest weight of the module $\bigwedge^i \CC^n$ will also be denoted $\omega_i$.  The set $\{\omega_1,\ldots,\omega_n\}$ forms a  basis of the weight lattice $\Lambda$ of $\GL(n)$. Moreover, we put $\omega_0= 0$. It is well-known that the simple roots of $\GL(n)$ have the following expressions in terms of the $\omega_i$:
\begin{equation} \label{eq:simpleroots}
 \alpha_i = -\omega_{i-1} + 2\omega_i - \omega_{i+1} \quad \text{ for } i\in\{1,2,\ldots,n-1\},
\end{equation} 
and that the same formulas also hold for $\SL(n)$. The representations $V(\omega_i)$ are called the \emph{fundamental representations} of $\GL(n)$ (resp.\ $\SL(n)$). 

A finitely generated $\CC$-algebra $A$ is called a \emph{$G$-algebra} if it comes equipped with an action of $G$ (by automorphisms) for which $A$ is a rational $G$-module. The \emph{weight set} of $A$ is then defined as
\[\Lambda^+_A:= \{\lambda \in \Lambda^+ \colon A_{(\lambda)} \neq 0
\}.\]
Such an algebra $A$ is called \emph{multiplicity-free} if it is multiplicity-free as a $G$-module.  
When the $G$-algebra $A$ is an integral domain, the multiplication on $A$ induces a
monoid structure on $\Lambda^+_A$, which we then call  the \emph{weight
  monoid} of A; it is a finitely generated submonoid of $\Lambda^+$
(see e.g.~
\cite[Corollary 2.8]{brion-cirmactions}).

For an affine scheme $X$, we will use $\CC[X]$ for its ring of regular functions. In particular,
$X = \Spec \CC[X]$.
As in \cite{alexeev&brion-modaff}, an \emph{affine $G$-scheme} is an affine scheme $X$ of finite type equipped with an action of $G$. Then $\CC[X]$ is a $G$-algebra for the following action:
\[(g\cdot f)(x) = f(g^{-1}\cdot x)\quad \text{for $f\in \CC[X], g\in G$ and $x \in X$}.\]
We remark that even when $G$ is abelian we use this action on $\CC[X]$. A \emph{$G$-variety} is a variety equipped with an action of $G$.
If $X$ is an affine $G$-scheme, then
its \emph{weight set} $\Lambda^+_{(G,X)}$ is defined, like in \cite[p.87]{alexeev&brion-modaff}, as
the weight set of the $G$-algebra $\CC[X]$. If $X$ is an 
affine $G$-variety, then we call $\Lambda^+_{(G,X)}$ its \emph{weight monoid},
and the \emph{weight group} $\Lambda_{(G,X)}$ of $X$ is defined
as the subgroup of $X(T)$ generated by $\Lambda_{(G,X)}^+$. It is
well--known that $\Lambda_{(G,X)}$  is also equal to the set of
$B$-weights  in the function field of $X$ (see e.g.~\cite[p. 17]{brion-cirmactions}).
When no confusion can arise about the group $G$ in question, we will use $\Lambda^+_X$ and $\Lambda_X$ for $\Lambda^+_{(G,X)}$ and $\Lambda_{(G,X)}$, respectively. 
An affine $G$-scheme $X$ is called \emph{multiplicity-free} if $\CC[X]$ is multiplicity-free as a $G$-module. 
An affine $G$-variety is multiplicity-free if and only if it has a dense $B$-orbit.
We call  a $G$-variety \emph{spherical} if it is normal and has a dense orbit for $B$. 
A \emph{spherical $G$-module} is a finite-dimensional $G$-module that is spherical as a $G$-variety.  
We remark that if $W$ is a spherical $G$-module, then any two distinct simple $G$-submodules of $W$ are non-isomorphic. For general information on spherical varieties we refer to \cite[Section 2]{brion-cirmactions} and \cite{pezzini-cirmspher}.

The indecomposable saturated spherical modules were classified up to geometric equivalence by Kac, Benson-Ratcliff and Leahy~\cite{kac-rmks, benson-ratcliff-mf, leahy}, see~\cite{knop-rmks} for an overview or Section~\ref{sec:reduction} for the definitions of these terms. We will use Knop's presentation in \cite[\S5]{knop-rmks} of this classification and refer to it as {\bf Knop's List}. For groups of type $\ssA$ we recall the classification in List~\ref{KnopLTypeA} on page~\pageref{KnopLTypeA}.

When $H$ is a torus and $M$ is a finite-dimensional $H$-module, then by the \emph{$H$-weight set} of $M$, we mean the (finite) set of elements $\lambda$ of $X(H)$ such that $M_{(\lambda)} \neq 0$. For the weight monoid $\Lambda^+_M$ of $M$ (seen as an $H$-variety) we then have that
\[\Lambda^+_M=\<-\lambda | \lambda \text{ is an element of the $H$-weight set of $M$}\>_{\NN}.\]

Given an affine $T$-scheme $Y$ such that each $T$-eigenspace in
$\CC[Y]$ is finite-dimensional, Alexeev and
Brion~\cite{alexeev&brion-modaff} introduced a moduli scheme $\tM_{Y}$
which classifies (equivalence classes of) pairs $(X,\varphi)$, where
$X$ is an affine $G$-scheme and $\varphi\colon X \quot U \to Y$ is a
$T$-equivariant isomorphism. Here 
$X\quot U := \Spec(\CC[X]^U)$ is the categorical quotient.  Moreover, they proved that $\tM_Y$ is a connected, affine scheme of finite type over $\CC$ and they equipped it with an action by $\Tad$, induced by the action of $\Aut^T(Y)$ on $\tM_Y$ and the map $T \to \Aut^T(Y)$.  We call $\gamma$ this action of $\Tad$ on $\tM_Y$ (see Section~\ref{subsec:emb} for details).

Now, suppose $\wm$ is a finitely generated submonoid of $\Lambda^+$ and $Y = \Spec \CC[\wm]$ is the multiplicity-free $T$-variety with weight monoid $\wm$. Like \cite{alexeev&brion-modaff}, we then put \[\tM_{\wm}:= \tM_{Y}.\]
We will use $\tM^G_{\wm}$ for $\tM_{\wm}$ when we want to stress the group under consideration. 

We need to define one more combinatorial invariant of affine $G$-varieties. Let $X$ be  such a variety.  Put $R:=\CC[X]$ and define the \emph{root monoid} $\Sigma_X$ of $X$ as the submonoid of $X(T)$ generated by
\[\{\lambda+\mu-\nu\in \Lambda\ |\ \lambda,\mu,\nu \in \Lambda^+: \<R_{(\lambda)}R_{(\mu)} \>_{\CC} \cap R_{(\nu)} \neq 0\},\]
where $\<R_{(\lambda)}R_{(\mu)} \>_{\CC}$ denotes the $\CC$-vector subspace of $R$ spanned by the set $\{fg\ | \ f\in R_{(\lambda)}, g\in  R_{(\mu)}\}$. Note that $\Sigma_X \inn \<\Pi\>_{\NN}$. 
We call $d_X$ the rank of the (free) abelian group generated (in $X(T)$) by $\Sigma_X$, that is,
\[d_X := \rk\<\Sigma_X\>_{\ZZ}. \]
We remark that for a given spherical module $W$, the invariant $d_W$ is easy to calculate from the rank of $\Lambda_W$, see Lemma~\ref{lem:camusbound}.

\subsection{Main results}
The main result of the present paper is the following theorem. Its formal proof will be given in Section~\ref{subsec:proofthmmain}.
\begin{theorem} \label{thm:main} 
Assume $W$ is a spherical $G$-module, where $G$ is a
connected reductive algebraic group of type $\ssA$.  Let
$\wm$ be the weight monoid of  $W$. Then
\begin{enumerate}[(a)]
\item $\Sigma_W$ is a freely generated monoid; and
\item the $\Tad$-scheme $\tM_{\wm}$, where the action is $\gamma$,  is $\Tad$-equivariantly isomorphic to the $\Tad$-module with weight monoid $\Sigma_W$. In particular, the scheme $\tM_{\wm}$ is
isomorphic to  the affine space $\AA^{d_W}$, hence it is irreducible and smooth. 
\end{enumerate} 
\end{theorem}

Our strategy for the proof of Theorem~\ref{thm:main} is as follows.  Suppose $W$ is a spherical module with weight monoid $\wm$. Because $\dim \tM_{\wm} \ge d_W$, it is sufficient to prove that $\dim T_{X_0} \tM_{\wm} \le d_W$, where $X_0$ is the unique $\Tad$-fixed point and the unique closed $\Tad$-orbit in $\tM_{\wm}$ (see Corollary~\ref{cor:apriori}). In Section~\ref{sec:reduction} (see Corollary~\ref{cor:mainthm}) we further reduce the proof of Theorem~\ref{thm:main} to the following theorem.
\begin{theorem}\label{thm:cbc}
Suppose $(\overline{G},W)$ is an entry in Knop's  List of saturated indecomposable spherical modules with $\overline{G}$ of type $\ssA$ (see List~\ref{KnopLTypeA} on page~\pageref{KnopLTypeA}).
If $G$ is a connected reductive group such that
\begin{enumerate}
\item $\overline{G}' \inn G \inn \overline{G}$; and
\item $W$ is spherical as a $G$-module
\end{enumerate} 
then $$\dim T_{X_0}\tM^{G}_{\wm} = d_W,$$ where $\wm$ is the weight monoid of $(G,W)$.  
\end{theorem}
\noindent In Section~\ref{sec:cases} we will prove this theorem case-by-case 
for the $8$ families of spherical modules in Knop's List  with $\overline{G}$ of type $\ssA$. 

For that purpose $X_0$ is identified in Section~\ref{subsec:emb} with the closure of a certain orbit $G\cdot x_0$ in  a certain $G$-module $V$ and  $T_{X_0}\tM_{\wm}$ with the vector space of $G$-invariant global sections of the normal sheaf of $X_0$ in $V$. This is a subspace of the space of $G$-invariant sections of the same sheaf over $G\cdot x_0$. This latter space is naturally identified with $\Vgg$. 
In Section~\ref{sec:cases} we use the $\Tad$-action (more precisely a variation of it) to bound $\Vgg$ by explicit computations for the pairs $(G,W)$ in the statement of Theorem~\ref{thm:cbc}. In most cases we find that already $\dim \Vgg \le d_W$. To obtain the desired inequality for $\dim T_{X_0}\tM_{\wm}$ in the remaining cases we use the exclusion criterion of Section~\ref{sec:Brionstrat}, which was suggested to us by M. Brion, to prove that enough sections over $G\cdot x_0$ do not extend to $X_0$. 

\subsection{Formal proof of Theorem~\ref{thm:main}} \label{subsec:proofthmmain}
We now give the proof of Theorem~\ref{thm:main}. Corollary~\ref{cor:apriori} and Corollary~\ref{cor:mainthm} reduce the proof to Theorem~\ref{thm:cbc}, which we prove by a case-by-case verification in Section~\ref{sec:cases}.

\subsection{Structure of the paper}
In Section~\ref{sec:tools} we present known results, mostly from \cite{alexeev&brion-modaff} and \cite{bravi&cupit}, in the form we need them. In Section~\ref{sec:Brionstrat}, which may be of independent interest, we formulate a criterion about non-extension of invariant sections of  the normal sheaf. In Section~\ref{sec:reduction} we review the known classification of spherical modules~\cite{kac-rmks, benson-ratcliff-mf, leahy} as presented in~\cite{knop-rmks} and reduce the proof of Theorem~\ref{thm:main} to a case-by-case verification. We perform this case-by-case analysis in Section~\ref{sec:cases}, using results from \cite{bravi&cupit} mentioned in Section~\ref{sec:tools} and, for the most involved cases, also the exclusion criterion of Section~\ref{sec:Brionstrat}.

\section{From the literature} \label{sec:tools}

In this section we gather known results, mostly from~\cite{alexeev&brion-modaff} and \cite{bravi&cupit}, together with immediate consequences relevant to our purposes. In particular we explain
that to prove Theorem~\ref{thm:main} it is sufficient to show that $\tM_{\wm}$ is smooth when $\wm$ is the weight monoid of a spherical module $W$ for $G$ of type $\ssA.$ Indeed, \cite[Corollary 2.14]{alexeev&brion-modaff} then implies Theorem~\ref{thm:main} (see Corollary~\ref{cor:apriori}). That result of Alexeev and Brion's also tells us that $\dim \tM_{\wm} \ge d_W$.  Moreover, by \cite[Theorem 2.7]{alexeev&brion-modaff}, we only have to prove smoothness at a specific point $X_0$ of $\tM_{\wm}$ (see Corollary~\ref{cor:smoothX0}), and for that it is enough to show that
\begin{equation} 
\label{eq:maineq}
\dim T_{X_0}\tM_{\wm} \le d_W. 
\end{equation}

Here is an overview of the content of this section. In Sections
\ref{subsec:emb} and \ref{subsec:Tapprox} we recall known facts
(mostly from \cite{alexeev&brion-modaff}) about the moduli scheme
$\tM_{\wm}$ when $\wm$ is a freely generated submonoid
of $\Lambda^+$ and apply them to the case where $\wm$ is the weight monoid $\Lambda^+_W$ of a spherical $G$-module $W$. More specifically, in  Section~\ref{subsec:emb} we identify $\tM_{\wm}$ with a certain open subscheme of an invariant Hilbert scheme $\tHilb_{\wm}^G(V)$, where $V$ is a specific finite-dimensional $G$-module determined by $\wm$. Under this identification, the point $X_0$ of $\tM_{\wm}$ corresponds to a certain $G$-stable subvariety of $V$, which we also denote $X_0$. Moreover, $X_0$ is the closure of the $G$-orbit of a certain point $x_0 \in V$. We then have that 
\[T_{X_0}\tM_{\wm} \isom H^0(X_0,\shN_{X_0})^G \into  H^0(G\cdot x_0,\shN_{X_0})^G\isom \Vgg,\] where $\shN_{X_0}$ is the normal sheaf of $X_0$ in $V$. In addition, following~\cite{alexeev&brion-modaff} we introduce an action of $\Tad$ on $\tM_{\wm}$.  In Section~\ref{subsec:Tapprox} we give some more details about the inclusion 
$H^0(X_0,\shN_{X_0})^G \into \Vgg$ which will be of use in Section~\ref{sec:Brionstrat} and in the case-by-case analysis of Section~\ref{sec:cases}. 
In Section~\ref{subsec:auxiliary} we collect some elementary technical lemmas on $\Vgg$ and the $\Tad$-action. Finally, in Section~\ref{subsec:bcf} we recall some results from~\cite{bravi&cupit} about $\Vgg$. 

\subsection{Embedding of $\tM_{\wm}$ into an invariant Hilbert scheme  and the $\Tad$-action}\label{subsec:emb}
Here we recall, from \cite{alexeev&brion-modaff}, that if $\wm$ is a freely generated submonoid of $\Lambda^+$, then $\tM_{\wm}$ can be identified with an open subscheme of a certain invariant Hilbert scheme $\tHilb^G_{\wm}(V)$. We also review the $\Tad$-action  on $\tHilb^G_{\wm}(V)$ defined in \cite{alexeev&brion-modaff}, its relation to the natural action of $\GL(V)^G$ on that Hilbert scheme and how it allows us to reduce the question of the smoothness  of $\tM_{\wm}$ to the question whether $\tM_{\wm}$ is smooth at a specific point $X_0$. 

Like the results in Sections \ref{subsec:Tapprox}, \ref{subsec:auxiliary} and \ref{sec:Brionstrat} everything in this section applies to any $\tM_{\wm}$ with $\wm$ freely generated.  In particular, by the following well-known proposition 
it applies to $\tM_{\wm}$ with $\wm = \Lambda^+_W$ when $(G,W)$ is a spherical module. For a proof, see \cite[Theorem 3.2]{knop-rmks}.
\begin{prop}
The weight monoid of a spherical module is freely generated; that is, it is generated by a set of linearly independent dominant weights.
\end{prop}

For the following, we fix a freely generated submonoid $\wm$ of $\Lambda^+$ and let $E^*$ be its (unique) basis.
Put $E = \{\lambda^*\ |\ \lambda \in E^*\}$ and  
 \[V = \oplus_{\lambda \in E} V(\lambda).\] 

Alexeev and Brion~\cite{alexeev&brion-modaff} introduced the invariant Hilbert scheme $\tHilb_{\wm}^G(V)$, which parametrizes all multiplicity-free closed $G$-stable subschemes $X$ of $V$ with weight set $\wm$ (they actually introduced the invariant Hilbert scheme in a more general setting; for more information on this object, see the survey \cite{brion-ihs-arxivv2}).
They also defined an action of $\Tad$ on $\tHilb_{\wm}^G(V)$, see \cite[Section 2.1]{alexeev&brion-modaff}, which we call $\gamma$ and now briefly review. It is obtained by lifting the natural action of $GL(V)^G$ on $\tHilb_{\wm}^G(V)$ to $T$. First, define the following homomorphism:
\begin{equation}
h\colon T \to \GL(V)^G,\ t \mapsto (-\lambda^*(t))_{\lambda \in E}  \label{eq:TGLVGAB}
\end{equation}
Composing the natural action of $\GL(V)^G$ on $V$ with $h$ yields an action $\phi$ of $T$ on $V$:
\[\phi(t,v) = h(t) \cdot v \quad \text{ for $t \in T$ and $v\in V$}.\]
 Note 
that $\phi$ is a linear action on $V$ and that each $G$-isotypical component  $V(\lambda^*)$ of $V^*$ (with $\lambda \in E$) is the $T$-weight space for $\phi$  of weight $\lambda^*$. Since $\GL(V)^G$ acts naturally on $\tHilb_{\wm}^G(V)$, $\phi$ induces an action of $T$ on $\tHilb_{\wm}^G(V)$. We call this last action $\gamma$. 
It has $Z(G)$ in its kernel and so descends to an action of $\Tad = T/Z(G)$ on $\tHilb_{\wm}^G(V)$ which we also call $\gamma$.  Indeed, if $\rho: G\to \GL(V)$ is the (linear) action of $G$ on $V$, then for every $z\in Z(G)$, $\rho(z) = h(z)$,
because $-\lambda^*= w_0\lambda$ is the lowest weight of $V(\lambda)$ and therefore differs from all other weights in $V(\lambda)$ by an element of $\<\Pi\>_{\NN}$. This implies that if $I$ is a $G$-stable ideal in $\CC[V]$, then $h(z)\cdot I = \rho(z)\cdot I = I$. More generally, if $S$ is a scheme with trivial $G$-action and $\shI$ is a $G$-stable ideal sheaf on $V\times S$, then $\shI$ is also stable under the action induced by $h$ on the structure sheaf $\shO_{V\times S}$, since $\rho|_{Z(G)}=h|_{Z(G)}$. Because $\tHilb_{\wm}^G(V)$ represents the functor $\Sch \to \Sets$ that associates to a scheme $S$ the set of flat families $Z \inn V \times S$ with invariant Hilbert function the characteristic function of $\wm \inn \Lambda^+$ (see \cite[Section 1.2]{alexeev&brion-modaff} or \cite[Section 2.4]{brion-ihs-arxivv2}), this implies our claim.

From \cite[Corollary 1.17]{alexeev&brion-modaff} we know that  the open subscheme $\tHilb^G_{E^*}$ of $\tHilb_{\wm}^G(V)$ that classifies the (irreducible) non-degenerate subvarieties $X \inn V$ with $\Lambda^+_X = \wm$ is stable under $\GL(V)^G$ and therefore under the $\Tad$-action $\gamma$. Recall from \cite[Definition 1.14]{alexeev&brion-modaff}  that a closed $G$-stable subvariety of $V$ is called \emph{non-degenerate} if its projections to the simple components $V(\lambda)$ of $V$, where $\lambda \in E$, are all nonzero. We call a closed $G$-stable subvariety of $V$ \emph{degenerate} if it is not non-degenerate. 

Next suppose $Y = \Spec\CC[\wm]$, the multiplicity-free $T$-variety with weight monoid $\wm$. Recall that $\tM_{\wm}=\tM_Y$ classifies (equivalence classes of) pairs $(X,\varphi)$ where $X$ is an affine $G$-variety and $\varphi\colon X\quot U \to Y$ is a $T$-equivariant isomorphism. The action of $T$ on $Y$ through $T \to \Aut^T(Y)$ induces an action of $T$ on $\tM_{\wm}$. From \cite[Lemma 2.2]{alexeev&brion-modaff} we know that this action descends to an action of $\Tad$ on $\tM_{\wm}$.  By Corollary 1.17 and Lemma 2.2 in \cite{alexeev&brion-modaff} the moduli scheme $\tM_{\wm}$ is $\Tad$-equivariantly isomorphic to $\tHilb^G_{E^*}$, where the $\Tad$-action on $\tHilb^G_{E^*}$ is $\gamma$. 
{\bf From now on}, we will identify $\tM_{\wm}$ with $\tHilb_{E^*}^G$. As in~\cite{bravi&cupit}, the $\Tad$-action it carries will play a fundamental role in what follows.

\begin{remark} \label{rem:Wdegnondeg}
\begin{enumerate}[(a)]
\item Let $(G,W)$ be a spherical module with weight monoid $\wm$, put $Y = W \quot U$ and let $\pi: W\quot U \to Y$ be the identity map. Then $(W,\pi)$ corresponds to a closed point of $\tM_Y = \tM_{\wm} = \tHilb_{E^*}^G \inn \tHilb^G_{\wm}(V)$. On the other hand, note that the highest weights of $W$ belong to $E$. Put
$E_1= \{\lambda \in \Lambda^+\colon W_{(\lambda)} \neq 0\} \inn E$ and $E_2 = E \setminus E_1$.  Then
\[V = \oplus_{\lambda \in E}V(\lambda) = [\oplus_{\lambda \in E_1}V(\lambda)] \oplus [\oplus_{\lambda\in E_2}V(\lambda)] \isom W \oplus  [\oplus_{\lambda\in E_2}V(\lambda)] \]
Identifying $W$ with $\oplus_{\lambda \in E_1}V(\lambda) \inn V$ we see that $W$ corresponds to a closed point of $\tHilb^G_{\wm}(V)$. As soon as $E_{2} \neq \emptyset$, $W \inn V$ is a degenerate subvariety of $V$, that is, it corresponds to a closed point of $\tHilb^G_{\wm}(V) \setminus \tHilb^G_{E^*}$. 

\item \label{itemnondegW} 
The subvariety of $V$ corresponding to the closed point $(W,\pi)$ of $ \tM_{\wm} = \tHilb_{E^*}^G \inn \tHilb^G_{\wm}(V)$ can be described as follows.  Let $\Mor^G(W,V(\lambda))$ be the set of $G$-equivariant morphisms of algebraic varieties $W \to V(\lambda)$.  We consider $\Mor^G(W,V(\lambda))$ with vector space structure induced from the one of $V(\lambda)$. 
Note that, by Schur's lemma and because $W$ is spherical, 
\[\Mor^{G}(W,V(\lambda)) \isom (\CC[W] \otimes_{\CC} V(\lambda))^{G} \isom (V(\lambda^*) \otimes V(\lambda))^{G}\]
is one-dimensional for every $\lambda \in {E_2}$.  After choosing, for every $\lambda \in E_2$,  a  nonzero $f_{\lambda} \in \Mor^G(W,V(\lambda))$, we can define the following $G$-equivariant closed
embedding of $W$ into $V$: 
\[\varphi\colon W \to V,\  w \mapsto w + (\oplus_{\lambda \in E_2} f_{\lambda}(w)).\] Its image corresponds to a closed point of $\tHilb^G_{E^*}$. 
An appropriate choice of the functions $f_{\lambda}$ (which depends on the identification $\tM_{\wm} = \tHilb^G_{E^*}$) yields the closed point of $\tHilb^G_{E^*}$ corresponding to $(W,\pi)$.
\end{enumerate}
\end{remark}

The next proposition, taken from \cite[Theorem 2.7]{alexeev&brion-modaff}, means we can verify the smoothness of $\tM_{\wm}$ at just one of its points. It also implies that $\tM_{\wm}$ is connected. 
\begin{prop} \label{prop:X0uniqueclosed}
The affine scheme $\tM_{\wm}$ has a unique $\Tad$-fixed point $X_0$, which is also its 
only closed orbit. 
\end{prop}

\begin{cor} \label{cor:smoothX0}
$\tM_{\wm}$ is smooth if and only if it is smooth at $X_0$. 
\end{cor}
\begin{proof}
Denote by $\tM_{\wm}^{sm}$ the smooth locus of $\tM_{\wm}$. Assume  $Z$ is a $\Tad$-orbit inside $\tM_{\wm}$. Then the closure $\overline{Z}$ of $Z$ in 
$\tM_{\wm}$  contains a closed orbit (see, e.g. \cite[Proposition 21.4.5]{tauvel-yu-laag}) which has to be $\{X_0\}$. Hence
the intersection of $\tM_{\wm}^{sm}$ with $Z$ is not empty, because  $\tM_{\wm}^{sm}$ is open and, by assumption, $X_0 \in  \tM_{\wm}^{sm}$. Since local rings of $\tM_{\wm}$ are
isomorphic for points on the same orbit, we get  $Z \subset \tM_{\wm}^{sm}$. This
proves that $\tM_{\wm}$ is smooth.
\end{proof}

Under the identification of $\tM_{\wm}$ with $\tHilb^G_{E^*}$
the distinguished point $X_0$ of $\tM_{\wm}$ corresponds to a certain subvariety of $V$, which we also denote $X_0$ (see \cite[p.99]{alexeev&brion-modaff}). It is the closure of the $G$-orbit in $V$ of
\begin{equation*}
x_0 := \sum_{\lambda \in E} v_{\lambda} \in \oplus_{\lambda \in E}V(\lambda) = V. 
\end{equation*}
Indeed this orbit closure has the right weight monoid by \cite[Theorem 6]{vin&pop-quasi} and is fixed under the action of $\GL(V)^G$.
Yet another result of Alexeev and Brion's gives us an a priori lower bound on the dimension of the moduli schemes we are considering. We first recall a result of F.~Knop. Suppose $X$ is an affine $G$-variety. Let $\widetilde{\Sigma}_X$ be the saturated monoid generated by $\Sigma_X$, that is
\[\widetilde{\Sigma}_X:= \QQ_{\ge 0}\Sigma_X \cap \<\Sigma_X\>_{\ZZ} \inn X(T) \otimes_{\ZZ} \QQ.\]
 Then by \cite[Theorem 1.3]{knop-auto} the monoid $\widetilde{\Sigma}_X$ is free. 
 \begin{prop}[Cor 2.9, Prop 2.13 and Cor 2.14 in \cite{alexeev&brion-modaff}] \label{prop:apriori}
Suppose $X$ is a spherical affine $G$-variety. We view $X$ as a closed point of $\tM_{\Lambda^+_X}$. 
\begin{enumerate}
\item The weight monoid of the closure of the $\Tad$-orbit
  of $X$ in $\tM_{\Lambda^+_X}$ is $\Sigma_X$. Consequently,
  using that the dimension of a (not necessarily normal)
  toric variety
  is equal to the rank of the group spanned by its weight
  monoid, we obtain that $\dim \tM_{\Lambda^+_X} \ge d_X$.\label{clossigmax}
\item The normalization of the $\Tad$-orbit closure of $X$ in $\tM_{\Lambda^+_X}$ has weight monoid
$\widetilde{\Sigma}_X$. Consequently, it is $\Tad$-equivariantly
 isomorphic to a
 multiplicity-free $\Tad$-module of dimension $d_X$,
 since a toric variety is a representation when its weight monoid is
 free.
\item Suppose $X$ is a smooth variety. Then its $\Tad$-orbit
 is open in
 $\tM_{\Lambda^+_X}$. Consequently,
 combining (\ref{clossigmax}) with Proposition~\ref{prop:X0uniqueclosed}, we
 get that
 $\tM_{\Lambda^+_X}$ is smooth if and only if $ \dim T_{X_0} \tM_{\Lambda^+_X} \le d_X$.
\end{enumerate}
\end{prop}
Applying this proposition to our situation we immediately obtain the following corollary. It reduces the proof of Theorem~\ref{thm:main} to Corollary~\ref{cor:mainthm} and Theorem~\ref{thm:cbc}.
\begin{cor} \label{cor:apriori}
Let $W$ be a spherical $G$-module and let $\wm$ be its weight monoid. Then the following are equivalent
\begin{enumerate}
\item $\tM_{\wm}$ is smooth; \label{apriorio}
\item $\dim T_{X_0}\tM_{\wm} = d_W$;
\item $\dim T_{X_0}\tM_{\wm} \le d_W$ \label{apriorith}
\end{enumerate}
Moreover, if $\tM_{\wm}$ is smooth then $\Sigma_W = \widetilde{\Sigma}_W$ and $\tM_{\wm}$ is $\Tad$-equivariantly isomorphic to the multiplicity-free $\Tad$-module with $\Tad$-weight set $-\Psi_W$, where $\Psi_W$ is the (unique) basis of $\Sigma_W$. 
\end{cor}

The following formula for  $d_W$, which is an immediate consequence of \cite[Lemme 5.3]{camus}, will be of use. For the convenience of the reader, we provide a proof suggested by the referee.
\begin{lemma} \label{lem:camusbound}
If $W$ is a spherical $G$-module, then $d_W=a-b$, where $a$ is the rank of the (free) 
abelian group  $\Lambda_W$ and 
$b$ is the number of summands in the decomposition of $W$ into simple $G$-modules.
\end{lemma}
\begin{proof}
Let $G/H$ be the open orbit in $W$. From \cite[Th\'eor\`eme 4.3]{brion-spheriques}, we have that
$d_W = \rk \Lambda_W - \dim N_G(H)/H$. Since $N_G(H)/H$ is isomorphic to the group of $G$-equivariant automorphisms $\Aut^G(G/H)$ of $G/H$, we obtain by \cite[Lemma 3.1.2]{losev-uniqueness} that $N_G(H)/H \isom \Aut^G(W)$. Moreover, $\Aut^G(W)=\GL(W)^G$, because a $G$-automorphism of the multiplicity-free $G$-algebra $\CC[W]$ preserves all irreducible submodules of $\CC[W]$ and therefore sends $W^*$ to $W^*$. Since $\dim \GL(W)^G = b$, it follows that  $\dim N_G(H)/H = b$, which finishes the proof.
\end{proof}

\begin{remark} \label{rem:lwg}
In~\cite{knop-rmks} Knop computed the simple reflections of the so-called `little Weyl group' of $W^*$, whenever $W$ is a saturated indecomposable spherical module. This entry in Knop's List is equivalent to giving the basis of the free monoid $\widetilde{\Sigma}_{W^*}=-w_0\widetilde{\Sigma}_W$: that basis is the set of simple roots of a certain root system of which the `little Weyl group' is the Weyl group (see \cite[Section 1]{knop-auto}, \cite[Section 3]{losev-knopconj} or \cite[Appendix A]{bravi-cirmclass} for details). Knop's List also contains the basis
of $\Lambda^+_{W^*} = -w_0\Lambda^+_{W}$ for the same modules $W$. Those were computed in
\cite{howe-umeda} and \cite{leahy}.  
\end{remark}

Here now is a proposition which provides a concrete description of the tangent space 
$T_{X_0}\tM_{\wm}$.
\begin{prop}[\cite{alexeev&brion-modaff}, Proposition 1.13] \label{prop:concrtan}
Let $V$ be a finite dimensional $G$-module and suppose $X$ is a multiplicity-free closed $G$-subvariety of $V$. Also writing $X$ for the corresponding closed point in $\tHilb_{\Lambda^+_X}^G(V)$, we have that the Zariski tangent space $T_X\tHilb_{\Lambda^+_X}^G(V)$ is canonically isomorphic to 
$H^0(X, \shN_X)^G$, where $\shN_X$ is the normal sheaf of $X$ in $V$.
\end{prop}

\subsection{$\Vgg$ as a first estimate of  $T_{X_0}\tM_{\wm}$}\label{subsec:Tapprox}
In this section we describe a natural inclusion of $T_{X_0}\tM_{\wm}$ into $\Vgg$, see Corollary~\ref{cor:injtan}. 
For calculational purposes we introduce a second $\Tad$-action on $\tHilb^G_{\wm}(V)$ denoted $\widehat{\psi}$, 
which is a twist of the action $\gamma$ defined in Section~\ref{subsec:emb}, and also the infinitesimal version of $\widehat{\psi}$ on $\Vgg$ 
denoted $\alpha$. The action $\alpha$ is the one used throughout \cite{bravi&cupit}. ÊThe main ideas of this section come Ê
from the proof of \cite[Proposition 1.15]{alexeev&brion-modaff}. ÊWe continue to use the notation of Section~\ref{subsec:emb}.

Because $G\cdot x_0$ is dense in $X_0$, we have an injective restriction map  
\[H^0(X_0, \shN_{X_0}) \into H^0(G\cdot x_0, \shN_{X_0})=H^0(G\cdot x_0, \shN_{G\cdot x_0}),\] 
where $\shN_{G\cdot x_0}$ is defined as the restriction of $\shN_{X_0}$ to the open subset $G\cdot x_0\inn X_0$.
This map is $G\times \GL(V)^G$-equivariant because $X_0$ and $G \cdot x_0$ are stable under the natural action of $G \times \GL(V)^G$ on $V$. 
Restricting to $G$-invariants we obtain a $\GL(V)^G$-equivariant inclusion
\begin{equation} \label{eq:injtan}
H^0(X_0, \shN_{X_0})^G \into H^0(G\cdot x_0, \shN_{X_0})^G=H^0(G\cdot x_0, \shN_{G\cdot x_0})^G\end{equation}
Since $G\cdot x_0$ is homogeneous, $\shN_{G\cdot x_0}$ is the $G$-linearized sheaf
on $G/G_{x_0}$ associated with the $G_{x_0}$-module $V/{\fg\cdot x_0}$, that is, the vector bundle associated to $\shN_{G\cdot x_0}$ is $G$-equivariantly isomorphic to $G\times_{G_{x_0}}(\Vg)$. In particular,  
we have a canonical isomorphism 
\begin{equation}
H^0(G\cdot x_0, \shN_{G\cdot x_0})^G \to (V/{\fg\cdot x_0})^{G_{x_0}},\ s \mapsto s(x_0) \label{eq:Nx0}
\end{equation}
which is the precise way of saying that $G$-invariant global sections of $\shN_{G\cdot x_0}$ are determined by their value at $x_0$.

The $T$-action $\phi$ on $V$ defined in ÊSection~\ref{subsec:emb} induces an action 
on Ê$H^0(G\cdot x_0, \shN_{G\cdot x_0})^G$ and we could use the isomorphism (\ref{eq:Nx0}) to 
induce an action on $\Vgg$. ÊBecause it is better suited to our calculations, we prefer to 
work with a slightly different action. Recall that $\phi$ was obtained by composing the 
natural action of $\GL(V)^G$ Êwith the Êhomomorphism $h$ of (\ref{eq:TGLVGAB}).
Instead, we Êobtain a $T$-action, denoted $\psi$, on $V$ by composing the Êaction 
of Ê$\GL(V)^G$ with Êthe homomorphism
\begin{equation}
f\colon T \to \GL(V)^G,\ t \mapsto (\lambda(t))_{\lambda \in E}.  \label{eq:TGLVG}
\end{equation} 
In other words, $\psi$ is the following action:
\[\psi: T \times V \to V,\ \psi(t,v) = f(t) \cdot v.\] 
\begin{remark} \label{rem:fstinj}
Since the $T$-weights in $E$ are linearly independent, $f$ is surjective. 
\end{remark}

Since $\psi$ commutes with the action of $G$ on $V$, it induces an action of $T$ on $\tHilb_{\wm}^G(V)$ and on $H^0(G\cdot x_0, \shN_{G\cdot x_0})^G$. By slight abuse of notation we call both of these actions $\widehat{\psi}$.\label{widehatpsi} 
Using the isomorphism of equation (\ref{eq:Nx0}) we now translate this action into an action of $T$ on  $(V/{\fg\cdot x_0})^{G_{x_0}}$. The relationship (via $f$) between the action of $T$ on $\Vgg$ and the action of $\GL(V)^G$ on  $H^0(G\cdot x_0, \shN_{G\cdot x_0})^G \isom \Vgg$ will play a part in the proof of Proposition~\ref{prop:excl}.  Let $\rho: T\times V \to V$ be the action of $T$ on $V$ induced by restriction of the action of $G$.
\begin{defn} \label{def:alpha}
We denote $\alpha$ the action of $T$ on $V$ given by
\[\alpha(t,v) := \psi(t, \rho(t^{-1},v)) \quad \text{for $t \in T$ and $v\in V$}.\]
\end{defn}
\begin{remark}
One immediately checks that for all $\lambda \in E$ and every $v \in V(\lambda) \inn V$, 
\begin{equation}
 \alpha(t,v) = \lambda(t)t^{-1} v. \label{eq:defTadact}
 \end{equation} 
\end{remark}
\begin{prop}\label{prop:Teqv} 
The action $\alpha$ induces an action of $T$ on $(V/{\fg\cdot x_0})^{G_{x_0}}$, which we also call $\alpha$. For $H^0(G\cdot x_0, \shN_{X_0})$ equipped with the action $\widehat{\psi}$ and $\Vgg$ with the action $\alpha$, the isomorphism (\ref{eq:Nx0}) is $T$-equivariant.  
Moreover, both actions $\alpha$ and $\widehat{\psi}$ have $Z(G)$ in their kernel, whence the isomorphism (\ref{eq:Nx0}) is $\Tad$-equivariant. 
\end{prop} 
\begin{proof}
 In this proof, we will write $\shN$ for $\shN_{X_0}$. 
Suppose $t\in T$ and $s \in 
H^0(G\cdot x_0, \shN_{G\cdot x_0})^G$. Then
\[\widehat{\psi}(t, s)(x_0) = (f(t)\cdot s)(x_0) = f(t) \cdot s(f(t)^{-1}\cdot x_0) = f(t) \cdot s(\psi(t^{-1}, x_0)).\]
Now note that $\psi(t^{-1}, x_0) = \rho(t^{-1},x_0)$ by the definitions of $f$ and $x_0$. In other words, we have that 
\(\widehat{\psi}(t, s)(x_0) = f(t) \cdot s(\rho(t^{-1}, x_0)).\)
Let $v$ be an element of  $V$ such that $s(x_0) = [v] \in \shN|_{x_0} = V/\fg\cdot x_0$. 
Then \[s(\psi(t^{-1}, x_0)) = [\rho(t^{-1}, v)] \in \shN|_{\rho(t^{-1}, x_0)} \]
because $s$ is $G$-invariant and therefore $T$-invariant. 
It follows that 
\begin{equation}
f(t) \cdot s(\rho(t^{-1}, x_0)) = [\psi(t,\rho(t^{-1},v))] = [\alpha(t,v)]\in  V/\fg\cdot x_0. \label{eq:psitoalpha}
\end{equation}
A straightforward verification (or equation (\ref{eq:psitoalpha})) shows that
$\alpha$ induces a well-defined action on $\Vgg$. From (\ref{eq:psitoalpha}) we can conclude that  the isomorphism (\ref{eq:Nx0}) is $T$-equivariant. Finally, that $Z(G)$ belongs to the kernel of $\alpha$ is an immediate consequence of highest weight theory. 
\end{proof}
{\bf From now on}, the \emph{$\Tad$-action on $V$} (and on $\Vg$, $\Vgg$, etc.) will refer to the action given by Ê$\alpha$, and the
\emph{$\Tad$-action on $\tHilb^G_{\wm}(V)$} (and on $\tM_{\wm}$) Êwill refer to the action given by Ê$\widehat{\psi}$.
Combining Proposition~\ref{prop:concrtan} and equations (\ref{eq:injtan}) and (\ref{eq:Nx0}) we obtain a natural injection $T_{X_0}\tM_{\wm} \into \Vgg$. 
\begin{cor} \label{cor:injtan}
The natural injection
$T_{X_0}\tM_{\wm}  \into (V/\fg\cdot x_0)^{G_{x_0}}$ just defined is $\Tad$-equivariant, where we consider $T_{X_0}\tM_{\wm}$ as a $\Tad$-module via $\widehat{\psi}$ and $\Vgg$ via $\alpha$.
\end{cor}
\begin{remark} \label{rem:sigmawstar}
Let $(G,W)$ be a spherical module as in Theorem~\ref{thm:cbc} and let $\wm$ be its weight monoid. 
The $\Tad$-weight set we obtain in Section~\ref{sec:cases} for $T_{X_0}\tM^G_{\wm}$ using the action $\alpha$ is the basis of the free monoid $\widetilde{\Sigma}_{W^*}=-w_0\widetilde{\Sigma}_W$ (instead of $-\widetilde{\Sigma}_W$ as in Theorem~\ref{thm:main} where the action $\gamma$ was used).
\end{remark}

\begin{remark} \label{rem:abinjiso}
Thanks to \cite[Proposition 1.15(iii)]{alexeev&brion-modaff} and Lemma~\ref{lem:X0normal} below, 
we know that the injection in Corollary~\ref{cor:injtan} is an isomorphism when $X_0 \setminus G\cdot x_0$ 
has codimension at least $2$ in $X_0$. This condition is often not met in our situation. Even when it is not, the injection
is often an isomorphism, but Êwe also have a number of cases where the injection is not surjective; see, 
for example, Remark~\ref{rem:propcase17}.
\end{remark}

\subsection{Auxiliary lemmas on $\Vgg$ and the $\Tad$-action} \label{subsec:auxiliary}
We continue to use the notation  of Sections \ref{subsec:emb} and \ref{subsec:Tapprox}.
Let $G\rtimes \Tad$ be the semidirect product of $G$ and $\Tad$, where $\Tad$ acts on $G$ as follows:
\begin{equation}
\Tad \times G \to G, (t,g) \mapsto t^{-1}gt. \label{eq:TactonG}
\end{equation} 
As explained in \cite[p.102]{alexeev&brion-modaff}, the linear actions of $\Tad$ and $G$ on $V$ can be extended together to a linear action of $G\rtimes \Tad$ on $V$ as follows. 
 Suppose $(g,t) \in G\rtimes \Tad$ and $v\in V$, then
\begin{equation}
(g,t) \cdot v := g\cdot\alpha(t,v) = \alpha(t,(tgt^{-1})\cdot v),   \label{eq:GrtimesTad}
\end{equation} 
where $\alpha$ is the $\Tad$-action. Since $\Tad$ fixes $x_0$, we have that $(G\rtimes \Tad)_{x_0} = G_{x_0} \rtimes \Tad$ and
$(G\rtimes \Tad) \cdot x_0 = G\cdot x_0$. It follows that $(G\rtimes \Tad)_{x_0}$ acts on $\fg \cdot x_0 = T_{x_0}(G\cdot x_0)$ and
we have an exact sequence of $(G_{x_0} \rtimes \Tad)$-modules
\begin{equation} \label{eq:fundseq}
0 \longrightarrow \fg\cdot x_0 \longrightarrow V \longrightarrow V/\fg\cdot x_0 \longrightarrow 0.
\end{equation}

The next lemma gathers some elementary facts about $G_{x_0}$ and $\fg \cdot x_0$. They will  be of use in Sections \ref{sec:Brionstrat} and \ref{sec:cases}. 
\begin{lemma}\label{lem:isotropygroup}
Let $E$ be a finite subset 
 of $\Lambda^+$, and define $V$ and $x_0$ as before, that is,
\(x_0:=\sum_{\lambda\in E}v_{\lambda} \in V := \oplus_{\lambda\in E}V(\lambda).\)
Then the following hold:
\begin{enumerate}
\item $G_{x_0} = T_{x_0}.G_{x_0}^{\circ} $, where $G_{x_0}^{\circ}$ is the connected component of $G_{x_0}$ containing the identity; \label{iso1}
\item $T_{x_0} = \cap_{\lambda \in E} \ker \lambda$; \label{iso2}
\item $\fg_{x_0} = \fu \oplus \ft_{x_0} \oplus \bigoplus_{\alpha \in E^{\perp}} \fg_{-\alpha}$, where 
$E^{\perp}:=\{\alpha \in R^+\ |\ \<\lambda,\alpha^{\vee}\> = 0 \text{ for all } \lambda \in E\}$; \label{iso3}
\item The $\Tad$-weight set of $\fg\cdot x_0$ is $(R^+ \setminus  E^{\perp})\cup \{0\}$. 
\label{weightsgx0}
\end{enumerate}
\end{lemma}
\begin{proof}
The proof of (\ref{iso1})  just requires replacing $v_{\lambda}$ by $x_0$ in the proof of \cite[Lemme1.7]{jansou-deformations}. (\ref{iso2}) is immediate. (\ref{iso3})  follows form the well-known properties of the action of root operators on highest weight vectors. For (\ref{weightsgx0})
just note that $\fg\cdot x_0 = \fb^{-}\cdot x_0$, where $\fb^-$ is the Lie algebra of the Borel subgroup $B^-$ opposite to $B$ with respect to $T$. 
\end{proof}

In addition to the facts listed in Lemma~\ref{lem:isotropygroup}, the following will be useful too in Section~\ref{sec:cases}. Recall our convention that $G'_{x_0} := (G')_{x_0}$ and $\fg'_{x_0}:=(\fg')_{x_0}$. Recall also that if $\fk$ is a Lie-subalgebra of $\fg_{x_0}$, then $(V/\fg\cdot{x_0})^{\fk} = \{[v] \in V/\fg\cdot{x_0}\ |\ Xv \in \fg\cdot{x_0} \text{ for all } X \in \fk\}$, by definition.

\begin{lemma} \label{lem:Tadweights}
Using the notations of this section, the following hold:
\begin{enumerate}[(a)]
\item  The inclusions $\Vgg \inn  \Vggp \inn (V/\fg\cdot{x_0})^{\fg'_{x_0}}$ are inclusions of $T_{\ad}$-modules; \label{inclVgs}
\item Let $H$ be a subgroup of $G$  and let $T_H$ be a subtorus of $T \cap H$. Let $\Gamma$ be the subgroup of $X(T_H)$ generated by the image of $E$ under the restriction map $p\colon X(T) \onto X(T_H)$.  Suppose $v \in V$ is a $\Tad$-eigenvector of weight $\beta$ so that 
$[v]$ is a nonzero element of $(\Vg)^{H_{x_0}}$. Then $p(\beta)$ belongs to $\Gamma$;
\label{weightinvar}
\item If $\fh$ is a Lie-subalgebra of $\fg$ containing $\fg'$, then $(V/\fg{x_0})^{G_{x_0}} = (V/\fg{x_0})^{\fh_{x_0}}_{\<E\>}$, where $(V/\fg{x_0})^{\fh_{x_0}}_{\<E\>}$ is the subspace
of $(V/\fg{x_0})^{\fh_{x_0}}$ spanned by $$\{[v] \in (V/\fg{x_0})^{\fh_{x_0}}\ |\ v \text{ is a $\Tad$-eigenvector with
    weight in } \<E\>_{\ZZ}\}.$$ \label{liealgandtorinvs}
\end{enumerate} 
\end{lemma}
\begin{proof}
For assertion (\ref{inclVgs}) we first note that the subgroups $G'_{x_0}$ and 
$(G'_{x_0})^{\circ}$ of $G$ are stable under the action 
of $\Tad$ on $G$ in (\ref{eq:TactonG}), so that the $(G_{x_0}\rtimes \Tad)$-action on $\Vg$ restricts to 
$G'_{x_0} \rtimes \Tad$ and $(G'_{x_0})^{\circ}\rtimes \Tad$. The assertion now follows since
$\Lie(G'_{x_0}) = \fg'_{x_0}$. 
We now prove (\ref{weightinvar}). Let $\beta$ be the $\Tad$-weight of $v$ and for every $\lambda\in E$, let $x_{\lambda}$ be the projection of $v$ onto $V(\lambda) \inn V$.  Then $v =\sum_{\lambda \in E} x_{\lambda}$. Since $v$ is nonzero, at least one of the $x_{\lambda}$ is nonzero. Choose one.
Then $x_{\lambda}$ is a $T$-eigenvector of weight $\lambda - \beta$. Since $v$ is fixed by $(T_H)_{x_0}$ it follows that $x_{\lambda}$ is and so $(\lambda - \beta)|_{(T_H)_{x_0}} = 0$. Since $(T_H)_{x_0} = \cap_{\lambda \in E}\ker p(\lambda)$ this implies that $p(\lambda- \beta)$ and therefore $p(\beta)$ lie in $\Gamma$. Assertion (\ref{liealgandtorinvs}), finally, is a consequence of parts (\ref{iso1}) and (\ref{iso2}) of Lemma~\ref{lem:isotropygroup}.
\end{proof}

\begin{lemma} \label{lem:tangentcrit}
  We use the notations of this section. Let $v \in V$ be a $\Tad$-eigenvector.
  If $[v]$ is a nonzero element of $(V/\fg{x_0})^{\fg'_{x_0}}$, then the following two statements hold.
\begin{enumerate}[(A)]
\item For every positive root $\alpha$ one of the following situations occurs \label{AAA}
\begin{enumerate}[(1)]
\item $X_{\alpha}v =0$;
\item $X_{\alpha}v$ is a $T_{\ad}$-eigenvector of weight $0$;
\item $X_{\alpha}v$ is a $T_{\ad}$-eigenvector with weight in $R^+ \setminus E^{\perp}$;
\end{enumerate}
\item  There is at least one simple root $\alpha$ such that $X_{\alpha}v \neq 0$. \label{BBB}
\end{enumerate}  
\end{lemma}
\begin{proof}
  Part (\ref{AAA}) follows from the fact that $\fu \inn \fg'_{x_0}$ and part (\ref{weightsgx0}) of
  Lemma~\ref{lem:isotropygroup}. For (\ref{BBB}) first note that the linear
  independence of $E$ implies that the subspace $\ft \cdot x_0$ of $\fg\cdot x_0$ contains all the highest weight vectors of $V$.
Therefore $[v]\neq 0$ implies that $v$ is not a sum of highest weight vectors.
\end{proof}

\begin{lemma} \label{lem:gx0}
Let $(\overline{G},W)$ be a spherical $\overline{G}$-module and let $G$ be a reductive subgroup of $\overline{G}$ containing $\overline{G}'$ and such that $(G,W)$ is spherical. Then $\fg\cdot x_0 = \overline{\fg}\cdot x_0$.   
\end{lemma}
\begin{proof}
We have that  $\fg\cdot x_0 = \ft \cdot x_0 + \fg' \cdot x_0$. By hypothesis, $\fg' = \overline{\fg}'$. Finally 
\(\ft \cdot x_0 = \<v_{\lambda}\colon \lambda \in E\>_{\CC} = \overline{\ft} \cdot x_0\) 
because the elements of $E$ are linearly independent (for both $G$ and $\overline{G}$).
\end{proof}

\subsection{Further results and notions from~\cite{bravi&cupit}}  \label{subsec:bcf}
We continue to use the notation of Sections \ref{subsec:emb} and \ref{subsec:Tapprox}.
In this section we recall results from \cite{bravi&cupit} about $\tM_{\wm}$ and $T_{X_0}\tM_{\wm}$ under the condition that $\wm$ is $G$-saturated (see Definition~\ref{def:Psat}), and we mention some immediate consequences. 

The following condition on submonoids of $\Lambda^+$ was considered by D.~Panyushev in \cite{panyushev-deform}. It also occurs in \cite{vin&pop-quasi}. We will use the terminology of \cite[Section 4.5]{brion-ihs-arxivv2}. 
\begin{defn}  \label{def:Psat}
A submonoid of $\wm$ of $\Lambda^+$ is called \emph{$G$-saturated} if 
\[\<\wm\>_{\ZZ} \cap \Lambda^+ = \wm. \]
\end{defn}

\begin{remark} \label{rem:Psatinjtan}
As explained in \cite[Section 3]{bravi&cupit} the injection 
\[T_{X_0}\tM_{\wm} \into (V/\fg\cdot x_0)^{G_{x_0}} \]
of Corollary~\ref{cor:injtan} is an isomorphism when $\wm$ is $G$-saturated. The reason is that, by  Theorem 9 of \cite{vin&pop-quasi}, $X_0 \setminus G\cdot x_0$ then has codimension at least $2$ in the normal variety $X_0$; see Remark~\ref{rem:abinjiso}. 
\end{remark}

\begin{remark}
Clearly, a  submonoid  $\wm \inn \Lambda^+$ is $G$-saturated if and only if $-w_0(\wm)$ is. This fact will be used in Section~\ref{sec:cases}, because if $\wm$ is the weight monoid of a spherical module $(G,W)$, then $-w_0(\wm)$ is the weight monoid of the dual module $(G,W^*)$. 
\end{remark}

\begin{lemma}[Lemma 2.1 in \cite{bravi&cupit}] \label{lem:bcfcritpsat}
Let $\lambda_1, \ldots, \lambda_k$ be linearly independent dominant weights. The following are equivalent:
\begin{enumerate}[(a)]
\item 
$\wm = \<\lambda_1, \ldots, \lambda_k\>_{\NN}$ is $G$-saturated;
\item  there exist $k$ simple roots $\alpha_{t_1},..,\alpha_{t_k}$ such that  $\< \lambda_i , \alpha_{t_j}^{\vee} \> \neq 0$ if and only if $i = j$.
\end{enumerate}
\end{lemma}

\begin{theorem}[Theorem 2.2 and Corollary 2.4 in \cite{bravi&cupit}] \label{thm:bcf}
Suppose $G$ is a semisimple group and $\wm$ is a $G$-saturated and freely generated submonoid of $\Lambda^+$.  
Then 
\begin{enumerate}
\item the tangent space $T_{X_0}\tM^G_{\wm} \isom \Vgg$ is a multiplicity-free $\Tad$-module whose $\Tad$-weights belong to Table 1 of \cite[p. 2810]{bravi&cupit};
\item  the moduli scheme $\tM^G_{\wm}$ is isomorphic as a $\Tad$-scheme to $\Vgg$.
\end{enumerate}
\end{theorem}

\begin{remark}\label{rem:srA}
When $G$ is of type $\ssA$, the $\Tad$-weights which can occur in the space $\Vgg$ of Theorem~\ref{thm:bcf} are (see \cite[Table 1 p. 2810]{bravi&cupit}):
\begin{itemize}
\item[(\SRo)] $\alpha + \alpha'$ with $\alpha, \alpha' \in \Pi$ and $\alpha \perp \alpha'$; 
\item[(\SRt)] $2\alpha$ with $\alpha \in \Pi$;
\item[(\SRth)] $\alpha_{i+1} +\alpha_{i+2} + \ldots + \alpha_{i+r}$ with $r \ge 2$ and $\alpha_i, \alpha_{i+1}, \ldots, \alpha_{i+r}$ simple roots that correspond to consecutive vertices in a connected component of the Dynkin diagram of $G$;
\item[(\SRf)] $\alpha_{i} + 2\alpha_{i+1} + \alpha_{i+2}$ with $\alpha_i, \alpha_{i+1}, \alpha_{i+2}$ simple roots that correspond to consecutive vertices in a connected component of the Dynkin diagram of $G$.
\end{itemize}
\end{remark}

For several cases in Knop's List, Theorem~\ref{thm:cbc} is a consequence of
Bravi and Cupit-Foutou's result 
mentioned above, thanks to Corollary~\ref{cor:bcfred} below. We first establish a lemma needed in the proof of Corollary~\ref{cor:bcfred} and of Proposition~\ref{prop:cbc}. 

\begin{lemma} \label{lem:psigmax}
Suppose $X$ is  an affine $G$-variety and let $H$ be a connected subgroup of $G$ containing $G'$ ($H$ is reductive by Lemma~\ref{lem:redsgr} below). Let $B_H$ be the Borel subgroup $B\cap H$ of $H$ and let $p: X(B) \onto X(B_H)$ be the restriction map. Let $\Sigma_X$ be the root monoid of the $G$-variety $X$ and let $\Sigma'_X$ be the root monoid of $X$ considered as an $H$-variety (where $H$ acts as a subgroup of $G$). Assume that the restriction of $p$ to $\Lambda_{(G,X)} \inn X(B)$ is injective. Then \[\Sigma'_X = p(\Sigma_X).\]
Consequently, the invariant $d_X$ is the same for $(G,X)$ as for $(H,X)$.  
\end{lemma}
\begin{proof}
By Lemma~\ref{lem:spherrestr} below, $p(\Lambda^+_{(G,X)}) = \Lambda^+_{(H,X)}$. 
Put $R = \CC[X]$ and let $R = \oplus_{\lambda \in \Lambda^+_{(G,X)}} R_{(\lambda)}$ be its decomposition into isotypical components as a $G$-module. Then, because $p|_{\Lambda^+_{(G,X)}}$ is injective and $G' \inn H$, we have that for every $\lambda \in \Lambda^+_{(G,X)}$,  $R_{(\lambda)} \inn R$ is the $H$-isotypical component of $R$ of type $V(p(\lambda))$. The lemma now follows from the definitions of $\Sigma_X$ and $d_X$. 
\end{proof}

\begin{cor} \label{cor:bcfred}
Let $G$ be a connected reductive group and let $X$ be a smooth affine spherical $G$-variety with weight monoid $\wm$. Suppose $X$ is spherical for the restriction of the $G$-action to $G'$. Put $T' = T\cap G'$.  Let $\wm'$ be the image\footnote{By Lemma~\ref{lem:spherrestr} below, $\wm'$ is the weight monoid of the $G'$-variety $X$.}  of $\wm$ under the restriction map $p:X(T) \onto X(T')$.  

If $\wm'$ is  freely generated then so is $\wm$.  Suppose $\wm'$ is freely generated and $G'$-saturated. Then $\dim \Vggp = d_X$ and, consequently, $\dim T_{X_0} \tM^G_{\wm} = d_X$. 
\end{cor}

\begin{proof}
The fact that $X$ is spherical for $G'$ implies  that the restriction of $p$ to $\wm$ is injective (see Lemma~\ref{lem:spherrestr} below). This proves that $\wm$ is freely generated when $\wm'$ is. 

We now assume that $\wm'$ is freely generated and $G'$-saturated. 
First note that  
\begin{equation}V \isom  \oplus_{\lambda \in E} V(p(\lambda)) \label{eq:Vforgpr}\end{equation}
as a $G'$-module and that 
\begin{equation}
\fg \cdot x_0 = \ft\cdot x_0 + \fu^- \cdot x_0 = \ft'\cdot x_0 + \fu^- \cdot x_0 = \fg' \cdot x_0.
\label{eq:gx0andgprx0}
\end{equation} 
where $\fu^-$ is the sum of the negative root spaces of $\fg'$. 
Here the second equality follows from the fact that because the sets $E \inn X(T)$ and
$p(E) \inn X(T')$ are linearly independent, 
\[\ft \cdot x_0 = \<v_{\lambda} \colon \lambda \in E\>_{\CC} = \ft'\cdot x_0. \]

Now consider $X$ as a closed point of $\tM_{\wm'}^{G'}$. By Theorem~\ref{thm:bcf}, $\tM_{\wm'}^{G'}$ is smooth, and so Proposition~\ref{prop:apriori} (with Lemma~\ref{lem:psigmax}) tells us that $\dim T_{X_0}\tM_{\wm'}^{G'} =d_X$. Since $T_{X_0}\tM_{\wm'}^{G'} \isom (V/\fg'\cdot x_0)^{G'_{x_0}}$ (using (\ref{eq:Vforgpr})) and, since from (\ref{eq:gx0andgprx0}) we have that $(V/\fg\cdot x_0)=(V/\fg'\cdot x_0)$ and therefore that $\Vggp = (V/\fg'\cdot x_0)^{G'_{x_0}}$, it follows that $\dim \Vggp = d_X$. By Corollary~\ref{cor:injtan}, $T_{X_0}\tM^G_{\wm} \inn \Vgg \inn \Vggp$, and Proposition~\ref{prop:apriori} now finishes the proof.
\end{proof}

\section{Criterion for non-extension of sections} \label{sec:Brionstrat}
We continue to use the notation of Sections \ref{subsec:emb} and \ref{subsec:Tapprox}. In particular, by the $\Tad$-action on $V$ and $\Vgg$ we mean the action $\alpha$ of Definition~\ref{def:alpha}.
The criterion we give here (Proposition~\ref{prop:excl}) for  excluding certain $\Tad$-weight spaces of $\Vgg$ from $T_{X_0}\tM_{\wm}$ was  suggested to us by M.~Brion. 
It consists of sufficient conditions on a section $s \in H^0(G\cdot x_0, \shN_{X_0})^G \isom \Vgg$ for it not to extend to $X_0$. The basic idea is that the conditions guarantee that there is a point $z_0 \in X_0$ (which depends on $s$) whose $G$-orbit has codimension $1$ in $X_0$ and such that $s$ does not extend to $z_0$ along the line joining $x_0$ and $z_0$. 

Before we prove the criterion we recall some facts. 
We begin with the orbit structure of $X_0$. It is known (see \cite[Theorem 8]{vin&pop-quasi}) that the following map describes a one-to-one correspondence between the set of subsets of $E$ and the set of $G$-orbits in $X_0$:
\[ (D \inn E) \mapsto G\cdot v_D \quad  \text{ where } v_D:=\sum_{\lambda \in D}v_{\lambda}.\] 
Recall that $\GL(V)^G \isom \GGm^{|E|}$ and that an element $(t_{\lambda})_{\lambda\in E} \in \GL(V)^G$ acts on $V = \oplus_{\lambda \in E} V(\lambda)$ by scalar multiplication by $t_\lambda \in \GGm$ on the submodule $V(\lambda)$. Given $D \inn E$, define the one-parameter subgroup $\sigma_D$ of $\GL(V)^G$ as follows:
\begin{equation*}
\sigma_D\colon \GGm \to \GL(V)^G, t\mapsto (p_{\lambda}(t))_{\lambda \in E}
\end{equation*}
where  
$p_{\lambda}(t)  = t$ if $\lambda \notin D$ and $p_{\lambda}(t)  = 1$
otherwise. 
Then \(\lim_{t\to 0} \sigma_D(t) \cdot x_0 = v_D.\) 
We also put
\(
z_t:=\sigma_D(t)\cdot x_0\text{ for $t \in \GGm$}\)  and \(z_0:=v_D\)
so that $\lim_{t\to 0} z_t = z_0$. The orbits (of codimension $1$) that will play a part in the criterion correspond to subsets $D = E\setminus\{\lambda\}$ where $\lambda \in E$ is a judiciously chosen element, depending on the section to be excluded.

The following proposition tells us which subsets $D \inn E$ correspond to orbits of codimension $1$ in $X_0$. 
\begin{prop} \label{prop:codim1stab}
Let $E$, $V$ and $x_0$ be as before.  
Suppose $\lambda_0 \in E$.
Put
\(
z_0 = \sum_{\lambda \in E, \lambda \neq \lambda_0} v_{\lambda}.
\)
Then $\dim \ft_{z_0} = \dim \ft_{x_0}+1$.
Consequently, the following are equivalent:
\begin{enumerate}[(a)]
\item $\dim \fg_{z_0} = \dim \fg_{x_0} + 1$; \label{codim1staba}
\item $E^{\perp} = (E\setminus \{\lambda_0\})^{\perp}$ (see Lemma~\ref{lem:isotropygroup}(\ref{iso3})  for the definition of $\perp$);\label{codim1stabb}
\item $E^{\perp} \cap \Pi = (E\setminus \{\lambda_0\})^{\perp} \cap \Pi$.\label{codim1stabc}
\end{enumerate}
\end{prop}
\begin{proof}
The first assertion follows from (the Lie-algebra version of) Lemma~\ref{lem:isotropygroup}(\ref{iso2})
and the fact that $E$ is linearly independent. 
The equivalence of (\ref{codim1staba}) and (\ref{codim1stabb}) is an immediate consequence of  Lemma~\ref{lem:isotropygroup}(\ref{iso3}). For (\ref{codim1stabb}) $\Leftrightarrow$ (\ref{codim1stabc}) we use a standard fact about parabolic subgroups containing $B$. Indeed,
let $\PP(V)$ be the projective space of lines through $0$ in $V$ and $V\setminus\{0\}\to \PP(V), v \mapsto [v]$ the  canonical map. Define the parabolic subgroup $P$ of $G$ by $P:=G_{[x_0]}$.
Then $-E^{\perp}$ is the set of negative roots of $P$. As is well known (see, e.g. \cite[Theorem 30.1]{humphreys-lag}), $-E^{\perp}$ is the set of negative roots of $G$ that are $\ZZ$-linear combinations of the simple roots in $E^{\perp} \cap \Pi$. Consequently, $E^{\perp}$  is completely determined by $E^{\perp} \cap \Pi$. Similarly, $(E\setminus \{\lambda_0\})^{\perp} \cap \Pi$ determines $(E\setminus \{\lambda_0\})^{\perp}$.
\end{proof}

\begin{lemma}\label{lem:X0normal}
The $G$-variety $X_0$ is normal. 
\end{lemma}
\begin{proof}
Because $\wm$ is freely generated, we have that $\< \wm \>_{\ZZ} \cap \QQ_{\ge0}\wm = \wm$ in $\Lambda \otimes_{\ZZ}\QQ$. 
We then apply \cite[Theorem 10]{vin&pop-quasi} or the general fact~\cite[Theorem 6]{popov-contractions} that $X_0$ is normal if and only if $X_0 \quot U$ is a normal $T$-variety  (recall that $X_0 \quot U \isom \Spec\CC[\wm]$).
\end{proof}

\begin{lemma} \label{lem:Tz0}
Suppose $\lambda \in E$ is such that for $D = E \setminus \{\lambda\}$, the $G$-orbit of $z_0 = v_D$ 
has codimension $1$ in $X_0$. 
Then $T_{z_0}X_0 = \fg\cdot z_0 \oplus \CC v_{\lambda}$.
\end{lemma} 
\begin{proof}
By Lemma~\ref{lem:X0normal}, $X_0$ is normal. Therefore its singular locus has codimension at least $2$. Since the singular locus is $G$-stable and $G\cdot z_0$ has codimension $1$, it follows that $X_0$ is smooth at $z_0$.  Therefore, $\dim T_{z_0}X_0 = \dim \fg\cdot z_0 + 1$. Moreover 
$t\mapsto z_t= \sigma_D(t)\cdot x_0$ is an irreducible curve in $X_0$ (because the elements of $E$ are linearly independent) and $z_t = t\cdot v_{\lambda} + z_0$. Thus $\frac{d}{dt}|_{t=0} z_t = v_{\lambda}$ and so $v_{\lambda}\in T_{z_0}X_0$. Further $v_\lambda \notin \fg \cdot z_0$ since $\fg\cdot z_0$ lies in the complement of $V(\lambda) \inn V$. 
\end{proof}

Now let $[v]$ be a $\Tad$-eigenvector in $(V/\fg\cdot x_0)^{G_{x_0}}$. We denote the corresponding  section in $H^0(G\cdot x_0, \shN_{X_0})^G$ by $s$, that is, $s(x_0) = [v]$. Recall from Proposition~\ref{prop:Teqv} that the $\Tad$-action on $\Vgg$ comes from  the action of $T$ on $H^0(G\cdot x_0, \shN_{X_0})^G$ through $f\colon T \to GL(V)^G$, defined in~(\ref{eq:TGLVG}). Since $f$ is surjective (see Remark~\ref{rem:fstinj}), we can also consider $s$ as an eigenvector for $GL(V)^G$. Because it will play a part in what follows, we remark that if the $GL(V)^G$-weight of $s$ is $\delta$, then the $\Tad$-weight of $s(x_0)=[v]$ is $f^{*}(\delta)$.   By definition, we have that for $a\in GL(V)^G$
\begin{equation*}
s^a(x_0) := a\cdot s (a^{-1}\cdot x_0) = \delta(a)s(x_0).
\end{equation*}
This implies that for every $D\inn E$ and $t \in \GGm$,
\begin{equation}
s(z_t) = s (\sigma_D(t) \cdot x_0) = \delta(\sigma_D(t))^{-1}\sigma_D(t)\cdot s(x_0)
= [\delta(\sigma_D(t))^{-1}\sigma_D(t) v] \in V/\fg\cdot z_t. \label{eq:szt}
\end{equation}

We need one final ingredient for the proof of Proposition~\ref{prop:excl}. Recall that any $v \in V$ defines a global section $s_v \in H^0(X_0, \shN_{X_0})$ by
\[s_v(x) = [v] \in V/T_xX_0 \text{ for all }x \in X_0.\]
Here then is the proposition we will use in Sections~\ref{subsec:case16}, \ref{subsec:case17} and \ref{subsec:case18} to prove that certain sections in $H^0(G\cdot x_0, \shN_{X_0})^G$ do not extend to $X_0$. As mentioned at the beginning
of this section, by the $\Tad$-action on $V$ we mean $\alpha$.
\begin{prop} \label{prop:excl}
Suppose $v \in V$ is a $\Tad$-eigenvector of weight $\beta \in \Lambda_R$ such that $[v] \in \Vgg$. Let $s \in H^0(G\cdot x_0, \shN_{X_0})^G$ be defined by $s(x_0) = [v]$. 
If there exists $\lambda \in E$ so that
\begin{itemize}
\item[(\ESo)] the coefficient of $\lambda$ in the unique expression of $\beta \in \<E\>_{\ZZ}$ as a $\ZZ$-linear combination of the elements of $E$ is positive;
\item[(\ESt)] the projection of $v \in V$ onto $V(\lambda) \inn V$ is zero;
\item[(\ESth)] if  $\eta$ is a simple root so that $\<\lambda, \eta^{\vee}\> \neq 0$ then there exists $\widetilde{\lambda} \in E\setminus\{\lambda\}$ so that $\<\widetilde{\lambda}, \eta^{\vee}\> \neq 0$;
\item[(\ESf)] if $\beta \in R_{+}\setminus E^{\perp}$ (see Lemma~\ref{lem:isotropygroup} for the definition of $E^{\perp}$), then there exists  $\xi$ in $E \setminus \{\lambda\}$ so that $\<\xi, \beta^{\vee}\> \neq 0$ and the projection of $v$ onto $V(\xi)$ is zero;
\end{itemize}
then $s$ does not extend to $X_0$. 
\end{prop}
\begin{proof}
The idea of the proof is to `compare' the section $s$ to the section $s_v \in H^0(X_0,\shN_{X_0})$. Put $D= E\setminus\{\lambda\}$. We first show that 
\begin{enumerate}[(i)]
\item there exists a positive integer $k$ so that $s(\sigma_D(t)\cdot x_0) = t^{-k} s_v(\sigma_D(t)\cdot x_0)$ for all $t \in \GGm$;\label{tpext1} 
\item $s_v(z_0) \neq 0$, \label{tpext2}  
\end{enumerate}
where $z_0 = v_D = \lim_{t\to 0} \sigma_D(t)\cdot x_0$. We then show
that (\ref{tpext1}) and (\ref{tpext2}) imply that $\lim_{t\to 0} s(\sigma_D(t)\cdot x_0)$ does not exist, i.e.\ that $s(z_0)$ does not exist.

We first prove (\ref{tpext1}). Let $f: T \to \GL(V)^G$ be the map (\ref{eq:TGLVG}) on page~\pageref{eq:TGLVG}. Since it is surjective, \(f^*\colon X(\GL(V)^G) \to X(T), \delta \mapsto \delta \circ f\) is injective and $\beta \in \im(f^*)$. 
Put $\delta:=(f^*)^{-1}(\beta)$, the $\GL(V)^G$-weight of $s$. From equation (\ref{eq:szt}) we have that \(
s(z_t) =  [\delta(\sigma_D(t))^{-1}\sigma_D(t) v] 
\) for every $t\in \GGm$.
Using (\ESt), $\sigma_D(t)v = v$ for every $t \in \GGm$. Therefore
\begin{equation*}
s(z_t) =  [\delta(\sigma_D(t))^{-1} v] = \delta(\sigma_D(t))^{-1}[v] = \delta(\sigma_D(t))^{-1} s_v(z_t) 
\end{equation*}
for all $t\in \GGm$. 
Let $k$ be the coefficient of $\lambda$ in the expression of $\beta$ as a $\ZZ$-linear combination of the elements of $E$.  Then $\delta(\sigma_D(t)) = t^k$ for every $t \in \GGm$. Consequently
\(
s(z_t) = t^{-k} s_v(z_t)\) for all $t \in \GGm$.
By (\ESo) $k>0$, and we have proved (\ref{tpext1}). 

For (\ref{tpext2}) we have to prove that $s_v(z_0) = [v] \in V/T_{z_0}X_0$ is nonzero. Condition (\ESth) together with Proposition~\ref{prop:codim1stab} tells us that $G\cdot z_0$ has codimension $1$ in $X_0$. It follows from Lemma~\ref{lem:Tz0} that $T_{z_0}X_0 = \fg \cdot z_0 \oplus \CC v_{\lambda}$. 
We now proceed by contradiction. Indeed, if $s_v(z_0)=[v]$ were zero than we would have $v \in \fg\cdot z_0 \oplus v_{\lambda}$. Since, by (\ESo), $v$ has nonzero $\Tad$-weight this would imply that $v \in \fg \cdot z_0$. The nonzero $\Tad$-weights in $\fg \cdot z_0$ are (by (\ESth)) the same as those in $\fg \cdot x_0$, that is, they are the elements of $R^+ \setminus E^{\perp}$ (by (\ref{weightsgx0}) of Lemma~\ref{lem:isotropygroup}). So if $\beta \notin R^+ \setminus E^{\perp}$ we are done. We only need to deal with the case where  $\beta \in R^+ \setminus E^{\perp}$. Then the $\Tad$-weight space in $\fg \cdot z_0$ of weight $\beta$ is the line spanned by $X_{-\beta} z_0$. Now (\ESf) tells us that $v$ cannot belong to that line:  $X_{-\beta} z_0$ has a nonzero projection to $V(\xi)$, whereas $v$ does not.

We now prove the claim that (\ref{tpext1}) and (\ref{tpext2}) establish the
proposition. Denote by $X_0^{\le 1}$ the union of $G\cdot x_0$ and all
$G$-orbits of codimension $1$ in $X_0$. Then $X_0^{\le1}$ is open
because $X_0$ has finitely many orbits, and it is smooth because $X_0$
is normal. Again by the normality of $X_0$, $s$ extends to $X_0$ if
and only if it extends to $X_0^{\le 1}$ (cf.~\cite[Lemma 3.7]{brion-ihs-arxivv2}). Since  $X_0^{\le 1}$ is smooth, the normal sheaf  $\shN_{X_0^{\le 1}}$ of $X_0^{\le 1}$ in $V$, which is the restriction of $\shN_{X_0}$ to $X_0^{\le 1}$, is locally free. The claim follows. 
\end{proof}

\section{Reduction to classification of spherical modules}  \label{sec:reduction}
In this section we reduce the proof of Theorem~\ref{thm:main} to a case-by-case verification, that is, we reduce it to Theorem~\ref{thm:cbc}. This reduction (formally, Corollary~\ref{cor:mainthm}) does not use the fact that $G$ is of type $\ssA$: if Theorem~\ref{thm:cbc} holds for groups of arbitrary type, then so does Theorem~\ref{thm:main}.  
We first introduce some more notation. We will use $R$ for the radical of $G$; since $G$ is reductive, $R$ is the connected component $Z(G)^{\circ}$ of $Z(G)$ containing the identity. 
When $(G,W)$ is a spherical module and $\wm$ is its weight monoid, we will use 
$\tM^G_W$ for the moduli scheme $\tM_{\wm}$ (in fact, we check in Lemma~\ref{lem:conjandwm}  that $\tM^G_{\wm}$ is, up to isomorphism (of schemes),  independent of the choice of maximal torus $T$ and Borel subgroup $B$ and therefore determined by the pair $(G,W)$).
We introduce this notation because we will have to relate moduli
schemes for different modules and different groups to one
another. More generally, when $\Gamma$ is the weight monoid of a
multiplicity-free  $G$-variety $X$, $\tM_{\Gamma} = \tM^G_{\Gamma}$ will stand for the moduli scheme $\tM_Y$ of \cite{alexeev&brion-modaff} with $Y= X\quot U$. 
Given a spherical module $(G,W)$ we will also use $\rho: G \to \GL(W)$ for the representation and we put
\[G^{\st}:= G' \times \GL(W)^G.\] 
 
We begin with an overview of the reduction.
To make the classification of spherical modules in \cite{kac-rmks, benson-ratcliff-mf, leahy} possible, several issues had to be dealt with (see \cite[Section 5]{knop-rmks}). Indeed, Knop's List gives the \emph{saturated} \emph{indecomposable} spherical modules up to \emph{geometric equivalence}. We begin by recalling the definitions of these terms from \cite[Section 5]{knop-rmks}.
\begin{defn} \label{def:geomeqv}
\begin{enumerate}[(a)]
\item Two finite-dimensional representations $\rho_1\colon G_1 \to \GL(W_1)$ and $\rho_2\colon G_2 \to \GL(W_2)$ are called \emph{geometrically equivalent} if there is an isomorphism of vector spaces $\phi \colon W_1 \to W_2$ such that for the induced map\footnote{By definition, $\GL(\phi)(f) = \phi \circ f \circ \phi^{-1}$ for every $f\in \GL(W_1)$.} $\GL(\phi) \colon \GL(W_1) \to \GL(W_2)$ we have $\GL(\phi)(\rho_1(G_1))=\rho_2(G_2)$. 
\item By the {\em product} of the representations $(G_1,W_1),\ldots,(G_n,W_n)$ we mean the representation $(G_1\times\ldots\times G_n, W_1 \oplus \ldots \oplus W_n)$. 
\item A finite-dimensional representation $(G,W)$ is \emph{decomposable} if it is geometrically equivalent to a representation of the form $(G_1 \times G_2, W_1 \oplus W_2)$ with $W_1$ a non-zero $G_1$-module and $W_2$ a non-zero $G_2$-module. It is called \emph{indecomposable}
if it is not decomposable. 
\item A finite-dimensional representation $\rho\colon G \to \GL(W)$ is called \emph{saturated} if the dimension of the center of $\rho(G)$ equals the number of irreducible summands of $W$.
\end{enumerate}
\end{defn}
\begin{remark} \label{rem:geomeqv}
\begin{enumerate}[(a)]
\item If $\rho$ is saturated and multiplicity-free, then the center of
  $\rho(G)$ is equal to the
  centralizer $\GL(W)^G$.
\item Suppose $(G_1,W_1)$ and $(G_2,W_2)$ are geometrically equivalent representations. Then 
$(G_1,W_1)$ is spherical if and only if $(G_2,W_2)$ is, and $(G_1,W_1)$ is saturated if and only if $(G_2,W_2)$ is. \label{geomeqvxy}
\end{enumerate}
\end{remark}

\begin{example}(\cite[p.311]{knop-rmks})
The spherical modules $(\SL(2), S^2\CC^2)$ and $(\SO(3), \CC^3)$ are geometrically equivalent. Every finite-dimensional representation is geometrically equivalent to its dual representation. The spherical module 
\begin{align*}
(\SL(2) \times \GGm \times \SL(2)) &\times (\CC^2 \oplus \CC^2) &\longrightarrow  &&\CC^2 \oplus \CC^2 \\
((A,t,B)&,(x,y))&\mapsto &&(tAx,tBy)
\end{align*}
is indecomposable but not saturated.
\end{example}

For our reduction to Theorem~\ref{thm:cbc}, we deal with geometric equivalence and products of spherical modules  in a straightforward matter. Indeed, we prove in Proposition~\ref{prop:geomeqv} that if $(G_1,W_1)$ and $(G_2,W_2)$ are geometrically equivalent spherical modules, then $\tM^{G_1}_{W_1} \isom \tM^{G_2}_{W_2}$ as schemes. That the tangent space to $\tM^G_W$ behaves as expected under products is proved in Proposition~\ref{prop:nshfprod}.
Dealing with the fact that the classification consists of
\emph{saturated} spherical modules requires a bit more effort. Indeed,
we could not establish an \emph{a priori} isomorphism between
$\tM^{\overline{G}}_W$ and $\tM^G_W$, where $(\overline{G},W)$ is a
(saturated) spherical module and $G$ is a  subgroup of $\overline{G}$
containing $\overline{G}'$ such that $(G,W)$ is spherical. This is why
in Theorem~\ref{thm:cbc}  we cannot restrict ourselves  to the modules $(\overline{G},W)$ of Knop's List. 
We circumvent this difficulty by proving in Proposition~\ref{prop:torrestr} that even when $(G^{\st},W)$ is decomposable Theorem~\ref{thm:cbc} implies  the equality 
\begin{equation} \label{eq:TX0Ms}
\dim T_{X_0}\tM^{G^{\st}}_{W} = 
\dim T_{X_0}\tM^{G' \times \rho(R)} _W
\end{equation}
for a spherical module $\rho:G \to \GL(W)$ with $G$ of type $\ssA$. In (\ref{eq:TX0Ms}), by abuse of notation, $X_0$ on each side denotes the unique closed orbit of the corresponding moduli scheme. From Proposition~\ref{prop:leahy} we have
that $(G'\times \rho(R),W)$ is geometrically equivalent to $(G,W)$.
Using Theorem~\ref{thm:cbc} and Lemma~\ref{lem:psigmax} we then deduce that $\dim T_{X_0}\tM^{G^{\st}}_{W}= d_W$, thus proving Corollary~\ref{cor:mainthm}. 

\begin{remark}    \label{rem:saturation}
Theorem~\ref{thm:main} proves, \emph{a posteriori}, that $\tM^{\overline{G}}_W$ and $\tM^G_W$ are isomorphic, when $\overline{G}$ is of type $\ssA$, $(\overline{G},W)$ is a (saturated) spherical module and $G$ is a  subgroup of $\overline{G}$ containing $\overline{G}'$ such that $(G,W)$ is spherical. We note that Remarks \ref{rem:vgg16} and \ref{rem:propcase17} show that, contrary to the tangent space $T_{X_0}\tM^G_{W}$, the $\Tad$-module $\Vgg$ that contains it does in general depend on the subgroup $G$  of $\overline{G}$ as above: these remarks give instances where the inclusion $(V/\overline{\fg} \cdot x_0)^{\overline{G}_{x_0}} \inn \Vgg$ is strict. (Recall that $V = \oplus_{\lambda \in E} V(\lambda)$ with $E$ the basis of the weight monoid of the dual module $W^*$.) Furthermore, we expect that the isomorphism $\tM^G_{W} \isom \tM^{\overline{G}}_{W}$ cannot follow from ``very general'' considerations, as the following example, where $\wm$ is not the weight monoid of a spherical module $W$, illustrates. Take $\overline{G}  = \SL(3)\times \GGm$, $G=\SL(3)$ and $\wm = \<\omega_1+\varepsilon, \omega_2 + \varepsilon\>$, where $\epsilon$ is a nonzero character of $\GGm$. Set $V = V(\omega_1+\varepsilon)^* \oplus V(\omega_2+\varepsilon)^*$ as in Section~\ref{subsec:emb}. Since $\wm$ is $G$-saturated, $T_{X_0}\tM_{\wm}^{\overline{G}} \isom (V/\overline{\fg} \cdot x_0)^{\overline{G}_{x_0}}$ and $T_{X_0}\tM_{\wm}^{G} \isom \Vgg$ by Remark~\ref{rem:Psatinjtan}. A direct calculation shows that
$\dim \Vgg=1$, whereas  $\dim (V/\overline{\fg} \cdot x_0)^{\overline{G}_{x_0}} = 0$. 
\end{remark}

The following proposition explains how a general spherical module $(G,W)$ fits into the classification of spherical modules. It is (somewhat implicitly) contained in \cite[Section 2]{leahy} and \cite[Section 5.1]{camus}. Recall that given a spherical module $(G,W)$, we put $G^{\st}:=G' \times \GL(W)^G$.
\begin{prop}[Leahy] \label{prop:leahy}
Suppose $\rho:G \to \GL(W)$ is a spherical module. Then the following hold:
\begin{enumerate}[(i)]
\item If $(G,W)$ is saturated and indecomposable, then $(G,W)$ is geometrically equivalent to an entry in Knop's List; \label{leahy0}
\item $(G^{\st}, W)$ is a saturated spherical module;  \label{leahy1}
\item $(G^{\st}, W)$ is  geometrically equivalent to a product of indecomposable saturated spherical modules; \label{leahy2}
\item 
$\rho(R) \inn \GL(W)^G$ and  $\rho(G) = \rho(G')\rho(R) \inn \GL(W)$; \label{leahy3}
\item Suppose $(G_1, W_1), (G_2,W_2),\ldots, (G_n,W_n)$ are spherical modules and let $(K,E)$ be their product. Suppose  $(K,E)$ and   $(G^{\st}, W)$ are geometrically
   equivalent and  denote by $\phi : W \to E$ a linear isomorphism
   establishing their geometric equivalence (see
   Definition~\ref{def:geomeqv}). If $A = \GL(\phi)(\rho(R))$, then $A \inn \GL(E)^K$ and $(G,W)$ is geometrically equivalent to $(K' \times A, E)$. 
 \label{leahy4}
\end{enumerate}
 \end{prop}
\begin{proof}
Assertion (\ref{leahy0}) just says that Knop's List contains all indecomposable saturated spherical modules up to geometric equivalence (see \cite[Theorem 2.5]{leahy} or \cite[Theorem 2]{benson-ratcliff-mf}).
Next, let $b$ be the number of irreducible components of $(G,W)$. 
Assertion (\ref{leahy1}) follows from the fact that $\GL(W)^G \isom \GGm^b$ (because $W$ is a multiplicity-free $G$-module). 
Assertion (\ref{leahy2}) follows from the fact that if $(G_1\times G_2, W_1 \oplus W_2)$ is saturated (resp.\ spherical) then $(G_1, W_1)$ and $(G_2,W_2)$ are saturated (resp.\ spherical). 
 We come to (\ref{leahy3}). Note that $R$ commutes with $G$ and so $\rho(R)$ commutes with $\rho(G)$ hence the first assertion. For the second, we use a well-known decomposition of reductive groups: $G = G'R$.  
 Finally we prove (\ref{leahy4}). Let us call $\psi\colon K \to \GL(E)$ and $\rho^{\st}\colon G^{\st} \to \GL(W)$ the representations. Then $\GL(\phi)\colon \rho^{\st}(G^{\st}) \to \psi(K)$ is an isomorphism of algebraic groups. As $\GL(W)^G \inn Z(G^{\st}$), its image $\GL(\phi)(\GL(W)^G)$ belongs to the center of $\psi(K)$, which is a subset of $\GL(E)^K$. Because $\rho(R) \inn GL(W)^G$, it follows that $A=\GL(\phi)(\rho(R)) \inn \GL(E)^K$. This proves the first assertion.  Next, note that $\GL(\phi)(\rho(G)) = \GL(\phi)(\rho(G'))\cdot \GL(\phi)(\rho(R))$. Moreover, 
$$\GL(\phi)(\rho(G')) = \GL(\phi)(\rho((G^{\st})')) = [\GL(\phi)(\rho^{\st}(G^{\st}))]' = \psi(K)' = \psi(K')$$ 
and the second assertion follows. 
\end{proof}

The following lemma is well-known and straightforward. 
\begin{lemma} \label{lem:spherrestr}
Let $X$ be an affine $G$-variety and let $H$ be a connected subgroup
of $G$ containing $G'$ ($H$ is reductive by
Lemma~\ref{lem:redsgr}). Let $B_H$ be the Borel subgroup $B\cap H$ of
$H$ and let $p: X(B) \onto X(B_H)$ be the restriction map.  If we
consider $X$ as an $H$-variety, then its weight monoid is $p(\Lambda^+_{(G,X)})$. 

If, moreover, $X$ is an affine \emph{spherical $G$-variety}, then the following are equivalent
\begin{enumerate}[(i)]
\item$X$ is spherical as an $H$-variety; \label{spherrone}
\item  the restriction of $p$ to $\Lambda^{+}_{(G,X)}$ is injective \label{spherrtwo}
\item  the restriction of $p$ to $\Lambda_{(G,X)}$ is injective. \label{spherrthree}
 \end{enumerate}
\end{lemma}
\begin{proof}
The basic fact behind both assertions is that if $V=V(\lambda)$ is an irreducible $G$-module of highest weight $\lambda$, then restricting the action to $H$ makes $V$ into an irreducible $H$-module of highest weight $p(\lambda)$.  This immediately implies that $p(\Lambda^+_{(G,X)}) \inn \Lambda^+_{(H,X)}$. For the reverse inclusion, we can argue as follows. Let $R$ be the radical of $G$ and let $\mu$ be dominant weight of $H$ such that the isotypical component $A_{(\mu)}$ of $A = \CC[X]$ is nonzero. Then $R$ stabilizes $A_{(\mu)}^U$, and so  $A_{(\mu)}^U$ contains an eigenvector $f$ for $R$. Then $f$ is a $B$-eigenvector whose $B_H$-weight is $\mu$.  We have proved the first assertion.
 
For the equivalence of  (\ref{spherrone}) and (\ref{spherrtwo}), see for example \cite[Lemma 4.1]{camus}. The equivalence of (\ref{spherrtwo}) and (\ref{spherrthree}) follows from the fact that $\Lambda_{(G,X)}$ is the subgroup of $X(T)$ generated by $\Lambda^+_{(G,X)}$ (see e.g. \cite[Lemma 3.6.3]{losev-uniqueness}). 
\end{proof}

\begin{remark} \label{rem:spherrestr}
\begin{enumerate}[(1)]
\item Theorem 5.1 of \cite{knop-rmks} is a somewhat refined version of Lemma~\ref{lem:spherrestr}.
\item For every saturated indecomposable spherical module $(G,W)$, Knop's List, following \cite{leahy},
gives a basis for $\<\ker p\>_{\CC} \cap \<\Lambda_W\>_{\CC} \inn \ft^*$, where $p$ is as in Lemma~\ref{lem:spherrestr}. In Knop's List, $\<\ker p\>_{\CC}$ is denoted $\fz^*$ and  $\<\Lambda_W\>_{\CC}$ is denoted $\fa^*$.
\item For a simple example, suppose $G=\GL(n)$. Then $G'=\SL(n)$ and  the kernel of the restriction map $X(B) \onto X(B\cap G')$ is $\<\om_n\>_{\ZZ}$. 
\end{enumerate}
\end{remark}

Our first step in the reduction is to verify that our moduli schemes are `invariant' under geometric equivalence. We do this in Proposition~\ref{prop:geomeqv} below. First we prove an auxiliary proposition stating that our moduli schemes behave `as expected' under surjective group homomorphisms (Proposition~\ref{prop:cover}). It will be of use on several occasions and is followed by a few lemmas we need in the proof of Proposition~\ref{prop:geomeqv}. 

Before getting started we recall the following well-known fact (see e.g.\ \cite[Corollary C of \S 21.3]{humphreys-lag}) which will be used in what follows without being explicitly mentioned.
\begin{prop}\label{prop:surjBandT}
Suppose $f\colon G \to H$ is a surjective homomorphism of connected linear algebraic groups. If $T$ is a maximal torus of $G$ and $B$ is a Borel subgroup of $G$, then $f(T)$ is a maximal torus of $H$ and $f(B)$ is a Borel subgroup of $H$.
\end{prop}
We also recall Yoneda's lemma (in its formulation taken from \cite[Lemma VI-1]{geomschemes}). In the proof of Proposition~\ref{prop:cover} we prove that two functors $\Sch_{\CC} \to \Sets$ are isomorphic and conclude that the representing schemes are isomorphic as schemes (over $\CC$). 
\begin{prop}[Yoneda's lemma] \label{prop:yoneda}
If  $\Cat$ is a category, $X_1, X_2$  are two objects in
$\Cat$ and the two functors $ F_1, F_2 :  \Cat \to \Sets$ , 
 with  $F_i = \Hom_{\Cat}( -, X_i)$ are isomorphic
(as functors)  then  $X_1$ is isomorphic to $X_2$  in $\Cat$.
\end{prop}

\begin{prop}  \label{prop:cover} 
Suppose $f: G \onto  H$  is a surjective group homomorphism between connected reductive groups. Put $T_H := f(T)$ and $B_H=f(B)$ and write $f^*$ for the map $X(T_H) \into X(T)$ given by $\lambda \to \lambda \circ f$.
Let $\wm \inn X(T_H)$ be the weight monoid of an  affine spherical $H$-variety (with respect to the Borel subgroup $B_H$).  Then $\tM_{\wm}^H \isom \tM^G_{f^*(\wm)}$ as schemes.
\end{prop} 
\begin{proof}
Using the description in \cite[Proposition 2.10]{alexeev&brion-modaff} we show that the two functors $\shM_{\wm}^H$ and $\shM^G_{f^*(\wm)}$ are isomorphic. Underlying the proof, which is an exercise in `abstract nonsense,' is the (elementary) fact that the category of $H$-modules is equivalent to the category of $G$-modules with highest weights in $f^*(\Lambda_H^+)$, where $\Lambda_H^+$ is the set of dominant weights in $X(T_H)$ with respect to the Borel subgroup  $B_H:=f(B)$ of $H$.

The aforementioned equivalence is established by (the co-restriction of) the obvious functor $\shE: H\text{-modules} \to G\text{-modules}$ induced by $f$. Further let $\shF: T_H\text{-modules} \to T\text{-modules}$ also be the obvious functor induced by $f$. Let $A$ be the $T_H$-algebra $\CC[\wm]$ and put $R:= \Coind^H_{B_H}(A)$ (For a quick discussion of the co-induced module see \cite[p.104]{alexeev&brion-modaff}). Note that $\shF(A)$ and $A$ have the same underlying set and that $\shF(A)$  is naturally a $T$-algebra for the same multiplication  as that of $A$. 
We also introduce the contravariant functor
$\shN \colon \Sch_{\CC} \to \Sets$ which assigns to a scheme $X$ the set of $G$-multiplication laws on $ \Oh_X \otimes_{\CC} \shE(R)$ that extend the $T$-multiplication law on $\shF(A)$. Note that $R$ and $\shE(R)$ have the same underlying set as well and so, trivially, $\shN(X) = \shM_{\wm}^H(X)$ for any scheme $X$.  What remains is to prove that $\shN \isom \shM^G_{f^*(\wm)}$. 

Let $D$ be the $T$-algebra $\CC[f^*(\wm)]$.  Then there exists an
isomorphism of $T$-algebras $\varphi: D \to \shF(A)$ by the well-known
fact that multiplicity-free $T$-algebras without zero divisors are determined up to $T$-isomorphism by their weight monoid, and Lemma~\ref{lem:surjandwm} below. Furthermore, the inclusion $j\colon\shF(A) \into \shE(R)$ is $T$-equivariant. By the universal property of the map $\imath\colon D \to \Coind^G_{B}(D)$, there exists a unique $G$-equivariant map of modules $\psi\colon \Coind^G_{B}(D) \to \shE(R)$ such that $\psi \circ \imath = j \circ \varphi$.   
Now, $\psi$ is in fact an isomorphism because it restricts to an isomorphism $[\Coind^G_{B}(D)]^U \to \shE(R)^U$.  
Using the commutative diagram
\[\begin{CD} 
D @>{\imath}>> {\Coind^G_{B}(D)}\\ 
@VV{\varphi}V @VV{\psi}V\\ 
\shF(A)@>j>> \shE(R) 
\end{CD}\]
it is straightforward to find the natural transformations that establish the isomorphism $\shN \isom \shM^G_{f^*(\wm)}$.
\end{proof}
\begin{remark}
One can prove that, with the obvious rephrasing, Proposition~\ref{prop:cover} holds for all moduli schemes $\tM_Y$ (as defined in \cite{alexeev&brion-modaff}) with $Y$ a multiplicity-finite affine $T_H$-scheme. 
\end{remark}

For what follows, some temporary notation will be useful. Suppose $T \inn B \inn G$ are a maximal torus and a Borel subgroup of $G$ and let $X$ be a  $G$-scheme, where the action is $\rho\colon G \to \Aut(X)$. Then we denote by $\Lambda(T,B,G,\rho)$ the weight set with respect to $B$ and $T$ of $X$, viewed as a $G$-scheme with action given by $\rho$. 
\begin{lemma} \label{lem:surjandwm}
Let $G,H,f, f^*, T_H$ and $B_H$ be as in proposition~\ref{prop:cover}. Let $X$ be an $H$-scheme and denote the action $\rho\colon H \to \Aut(X)$. Put
\begin{align*}
\cS&:=\Lambda(T_H,B_H,H,\rho) \\
\Gamma &:= \Lambda (T,B,G,\rho \circ f)
\end{align*}
Then $f^*(\cS) = \Gamma.$
\end{lemma}
\begin{proof}
This is  a straightforward verification. We will use $\cdot$ for the action of $G$ on $\CC[X]$ induced by $\rho \circ f$ and $*$ for the action of $H$ induced by $\rho$. Consequently, for $g \in G$ and $P \in \CC[X]$, we have $g \cdot P = f(g) * P$. 

First, we prove that $f^*(\cS) \inn \Gamma$. Take $\lambda \in \cS$. Then there exists a nonzero $P \in \CC[X]$ so that $\tilde{b}* P = \lambda(\tilde{b})P$ for every $\tilde{b} \in B_H$. Since $B_H = f(B)$, this means that $b\cdot P = f(b)*P = \lambda(f(b))P$ for every $b \in B$. In other words $f^*(\lambda) = \lambda \circ f$ is in $\Gamma.$ 

For the reverse inclusion, $\Gamma \inn f^*(\cS)$, take $\lambda \in \Gamma$. Then there exists a nonzero $P \in \CC[X]$ so that $b\cdot P = \lambda(b)P$. Since $b \cdot P = f(b)*P$, this implies that $P$ is a $B_H$-eigenvector. It follows that there exists $\delta \in X(B_H)$ so that $b\cdot P = f(b)*P = \delta(f(b))P$. Consequently, $\lambda = \delta \circ f = f^*(\delta) \in f^*(\cS)$.  
\end{proof}

Lemma~\ref{lem:conjandwm} checks, when $X$ is a multiplicity-free
 affine $G$-variety, that, as expected, $\tM^G_{\Lambda^+_X}$ is independent up to isomorphism of schemes (over $\CC$) of the choice of maximal torus $T$ and Borel subgroup $B$. In other words,  $\tM^G_{\Lambda^+_X}$ is determined by the pair $(G,X)$.
\begin{lemma}\label{lem:conjandwm}
Suppose $T_1, T_2$ are maximal tori of $G$ and $B_1,B_2$ are Borel subgroups of $G$ such that  $T_1 \inn B_1$ and $T_2 \inn B_2$. Suppose $X$ is a multiplicity-free  affine $G$-variety with action $\rho\colon G \to \Aut(X)$. For $i=1,2$, put $\Gamma_i := \Lambda(T_i, B_i, G, \rho)$.  Then we have an isomorphism of schemes
\[ \tM^G_{\Gamma_1} \isom \tM^G_{\Gamma_2}\] 
\end{lemma}

\begin{proof}
First note that there exists a $g \in G$ so that
\(g B_1 g^{-1} = B_2 \text{ and } gT_1g^{-1} = T_2.\) Let $f$ be conjugation by $g$, that is,
\(f\colon G\to G, h \mapsto ghg^{-1}.\)
Put $\cS = \Lambda(T_1, B_1, G, \rho\circ f)$. Then, by Lemma~\ref{lem:surjandwm}, $f^*(\Gamma_2)=\cS$. Further,  the map 
\[X\to X \colon x \mapsto \rho(g)(x)\] is a $G$-equivariant isomorphism between $\rho$ and $\rho \circ f$. This implies that the two representations of $G$ on $\CC[X]$, induced by $\rho$ and $\rho \circ f$ respectively, are isomorphic too. Consequently, $\Gamma_1 = \cS$, whence $f^*(\Gamma_2) = \Gamma _1$.  Proposition~\ref{prop:cover} then tells us that $\tM^G_{\Gamma_1} \isom \tM^G_{\Gamma_2}$. 
\end{proof}

\begin{lemma} \label{lem:modandmatrix}
Suppose $\rho\colon G \to \GL(W)$ is a spherical module and let $E=\{e_1, \ldots, e_n\}$ be a basis of $W$. 
As usual, this basis defines an isomorphism of algebraic groups
$\Mat_E: \GL(W) \to \GL(n)$, where for $h\in \GL(W)$, $\Mat_E(h)$ is the matrix uniquely specified by the property that   $h (e_i)  =  \sum_p   (\Mat_E(h))_{pi} e_p$ for all $i$. Put $H:=\Mat_E(\rho(G)) \inn \GL(n)$. Then we have an isomorphism of schemes
\[\tM^G_W \isom \tM^H_{\CC^n}.\]
\end{lemma}
\begin{proof}
The map \[p\colon:G\to \GL(n), g\mapsto \Mat_E(\rho(g))\] makes $\CC^n$ into a $G$-module and  Proposition~\ref{prop:cover} tells us that
$\tM^G_{\CC^n} \isom \tM^H_{\CC^n}$. It is a straightforward verification that the linear isomorphism 
\[\phi: \CC^n \to W, (a_1, \ldots, a_n) \mapsto \sum_i a_i e_i\] is $G$-equivariant. It follows that $\CC[\CC^n]$ and $\CC[W]$ are isomorphic $G$-modules and therefore that
$\tM^G_{\CC^n} \isom \tM^G_W$. 
\end{proof}

\begin{prop} \label{prop:geomeqv}
Suppose $\rho_1\colon G_1\to \GL(W_1)$ and $\rho_2\colon G_2\to \GL(W_2)$ are geometrically equivalent spherical modules.  Then we have an isomorphism of schemes \[\tM^{G_1}_{W_1}\isom \tM^{G_2}_{W_2}.\]
Consequently, $\dim T_{X_0} \tM^{G_1}_{W_1} =  \dim T_{X_0} \tM^{G_2}_{W_2}$, where by abuse of notation, $X_0$ on each side denotes the unique closed orbit of the corresponding moduli scheme. 
\end{prop}
\begin{proof}
Put $n:=\dim W_1$. 
As a straightforward verification shows, $(G_1,W_1)$ and $(G_2,W_2)$ are geometrically equivalent if and only if $\dim W_1 = \dim W_2$ and there exists a basis $E_1$ of $W_1$ and a basis $E_2$ of $W_2$ such that, using notation introduced in Lemma~\ref{lem:modandmatrix}, we have the following equality of subsets of $\GL(n)$: 
\[\Mat_{E_1}(\rho_1(G_1)) = \Mat_{E_2}(\rho_2(G_2)).\]
Put $H = \Mat_{E_1}(\rho_1(G_1)) = \Mat_{E_2}(\rho_2(G_2))$.
Then, by Lemma~\ref{lem:modandmatrix}, we have $\tM^{G_1}_{W_1}\isom \tM^H_{\CC^n} \isom \tM^{G_2}_{W_2}$, which proves the first assertion of the proposition.

Taking into account that, for $i=1,2$,  $X_0$  is in the
closure of all orbits on $\tM^{G_i}_{W_i}$ under the action of the adjoint torus of $G_i$, it follows from
the semicontinuity of the dimension of the Zariski tangent space
that $\dim T_{X_0} \tM^{G_i}_{W_i}$ is the maximum possible dimension
of the tangent space at closed points of $\tM^{G_i}_{W_i}$.
The second assertion follows.
\end{proof}

Lemma~\ref{lem:redsgr}  recalls an elementary fact we use in the proof of Proposition~\ref{prop:cbc}. 
\begin{lemma} \label{lem:redsgr}
Suppose $H$ is a connected reductive algebraic group and $K$ is a closed subgroup of $H$ containing $H'$. Then $K$ is reductive.
\end{lemma}
\begin{proof}
We have to prove that the unipotent radical $F$ of $K$ is trivial.  
As is well known, $H=H'R$, where $R$ is the radical of $H$. 
Since $H' \inn K$ we have that $H'$ normalizes $F$. Clearly $R$ (as a subset of the center of $H$) normalizes $F$. It follows that $F$ is a normal connected unipotent subgroup of $H$. Since $H$ is reductive, $F$ is trivial.
\end{proof}

The next lemma states how the invariant $d_W$ behaves under restriction of center, geometric equivalence and taking products. 
We need two of its assertions in the proof of Proposition~\ref{prop:cbc}. It will also be of use later.

\begin{lemma}\label{lem:dW}
\begin{enumerate}[(a)]
\item Suppose $\overline{G}$ is a connected reductive group and let $(\overline{G},W)$ be a spherical module. Let $G$ be a connected (reductive) subgroup of $\overline{G}$ containing $\overline{G}'$. Assume that the restriction $(G,W)$ of $(\overline{G},W)$ is also spherical. Then both modules have the same invariant $d_W$. \label{dWres}
\item Suppose $(G_1,W_1)$ and $(G_2,W_2)$ are geometrically equivalent spherical modules. Then $d_{W_1} = d_{W_2}$.  \label{dWgeqv}
\item Let $(G_1,W_1), (G_2,W_2), 
\ldots, (G_n,W_n) $ be spherical modules and let $(G, W)$ be their product. Then 
\[d_W= d_{W_1} +\ldots+ d_{W_n}.\] \label{dWprod}
\end{enumerate}
\end{lemma} 
\begin{proof}
For (\ref{dWres}) just combine Lemma~\ref{lem:spherrestr} with Lemma~\ref{lem:psigmax} (or with 
Lemma~\ref{lem:camusbound}). Next, to prove (\ref{dWgeqv}), let $\rho_1:G_1 \to \GL(W_1)$ and $\rho_2:G_2 \to \GL(W_2)$ be the representations. Suppose $\phi: W_1 \to W_2$ is a linear isomorphism establishing the geometric equivalence. Then $\GL(\phi): \rho_1(G_1) \to \rho_2(G_2)$ is an isomorphism of algebraic groups. Let $U$ be a maximal unipotent subgroup of $G_1$. Then $U_1:=\rho_1(U)$ is a maximal unipotent subgroup of $\rho_1(G_1)$ and $U_2:=\GL(\phi)(\rho_1(U))$ is a maximal unipotent subgroup of $\rho_2(G_2)$. Moreover, $\phi$ induces an isomorphism of vector spaces $W_1^{U_1}\isom W_2^{U_2}$ and an isomorphism of algebras
$\CC[W_1]^{U_1} \isom \CC[W_2]^{U_2}$. Since, for $i\in\{1,2\}$, $\dim W_i^{U_i}$ is the number of irreducible components of $W_i$ and $\dim \Spec(\CC[W_i]^{U_i})$ is the rank of the weight group of $W_i$, Lemma~\ref{lem:camusbound} proves assertion (\ref{dWgeqv}).  
We turn to (\ref{dWprod}). This assertion follows by combining  Lemma~\ref{lem:camusbound} with the fact that $\Lambda^+_{(G,W)} = \Lambda^+_{(G_1,W_1)} \oplus \ldots \oplus \Lambda^+_{(G_n,W_n)}$.
\end{proof} 
 
\begin{prop} \label{prop:cbc} 
Let $(G,W)$ be an indecomposable saturated spherical module. Suppose that 
$G = G^{\st}$ (hence $Z(G)^{\circ} = \GL(W)^G$) Êand 
that $H \inn Z(G)^{\circ}$ is a subtorus such that $W$ is spherical 
for $G' \times H$.
Assume that the conclusion of Theorem~\ref{thm:cbc} holds for every pair $(\overline{G},W)$ in Knop's List with $\overline{G}$  of a  type that occurs in the decomposition of $G'$ into almost simple components.
Then $\dim T_{X_0}\tM^{G'\times H}_W= \dim T_{X_0} \tM^{G}_W = d_W$, where by abuse of notation each $X_0$ stands for the unique closed orbit of the corresponding moduli scheme. 
\end{prop}
\begin{proof}
By Proposition~\ref{prop:leahy}(\ref{leahy0}), $(G, W)$ is geometrically equivalent to an entry in Knop's List, say $(\overline{K},E)$. Suppose $\phi: W \to E$ is a map establishing the geometric equivalence (see Definition~\ref{def:geomeqv}) between $(G, W)$ 
and  $(\overline{K},E)$. We first claim that
there exists a connected reductive subgroup $K \inn \overline{K}$ containing $\overline{K}'$ for which $E$ is still spherical and so that $\phi$ also establishes the geometric equivalence of $(G' \times H, W)$ and
$(K,E)$.  Indeed, let $\rho: G \to \GL(W)$ and $\psi: \overline{K} \to
\GL(E)$ be the representations and put $\rho_1 = \rho|_{G' \times
  H}$. Then $F:=\GL(\phi)(\im \rho_1)$ is a connected subgroup of
$\psi(\overline{K})$ containing $\psi(\overline{K})' =
\psi(\overline{K}')$. The reason is that $\GL(\phi)(\im \rho_1)$ contains $\GL(\phi)((\im \rho)') = (\GL(\phi)(\im \rho))' = (\psi(\overline{K}))'$, since $\im \rho_1$ contains $(\im \rho)'$. 
Now set $\widetilde{K}:=\psi^{-1}(F)$ and let $K$ be the identity component of $\widetilde{K}$. Then $\widetilde{K}$ is a subgroup of $\overline{K}$ containing $\overline{K}'$ and therefore so is $K$. Lemma~\ref{lem:redsgr} then yields that $K$ is reductive. Clearly $\psi(\widetilde{K}) = F = \GL(\phi)(\im \rho_1)$ (since $F \inn \im \psi$). Since  $\psi(\widetilde{K})=\psi(K)$ because $\psi(\widetilde{K})$ is connected (see e.g.~\cite[Proposition B of \S7.4]{humphreys-lag}),  $\phi$ establishes the geometric equivalence of $\rho_1$ and $\psi |_K$. It also follows (by Remark~\ref{rem:geomeqv}(\ref{geomeqvxy})) that $E$ is a spherical module for $K$.  This proves the claim.
 
By Lemma~\ref{lem:dW}(\ref{dWres}), $(G,W)$ and $(G' \times H, W)$ have the same invariant $d_W$, and $(\overline{K},E)$ and $(K,E)$ have the same invariant $d_E$. 
By assumption, the conclusion of Theorem~\ref{thm:cbc} holds for $(\overline{K},E)$ and so $\dim T_{X_0} \tM^{\overline{K}}_E=\dim T_{X_0} \tM^K_E = d_E$. Thanks to Lemma~\ref{lem:dW}(\ref{dWgeqv}), $d_E = d_W$. Finally, by  Proposition~\ref{prop:geomeqv}, 
$\dim T_{X_0} \tM^{\overline{K}}_E=\dim T_{X_0} \tM^G_W$ and $\dim T_{X_0} \tM^{K}_E=\dim T_{X_0} \tM^{G' \times H}_W$, and we have proved the proposition.  
\end{proof}

The next proposition reminds us that the normal sheaf behaves as expected with respect to products. 
\begin{prop} \label{prop:nshfprod}
Let $n$ be a positive integer. Suppose that for every positive integer $i\le n$ we have a finite-dimensional $G$-module $V_i$ and a $G$-stable closed subscheme $X_i$ of $V_i$. For every $i$, we put
\begin{align*}
&R_i := \CC[V_i] \\
&I_i := I(X_i) \inn R_i \text{ the ideal of }X_i \text{ in }V_i \\
&N_i := \Hom_{R_i}(I_i,R_i/I_i) 
\end{align*}
We also put
\begin{align*}
& V := \oplus_i V_i 
&& R := \CC[V] \isom \otimes_i R_i \\
&X:= X_1 \times \ldots \times X_n 
&& I:= I(X) \inn \CC[V] \\
&N:= \Hom_R(I,R/I) 
&&\widehat{R}_i:= \otimes_{j\neq i} R_j/I_j
\end{align*}
where all the tensor products are over $\CC$. We then have canonical isomorphisms of $R$-$G$-modules:
\[ N \isom \oplus_i(N_i \otimes_{R_i} R) \isom \oplus_i(N_i \otimes_{\CC} \widehat{R}_i).
\]
\end{prop}
\begin{proof}
It is clear that, for $1 \leq j \leq n$, we can consider $I_j$ as a subset
of $I$.
For $1 \leq i \leq n$ we define the $G$-stable $R$-submodule $\widetilde
{N}_i \inn N$ by
\[
        \widetilde {N}_i  = \{ \phi \in N  \text{ such that 
              $\phi (a) = 0$ when  $a$ is in $I_j$ and  $j \not= i$}  \}.
\]
Using \cite[Lemma 9]{northcott-syz} it follows that
$N = \oplus_{i=1}^n   \widetilde {N}_i$, and that
$\widetilde {N}_i$ is canonically isomorphic to 
$N_i \otimes_{R_i} R$ as an $R$-module
with the isomorphism being $G$-equivariant. In turn, $N_i \otimes_{R_i} R$ is canonically isomorphic to $N_i \otimes_{\CC} \widehat{R}_i$.
\end{proof}

\begin{cor} \label{cor:product}  
Let $n$ be a positive integer and suppose that for every positive integer $i \le n$, $G_i$ is a connected reductive group, $V_i$ is a finite-dimensional $G_i$-module and $X_i$ is a \emph{multiplicity-free} $G_i$-stable closed subscheme of $V_i$. Put $\overline{G}:= G_1\times\ldots\times G_n$. Define $N$ and $N_i$ as in Proposition~\ref{prop:nshfprod}. 
Then we have a canonical isomorphism of $\CC$-vector spaces
\begin{equation}
N^{\overline{G}} \isom \oplus_iN_i^{G_i}.
\end{equation}
\end{cor}
\begin{proof}
In this proof all the tensor products are over $\CC$. We introduce the
following notation for every $i\in \{1,\ldots,n\}$: $\widehat{G}_i := \times_{j\neq i} G_j$. 
Using Proposition~\ref{prop:nshfprod} (and its notation) 
we have that
\begin{equation} \label{eq:product}
N^{\overline{G}}\isom\oplus_i(N_i \otimes \widehat{R}_i)^{\overline{G}} = \oplus_i(N_i^{G_i} \otimes \widehat{R}_i^{\widehat{G}_i}) = \oplus_i(N_i^{G_i} \otimes \CC),
\end{equation}
where the last equality uses the multiplicity-freeness of $\widehat{R}_i$.
\end{proof}
\begin{remark}
An immediate consequence of this corollary is that if $(G_1,W_1)$ and $(G_2,W_2)$ are spherical modules and $(G,W)$ is their product, then $\dim T_{X_0} \tM_W = \dim T_{X_0}\tM_{W_1} + \dim T_{X_0}\tM_{W_2}$, where by abuse of notation each $X_0$ denotes the unique closed orbit of the corresponding moduli scheme. This is how we will use the corollary (in the proof of Corollary~\ref{cor:mainthm}).  
\end{remark}

\begin{prop}\label{prop:torrestr}
Suppose that for every $i \in \{1, \ldots, n\}$ we have an indecomposable saturated spherical module $(G_i, W_i)$. For every $i$, assume that $G_i = G^{\st}_i$ and that  the conclusion of Theorem~\ref{thm:cbc} holds for every pair $(\overline{G},W)$ in Knop's List with $\overline{G}$  of a type that occurs in the decomposition of $G_i'$ into almost simple components.
For every $i$ we put
\begin{align*}
&Z_i:=Z(G_i)^{\circ} = \GL(W_i)^{G_i};\\
&E_i:=\Lambda^+_{W_i},\ V_i := \oplus_{\lambda_\in E_i}V(\lambda);\\
&X_i = \overline{G_i x_i}, \text{ where } x_i = \sum\nolimits_{\lambda\in E_i}v_{\lambda}.
\end{align*}
Put $\overline{G}: = G_1 \times \ldots \times G_n$. 
We also define $N_i$ and $N$ as in Proposition~\ref{prop:nshfprod}.
Finally suppose that $A$ is a subtorus of $Z_1\times \ldots \times Z_n$ such that $W_1 \oplus \ldots \oplus W_n$ is spherical for $G:= G'_1 \times \ldots \times G'_n \times A$. Then
\begin{equation}
N^{G} =N^{\overline{G}}\label{eq:torrestr}
\end{equation}
\end{prop}
\begin{proof}
We continue to use the following notation
\begin{align*}
&\widehat{G}_i := \times_{j\neq i} G_j
&&\widehat{G}'_i = \times_{j\neq i} G'_j.
\end{align*}
In this proof all the tensor products are over $\CC$.
To prove (\ref{eq:torrestr}) it is sufficient (by Proposition~\ref{prop:nshfprod}) to prove that $(N_i \otimes \widehat{R}_i)^{G} = (N_i \otimes \widehat{R}_i)^{\overline{G}}$ for every $i$. We clearly have that   
\[(N_i \otimes \widehat{R}_i)^{G} = (N_i^{G'_i} \otimes \widehat{R}_i^{\widehat{G}'_i})^A \]
Recall from equation~(\ref{eq:product}) that $(N_i \otimes \widehat{R}_i)^{\overline{G}} = 
N_i^{G_i} \otimes F_0$, where $F_0:= \widehat{R}_i^{\widehat{G}_i} \isom \CC$ is the set of constants in  $\widehat{R}_i$. 
We will prove that  
\[F:= (N_i^{G'_i} \otimes \widehat{R}_i^{\widehat{G}'_i})^A = N_i^{G_i} \otimes F_0.\]
The inclusion $N_i^{G_i} \otimes F_0 \inn F$ is clear. For the other inclusion, 
assume, by contradiction, that $F$ is not a subspace of $N_i^{G_i} \otimes F_0$.  Then there exist a character $\lambda \in X(A)$,  a nonzero vector $v$ in $N_i^{G'_i}$ of weight $-\lambda$ and a nonzero vector $w$ of weight $\lambda$ in $\widehat{R}_i^{\widehat{G}'_i}$ such that $v\otimes w \notin N_i^{G_i} \otimes F_0$. It follows that $\lambda \neq 0$, for otherwise 
\[ v\otimes w \in N_i^{G'_i \times A} \otimes \widehat{R}_i^{\widehat{G}'_i\times A} = N_i^{G'_i \times A} \otimes F_0 =N_i^{G'_i \times p(A)} \otimes F_0\] where $p: \times_j Z_j \to Z_i$ is the projection, while Proposition~\ref{prop:cbc} tells us that $N_i^{G'_i \times p(A)} = N_i^{G_i}$ (because $W_i$ is spherical for $G_i' \times p(A)$).

Now, by Lemma~\ref{lem:torrestr} below, we have that $X_i$ is spherical for $G'_i \times \ker\lambda$, hence for $G'_i \times p(\ker\lambda)$, since $A$ acts on $X_i$ through the factor $Z_i$.  Again by Proposition~\ref{prop:cbc}, we have that $N_i^{G'_i \times p(\ker\lambda)} = N_i^{G_i}$. We obtain a contradiction: $v \in  N_i^{G'_i \times p(\ker\lambda)}$ since $v$ has $A$-weight $\lambda$, but $v\notin N_i^{G_i}$ since $\lambda$ is nonzero and therefore $p(A)\inn G_i$ does not fix $v$.
\end{proof}

\begin{lemma}\label{lem:torrestr}
Let $G_1$ and $G_2$ be connected reductive groups and let $A_1$ and $A_2$ be tori. Suppose that for every $i \in \{1,2\}$ we have a normal affine $G_i \times A_i$-variety $X_i$.  Let $A \inn A_1 \times A_2$ be a subtorus such that $X_1 \times X_2$ is spherical for the action restricted to $G_1 \times A \times G_2 \inn G_1 \times A_1 \times A_2 \times G_2$. If  $\lambda \in X(A)$ is such that the eigenspace $\CC[X_2]^{G_2}$ contains a nonzero $A$-eigenvector of weight $\lambda$, then $X_1$ is spherical for $G_1\times \ker \lambda$.
\end{lemma}
\begin{proof}
Pick Borel subgroups and maximal tori $T_1 \inn B_1 \inn G_1$ and $T_2
\inn B_2 \inn G_2$. In this proof we identify $X(A)$ with its image
under the canonical embeddings into $X(A\times T_i)$ for $i\in \{1,2\}$ and into $X(A\times T_1 \times T_2)$. 

Clearly, $X_1$ is spherical for $G_1 \times A$. If $X_1$ is not spherical for the subgroup $G_1 \times \ker \lambda$, then there are highest weight vectors $f_{\alpha}, f_{\beta} \in \CC[X_1]^{(B_1\times A)}$ of weight $\alpha$ and $\beta$ respectively such that $\alpha \neq \beta$ and
$\alpha = \beta$ on $\ker \lambda \inn T_1 \times A$. This implies that $\alpha - \beta = d\lambda$ for some integer $d$. Reversing the roles of $\alpha$ and $\beta$ if necessary, we assume $d$ nonnegative.

It is given that there is a $g_\lambda$  in $\CC[X_2]^{(A \times B_2)}$ of weight $\lambda$. We then have that the two $(B_1\times A\times B_2)$-eigenvectors $f_{\alpha} \otimes  1$ and $f_{\beta} \otimes g^d_{\lambda}$ in $\CC[X_1] \otimes \CC[X_2]$ have the same weight. This contradicts the sphericality of $X_1\times X_2$ for the action of $G_1\times A \times G_2$. 
\end{proof}

\begin{cor} \label{cor:mainthm}  
Let $(G,W)$ be a spherical module and let $\wm$ be its weight monoid. Assume that the conclusion of Theorem~\ref{thm:cbc} holds for every pair $(\overline{G},W)$ in Knop's List with $\overline{G}$  of a  type that occurs in the decomposition of $G'$ into almost simple components. Then 
\begin{equation}
\dim T_{X_0} \tM_{\wm} = d_W. \label{eq:goal}
\end{equation}

\end{cor}
\begin{proof}  
In this proof, by abuse of notation, $X_0$ will stand for the unique closed orbit of the relevant moduli scheme. 
By Proposition~\ref{prop:leahy} there exist indecomposable saturated spherical modules 
\((G_i,W_i)\) in Knop's List, with $i \in \{1,2,\ldots,n\}$, such that
$(G^{\st},W)$ is geometrically equivalent to the product $(K,E)$ of the  $(G_i,W_i)$, and such that
$(G,W)$ is geometrically equivalent to $(K'\times A, E)$ where $A$ is a subtorus of $\GL(E)^K$.
By assumption, the conclusion of Theorem~\ref{thm:cbc} holds for each $(G_i,W_i)$ and so 
\[\dim T_{X_0}\tM^{G_i}_{W_i} = d_{W_i} \text{ for every $i \in \{1,2,\ldots,n\}$}. \]
As a consequence, Corollary~\ref{cor:product} and Lemma~\ref{lem:dW}(\ref{dWprod}) yield that
\begin{equation}
\dim T_{X_0}\tM^K_E = d_E. \label{eq:dke}
\end{equation}
On the other hand, using that $\GL(E)^K = \times_i\GL(W_i)^{G_i}$, Proposition~\ref{prop:torrestr} tells us that
\begin{equation}
\dim T_{X_0}\tM^{K' \times A}_E = \dim T_{X_0} \tM^K_E, \label{dkae}
\end{equation}
whereas by Proposition~\ref{prop:geomeqv}
\begin{equation*}
\dim T_{X_0}\tM^G_W = \dim T_{X_0} \tM_E^{K' \times A}.
\end{equation*}
With equations (\ref{eq:dke}) and (\ref{dkae}) and Lemma~\ref{lem:dW}(\ref{dWres},\ref{dWgeqv}) this implies equation~(\ref{eq:goal}), as desired.
 \end{proof}

\section{Proof of Theorem~\ref{thm:cbc}} \label{sec:cases}
In this section, we prove Theorem~\ref{thm:cbc} through case-by-case verification. Formally the proof runs as follows. We have to check the theorem for the $8$ families in List~\ref{KnopLTypeA} below. For families (\ref{list.1}), (\ref{list.2}) and (\ref{list.3}), the arguments are given in Sections~\ref{subsec:case1}, \ref{subsec:case2} and \ref{subsec:case3}, respectively. For family (\ref{list.15}), the theorem follows from Proposition~\ref{prop:case15} on page~\pageref{prop:case15}; for family (\ref{list.16}) it follows from Proposition~\ref{prop:case16} on page~\pageref{prop:case16}; for family (\ref{list.17}) from Proposition~\ref{prop:case17} on page~\pageref{prop:case17};   for family (\ref{list.18}) from Proposition~\ref{prop:case18} on page~\pageref{prop:case18}; and for family (\ref{list.21}) from Proposition~\ref{prop:case21} on page~\pageref{prop:case21}. Thus, all cases are covered. 

Each subsection of this section corresponds to one of the eight families given in the following list.  
\begin{listabc} \label{KnopLTypeA}
The $8$ families of saturated indecomposable spherical modules $(\overline{G},W)$ with $\overline{G}$ of type $\ssA$ in Knop's List are
\begin{enumerate}
\item $(\GL(m)\times \GL(n), \CC^m \otimes \CC^n)$ with $1 \leq m\leq n$; \label{list.1}
\item $(\GL(n), \Sym^2\CC^n)$ with $1 \leq n$; \label{list.2}
\item $(\GL(n), \bigwedge^2\CC^n)$ with $2 \leq n$;  \label{list.3}
\item $(\GL(n) \times \GGm, \bigwedge^2\CC^n \oplus \CC^n)$ with $4 \leq n$; \label{list.15}
\item $(\GL(n) \times \GGm, \bigwedge^2\CC^n \oplus (\CC^n)^*)$ with  $4 \leq n$; \label{list.16} 
\item $(\GL(m) \times \GL(n), (\CC^m \otimes \CC^n) \oplus \CC^n)$ with $1\leq m, 2\leq n$; \label{list.17} 
\item $(\GL(m) \times \GL(n), (\CC^m \otimes \CC^n) \oplus (\CC^n)^*)$ with $1\leq m, 2\leq n$; \label{list.18}
\item $(\GL(m) \times \SL(2) \times \GL(n), (\CC^m \otimes \CC^2)\oplus (\CC^2 \otimes \CC^n))$ with 
$2\leq m \leq n$. \label{list.21}
\end{enumerate}
\end{listabc}
\begin{remark}
The  indices $m$ and $n$ in family (\ref{list.17}) and family (\ref{list.18}) run through a larger set than that given in Knop's List. Knop communicated  the revised range of indices for these families to the second author. We remark that these cases do appear in the lists of \cite{leahy} and \cite{benson-ratcliff-mf}. 
\end{remark}
 
\begin{remark}
\begin{enumerate}[(i)]
\item Recall from Lemma~\ref{lem:camusbound} that for a given spherical module $W$ it is easy to compute $d_W$ from the rank of $\Lambda_W$. 
\item Recall that by Corollary~\ref{cor:apriori} it is enough to prove that $\dim T_{X_0}\tM^{G}_{\wm} \le d_W$ for every $(G,W)$ as in Theorem~\ref{thm:cbc} to establish the theorem. 
\item Note that by Lemma~\ref{lem:redsgr} a subgroup $G$ of $\overline{G}$ containing $\overline{G}'$ is automatically reductive. 
\end{enumerate}
\end{remark}

In each subsection, $(\overline{G},W)$ will denote a member of the family from List~\ref{KnopLTypeA} under consideration.
Here is some more  notation we will use for the rest of this section. Given a spherical module $(\overline{G},W)$ from Knop's List,
\begin{itemize}
\item[-] $E$ denotes the basis of the weight monoid $\Lambda^+_{(\overline{G},W^*)}$ of $W^*$ (the elements of $E$ are called the `basic weights' in Knop's List);
\item[-] $V= \oplus_{\lambda \in E} V(\lambda)$;
\item[-] $x_0 =  \sum_{\lambda \in E} v_{\lambda}$.
\end{itemize}
Except if stated otherwise, $G$ will denote  a connected subgroup of $\overline{G}$ containing $\overline{G}'$ such that $(G,W)$ is spherical. To lighten notation, we will use $G'$ for the derived subgroup $\overline{G}'$ of $\overline{G}$. This should not cause confusion since $(\overline{G},\overline{G}) = (G,G) = G'$. 
We will use $\overline{T}$ for a fixed maximal torus in $\overline{G}$ and put $T = \overline{T} \cap G$ and  $T' = \overline{T} \cap G'$. Then $T \inn G$ and $T' \inn G'$ are maximal tori. We will use $p\colon X(T) \onto X(T')$,  $q\colon X(\overline{T}) \onto X(T)$ and $r\colon X(\overline{T}) \onto X(T')$ for the restriction maps. Similarly, $\overline{B}$ is a fixed Borel subgroup of $\overline{G}$ containing $\overline{T}$ and we put $B = \overline{B} \cap G$ and $B' = \overline{B} \cap G'$. Then $B$ and $B$' are Borel subgroups of $G$ and $G'$, respectively. 
Note that the restriction of $p$ to $\Lambda_R$ is injective and we can, and will, identify the root lattices of $\overline{G}, G$ and $G'$. Moreover, our choice of Borel subgroups allows us to identify the sets of positive roots (which we denote $R^+$) and the sets of simple roots (which we denote $\Pi$) of $\overline{G}, G$ and $G'$.
Note also that since $Z(G') = Z(G) \cap T'$, we have that $T' \into T$ induces an isomorphism $T'/Z(G') \isom T/Z(G)$. We therefore can (and will) identify the adjoint torus of $\overline{G}, G$ and of $G'$ and we denote it $\Tad$.  We will use $\om, \om', \om''$ for weights of the first, second and third non-abelian factor of $G$, while $\varepsilon$ will refer to the character $\GGm \to \GGm, z \mapsto z$ of $\GGm$.

Recall our convention that by the $\Tad$-action on $V$ (and on $\Vg$, $\Vgg$, etc.) we mean the action given by $\alpha$ (see Definition~\ref{def:alpha}). The $\Tad$-action on $\tM_{\wm}$ refers to the action given by $\widehat{\psi}$, see page~\pageref{widehatpsi}.  
\begin{remark} \label{rem:VggpG}
We have the following isomorphism of $G$-modules (where $G$ acts on $V$ as a subgroup of $\overline{G}$):
\(
V \isom \oplus_{\lambda\in E} V(q(\lambda)).
\)
Using Lemma~\ref{lem:gx0} it follows that the $\Tad$-module $\Vggp$  only depends on $(\overline{G},W)$ (that is, it does not depend on the particular subgroup $G$).
\end{remark}

We will also use 
\begin{itemize}
\item[-] $\overline{\wm}$ for the weight monoid of $(\overline{G},W)$, that is $\overline{\wm} = \Lambda^+_{(\overline{G},W)}$;
\item[-] $\wm$ for the weight monoid of $(G,W)$; i.e.~$\wm = q(\overline{\wm})$;
\item[-] $\wgo$ for the weight group of $(\overline{G},W^*)$, that is $\overline{\wg} = \Lambda_{(\overline{G},W^*)} = \<E\>_{\ZZ} \inn X(\overline{T})$;
\item[-] $\wg$ for the weight group of $(G,W^*)$; i.e.~$\wg = q(\overline{\wg})$.
\end{itemize}
Note that the weight group of $(G',W^*)$ (which is not necessarily spherical) is $r(\overline{\wg}) = p(\wg)$ and that the weight monoid of $(G',W)$ is $r(\overline{\wm}) = p(\wm)$. 

\begin{remark}
 In proving Theorem~\ref{thm:cbc}  for families (\ref{list.16}), (\ref{list.17}) and (\ref{list.18}) we exclude certain $\Tad$-weight spaces in $\Vgg$ from belonging to the subspace $T_{X_0}\tM^G_{\wm}$. Comparing with the simple reflections of the little Weyl group of $W^*$ computed in Knop's List suggested which $\Tad$-weights we had to exclude. Logically however, that information from Knop's List plays no part in our proof.  In fact,
because  $\dim T_{X_0}\tM^G_{\wm}$ is minimal (by Theorem~\ref{thm:cbc}), the computations of the $\Tad$-weights in $T_{X_0}\tM^G_{\wm}$ we perform in this section confirm Knop's computations of the little Weyl group of the spherical modules under consideration. For the relationship between the $\Tad$-weights in $T_{X_0}\tM^G_{\wm}$ and the little Weyl group of $W^*$,  see Remarks \ref{rem:lwg} and \ref{rem:sigmawstar}.
\end{remark}

\subsection{The modules $(\GL(m)\times \GL(n), \CC^m \otimes \CC^n)$ with $1 \leq m\leq n$}
\label{subsec:case1}
Here 
\begin{align*}
&E = \{\om_1+\om_1', \om_2+\om_2', \ldots, \om_{m}+ \om_m'\};\\
&d_W = m-1.
\end{align*}

When $m<n$ the module $W$ is spherical for $G'=\SL(m) \times \SL(n)$, because $\<\om_m, \om'_n\>_{\ZZ} \cap \overline{\wg} = 0$, and its weight monoid $p(\wm)$ is $G'$-saturated. 
Corollary~\ref{cor:bcfred} therefore takes care of these cases. 
The only case that remains is when $m=n$. Then $W$ is not spherical for $G'$ because $\om_m +\om'_m \in E$.  Moreover, for the same reason, $\wm$ is not $G$-saturated for any intermediate group $G$ for which $W$ is spherical. We prove that in that case too $\Vggp$ has dimension $d_W$.
\begin{prop} \label{prop:case1}
Suppose $m=n$.
Then the $\Tad$-module $\Vggp$ is multiplicity-free and its weight set is 
\[\{\alpha_1 + \alpha_1', \alpha_2 + \alpha_2', \ldots, \alpha_{m-1} + \alpha_{m-1}'\}. \]
In particular, $\dim \Vggp = d_W$. Consequently, $\dim T_{X_0}\tM^G_{\wm} = d_W$. 
\end{prop}

\begin{proof}
First note that 
\[p(\wg) = \<\om_1 + \om'_1, \ldots, \om_{m-1}+\om'_{m-1}\>_{\ZZ} \inn X(T').\]
Suppose $v$ is a $\Tad$-eigenvector in $V$ of weight $\gamma$ so that  $[v]$ is a nonzero element of $\Vggp$. 
Then \begin{equation} \label{eq:case1gamma1}
\gamma \in p(\wg) \cap \Lambda_R
\end{equation}
by Lemma~\ref{lem:Tadweights}(\ref{weightinvar}). Clearly, 
$p(\wg) \cap \Lambda_R$ is the diagonal of $\Lambda_R$, that is, the group
\[
\<\alpha_1 + \alpha_1', \alpha_2 + \alpha_2', \ldots, \alpha_{m-1} + \alpha'_{m-1}\>_{\ZZ} \inn \Lambda_R.\]
Moreover, Lemma~\ref{lem:tangentcrit}(\ref{BBB}) implies that there exists a simple root $\delta$ of $G'$ so that 
\begin{equation} \label{eq:case1gamma2}
\gamma - \delta \text{ (which is the weight of $X_{\delta}v$) belongs to }R^+ \cup \{0\}. \end{equation}
Equations (\ref{eq:case1gamma1}) and (\ref{eq:case1gamma2}) imply that $\gamma = \alpha_i + \alpha_i'$ for some $i$ with $1 \le i \le m-1$. 

We next claim that the $\Tad$-eigenspace of weight $\alpha_i + \alpha_i'$ in $V$ is one dimensional for every $i$ with $1\le i\le m-1$. Indeed, the only $G'$-submodule of $V$ which contains an eigenvector of that weight is $V(\omega_i + \omega_i')$ and the eigenspace is the line spanned by 
\(X_{-\alpha_i}X_{-\alpha_i'} x_0 =  X_{-\alpha'_i}X_{-\alpha_i} x_0. \)
This finishes the proof. 
\end{proof}

\begin{remark} \label{rem:case1bcf} 
To relate $\tM^G_{\wm}$ to $\tM^{G'}_{p(\wm)}$, put $E'= E \setminus \{\om_m + \om'_m\}$, $V' = \oplus_{\lambda \in E'} V(\lambda)$ and
$x_0' = \sum_{\lambda \in E'} v_{\lambda}$. Then 
\begin{equation}
T_{X_0}\tM^{G'}_{p(\wm)} \isom (V'/\fg'\cdot x_0')^{G'_{x_0'}} \isom \Vggp \isom T_{X_0}\tM^G_{\wm} \label{eq:remcase1bcf}
\end{equation}
as $\Tad$-modules. Indeed, this is straightforward since $p(\wm)$ is $G'$-saturated, and the $G$-module $V(\omega_m+\omega'_m)$ is one-dimensional and therefore a subspace of $\fg\cdot x_0$.  Moreover, with Theorem~\ref{thm:bcf}, the second isomorphism in (\ref{eq:remcase1bcf}) 
provides a second argument for the multiplicity-freeness of $\Vggp$. 
\end{remark}

\begin{example}
We illustrate Proposition~\ref{prop:case1} for $m=n=3$ and $G=\overline{G}=\GL(3) \times \GL(3)$. Consider two copies of $\CC^3$, one with basis $e_1, e_2,e_3$, the other with basis $f_1,f_2,f_3$, and with the first (resp.~second) copy of $\GL(3)$ acting on the first (resp.~second) copy of $\CC^3$ by the defining representation.
Then we can take
\begin{align*}
V &= \CC^3 \otimes \CC^3 \oplus \wedge^2\CC^3 \otimes \wedge^2\CC^3 \oplus \wedge^3 \CC^3 \otimes \wedge^3\CC^3;\\
x_0 &= e_1 \otimes f_1 + e_1 \wedge e_2 \otimes f_1 \wedge f_2 + e_1\wedge e_2 \wedge e_3 \otimes f_1 \wedge f_2 \wedge f_3.
\end{align*}
Consequently,
\begin{multline*}
\fg \cdot x_0 = \< e_1 \otimes f_1, e_1 \wedge e_2 \otimes f_1 \wedge f_2, e_1\wedge e_2 \wedge e_3 \otimes f_1 \wedge f_2 \wedge f_3, \\
e_2 \otimes f_1, e_3 \otimes f_1 - e_2 \wedge e_3 \otimes f_1 \wedge f_2, e_1 \wedge e_3\otimes f_1\wedge f_2, \\
e_1 \otimes f_2, e_1 \otimes f_3 - e_1 \wedge e_2 \otimes f_2 \wedge f_3, e_1 \wedge e_2 \otimes f_1 \wedge f_3\>_{\CC},
\end{multline*}
\begin{equation*}
G'_{x_0} = \{(\begin{pmatrix}a&c_1&c_2 \\
0 & b & c_3 \\
0 & 0 & (ab)^{-1}
 \end{pmatrix},\begin{pmatrix}a^{-1}&c_4&c_5 \\
0 & b^{-1} & c_6 \\
0 & 0 & ab
 \end{pmatrix})\bigm| a,b \in \CC^{\times}, c_i \in \CC\}
\end{equation*}
and
\(
\Vggp = \<[e_2 \otimes f_2], [e_1 \wedge e_3 \otimes f_1 \wedge f_3]\>_{\CC}.
\)
\end{example}

\subsection{The modules $(\GL(n), \Sym^2\CC^n)$ with $1 \leq n$} \label{subsec:case2}
Here 
\begin{align*}
&E =  \{2\omega_1, 2\omega_2, \ldots, 2\omega_n\};\\
&d_W = n-1.
\end{align*}
Because $2 \om_n \in E$, there is no group $G$ with $G' \inn G \subsetneq \overline{G}$ for which $(G,W)$ is spherical. Hence we assume that $G = \overline{G} = \GL(n)$.  For the same reason, $\wm=\wmo$ is not $G$-saturated. 

\begin{prop}\label{prop:case2}
The $\Tad$-module $\Vgg$ is multiplicity-free and has $\Tad$-weight set
\[\{2\alpha_1, 2\alpha_2, \ldots, 2\alpha_{n-1}\}.\]
In particular, its dimension is $d_W$. Consequently, $\dim T_{X_0}\tM^G_{\wm} = d_W$. 
\end{prop}
\begin{proof}
This proof is very similar to that of Proposition~\ref{prop:case1}. 
Suppose $v$ is a $\Tad$-eigenvector in $V$ of weight $\gamma$ so that  $[v]$ is a nonzero element of $\Vgg$. 
Then \begin{equation} \label{eq:case2gamma1}
\gamma \in\wg \cap \Lambda_R
\end{equation}
by Lemma~\ref{lem:Tadweights}(\ref{weightinvar}). A straightforward calculation yields that
\(\wg \cap \Lambda_R = 2\Lambda_R.\)
Lemma~\ref{lem:tangentcrit}(\ref{BBB}) implies that there exists a simple root $\delta$ of $G$ so that 
\begin{equation} \label{eq:case2gamma2}
\gamma - \delta \text{ (which is the weight of $X_{\delta}v$) belongs to }R^+ \cup \{0\}. \end{equation}
Equations (\ref{eq:case2gamma1}) and (\ref{eq:case2gamma2}) imply that
$\gamma = 2\alpha_i'$ for some $i\in \{1,\ldots,n-1\}$. 

We next claim that the $\Tad$-eigenspace of weight $2\alpha_i$ in $V$ is one dimensional for every $i$ with $1\le i\le n-1$. Indeed, the only $G$-submodule of $V$ which contains an eigenvector of that weight is $V(2\omega_i)$ and the eigenspace is the line spanned by 
\(X_{-\alpha_i}X_{-\alpha_i}x_0.\)
This finishes the proof. 
\end{proof}

\begin{remark} \label{rem:case2bcf}
As in Remark~\ref{rem:case1bcf}, we relate $\tM^G_{\wm}$ to $\tM^{G'}_{p(\wm)}$. 
Put $E' = E\setminus\{2\om_n\}$ and define $V'$ and $x_0'$ as in Remark~\ref{rem:case1bcf}. Then \[T_{X_0}\tM^{G'}_{p(\wm)} \isom (V/\fg'\cdot x_0')^{G'_{x_0'}}\isom  \Vgg \isom T_{X_0}\tM^G_{\wm}.\]
The first isomorphism holds because $p(\wm)$ is $G'$-saturated. Here is an argument for the second. 
One can check that when $n$ is odd, $p({\wg}) \cap \Lambda_R = 2 \Lambda_R$, but when $n$ is even, $p({\wg}) \cap \Lambda_R \supsetneq 2\Lambda_R$ (e.g. for $n=4$, $\alpha_1+ \alpha_3, \alpha_1 + 2\alpha_2 + \alpha_3 \in p({\wg}) \cap \Lambda_R$). So, for $n$ odd, the argument in the proof of Proposition~\ref{prop:case2} shows that $(V'/\fg'\cdot x_0')^{G'_{x'_0}}$ is a multiplicity-free $\Tad$-module with $\Tad$-weight set $\{2\alpha_1, \ldots, 2\alpha_{n-1}\}$. Of course, the multiplicity-freeness also follows from Theorem~\ref{thm:bcf}. For $n$ even, an argument ruling out possible $\Tad$-weights as listed in Remark~\ref{rem:srA} ---like, e.g., in the proof of Lemma~\ref{lem:consecfund}--- again shows that the $\Tad$-weights in  $(V'/\fg'\cdot x_0')^{G'_{x'_0}}$ belong to the set  $\{2\alpha_1, \ldots, 2\alpha_{n-1}\}$ and it follows as before that $(V'/\fg'\cdot x_0')^{G'_{x'_0}}$ is a multiplicity-free $\Tad$-module with $\Tad$-weight set $\{2\alpha_1, \ldots, 2\alpha_{n-1}\}$. 
\end{remark}

\subsection{The modules $(\GL(n), \bigwedge^2\CC^n)$ with $2 \leq n$} \label{subsec:case3}
Here 
\begin{align*}
&E =  \{\om_{2i}\colon 1\le i\le \lfloor\frac{n}{2}\rfloor\};\\
&d_W = \lfloor\frac{n}{2}\rfloor -1.
\end{align*}

When $n$ is odd this module is spherical for $G'=\SL(n)$, because $\<\om_n\>_{\ZZ} \cap \overline{\wg}=0$, and $p(\wm)$ is $G'$-saturated. Corollary~\ref{cor:bcfred} therefore takes care of these cases. 

On the other hand, when $n$ is even, $\om _n \in E$, and so there is no group $G$ with $G' \inn G \subsetneq \overline{G}$ for which $(G,W)$ is spherical. Moreover, for the same reason,  $\wm=\wmo$ is not $G$-saturated. As it needs no extra work compared to $\Vgg$, we prove that $\Vggp$ has dimension $d_W$. 

\begin{prop} \label{prop:case3}
Suppose $n\ge 2$ is even. 
Then the $\Tad$-module $\Vggp$ is multiplicity-free and has $\Tad$-weight set
\[
\{\alpha_{i}+ 2\alpha_{i+1} + \alpha_{i+2} \colon 1\le i\le n-3 \text{ and $i$ is odd}\}.\]
In particular, $\dim \Vggp = \frac{n}{2}-1 = d_W$. Consequently, $\dim T_{X_0}\tM^G_{\wm} = d_W$. 
\end{prop}
\begin{proof}
Put $E'=E \setminus\{\om_n\}$ and define $V'$ and $x_0'$ as in Remark~\ref{rem:case1bcf}.
Then $\fg \cdot x_0 = \fg' \cdot x_0' \oplus V(\omega_n)$ and $V= V' \oplus V(\omega_n)$ and so
$\Vg \isom V'/\fg'\cdot x_0'$ as $G'_{x_0} \rtimes \Tad$-modules. 
Since the freely generated submonoid 
\(p(\<E\>_{\NN}) = \<\om_2, \om_4, \ldots, \om_{n-2}\>_{\NN} \inn X(T')\)
of dominant weights of $G'$  is $G'$-saturated, Theorem~\ref{thm:bcf} tells us that 
\(\Vggp \isom (V'/\fg'\cdot x_0')^{G'_{x_0}}\)
is a multiplicity-free $\Tad$-module.

Now, suppose $v$ is a $\Tad$-eigenvector in $V$ of weight $\gamma$ so that  $[v]$ is a nonzero element of $\Vggp$. 
Then \begin{equation} \label{eq:case3gamma1}
\gamma \in p(\wg) \cap \Lambda_R
\end{equation}
by Lemma~\ref{lem:Tadweights}(\ref{weightinvar}).
We next claim that
\[p(\wg) \cap \Lambda_R = \<\alpha_{i}+ 2\alpha_{i+1} + \alpha_{i+2} \colon 1\le i\le n-3 \text{ and $i$ is odd}\>_{\ZZ}.\]
Indeed, the inclusion '$\supseteq$' is clear. For the other inclusion, let $\beta=\sum_{i=1}^{n-1}{u_i} \alpha_i$, with all $u_i\in \ZZ$,  be an element of $\Lambda_R$. Then
$\beta = \sum_{i=1}^{n-1}u_i(-\om_{i-1}+2\om_i -\om_{i+1})$, where $\om_0 = \om_n = 0$. 
Then, after rearranging terms, we see that $\beta \in p(\wg)$ if and
only if for all even $i \in \{0, \ldots, n-2\}$, we have
$2u_{i+1} = u_{i}+u_{i+2}$,
where $u_0 = u_n=0$. Consequently, when $i$ is even, so is $u_i$ and
$$\beta = \sum_{2\le i \le {n-2}, i\text{ even}} \frac{u_{i}}{2}(\alpha_{i-1}+2\alpha_{i} + \alpha_{i+1}),$$
which proves the claim. 

Now, Lemma~\ref{lem:tangentcrit}(\ref{BBB}) implies that there exists a simple root $\delta$ of $G'$ so that 
\begin{equation} \label{eq:case3gamma2}
\gamma - \delta \text{ (which is the weight of $X_{\delta}v$) belongs to }R^+ \cup \{0\}. \end{equation}
Equations (\ref{eq:case3gamma1}) and (\ref{eq:case3gamma2}) imply that $\gamma = \alpha_{i}+ 2\alpha_{i+1} + \alpha_{i+2}$ for some odd $i$ with $1\le i\le n-3$. We have proved the proposition. \end{proof}

\begin{remark}\label{rem:case3bcf}
The proof of Proposition~\ref{prop:case3} implies that when $n$ is even $$T_{X_0}\tM^{G'}_{p(\wm)} \isom (V'/\fg'\cdot x_0')^{G'_{x_0'}} \isom \Vggp \isom T_{X_0}\tM^G_{\wm}$$ as $\Tad$-modules. As said before, when $n$ is odd, $T_{X_0} \tM^{G'}_{p(\wm)} \isom T_{X_0} \tM^{G}_{\wm}$ by Corollary~\ref{cor:bcfred}.
\end{remark}

\subsection{The modules $(\GL(n) \times \GGm, \bigwedge^2\CC^n \oplus \CC^n)$ with $4 \leq n$}
\label{subsec:case15}
We now have
\begin{align*}
&E =  \{\om_{2i-1}+\epsilon\colon 1\le i\le \Bigl\lceil\frac{n}{2}\Bigr\rceil\} \cup \{\om_{2i}\colon 1\le i\le \Bigl\lfloor\frac{n}{2}\Bigr\rfloor\};\\
&d_W = n-2.
\end{align*}
The modules $W$ are not spherical for $G'$ because $\wg \cap \<\om_n, \epsilon\>_{\ZZ} \neq 0$.  
Moreover, for the same reason, $\wm$ is not $G$-saturated for any intermediate group $G$ for which $W$ is spherical. 

\begin{prop} \label{prop:case15}
The $\Tad$-module $\Vggp$ is multiplicity-free with $\Tad$-weight set
\[\{\alpha_i+\alpha_{i+1} \colon 1\le i\le n-2\}.\]
In particular, $\dim \Vggp = d_W$.  Consequently, $\dim T_{X_0}\tM^G_{\wm} = d_W$. 
\end{prop}
\begin{proof}
Note that
\(V= V' \oplus \CC z\),
where $
z := 
v_{\omega_n+\epsilon} \in V$ if $n$ is odd and $z:=
v_{\omega_n} \in V$ if $n$ is even,
and that $V'\isom V(\omega_1)\oplus V(\omega_2) \oplus \ldots \oplus V(\omega_{n-1})$ as a $G'$-module.
Since $\fg\cdot x_0 = \fg' \cdot x_0' \oplus \CC z $, where 
\(x_0' = x_0-z\),
it follows that \(\Vg \isom V'/\fg'\cdot x_0' \) as  $G'_{x_0} \rtimes \Tad$-modules.
Because $G'_{x_0} = G'_{x_0'}$ we have $\Vggp \isom (V'/\fg'\cdot x_0' )^{G'_{x_0'}}$, and by \cite[Corollary 3.9 and Theorem 3.10]{bravi&cupit} we know that $(V'/\fg'\cdot x_0' )^{G'_{x_0'}}$ is a multiplicity-free $\Tad$-module whose $\Tad$-weight set is 
\(\{\alpha_i+\alpha_{i+1} \colon 1\le i\le n-2\}.\)
\end{proof}

\begin{remark}\label{rem:case15bcf}
The proof of Proposition~\ref{prop:case15} implies that $$T_{X_0}\tM^{G'}_{p(\wm)} \isom (V'/\fg'\cdot x_0')^{G'_{x_0'}} \isom \Vggp \isom T_{X_0}\tM^G_{\wm}$$ as $\Tad$-modules. 
\end{remark}

\subsection{The modules $(\GL(n) \times \GGm, \bigwedge^2\CC^n \oplus (\CC^n)^*)$ with  $4 \leq n$}
\label{subsec:case16}
For these modules we have
\begin{align*}
&E = \{\lambda_i \colon 1\le i \le n-2, i \text{ odd}\} \cup \{\lambda_j\colon 1\le j \le n,  j \text{ even}\} \cup \{\mu\};\\
&d_W = n-2,
\end{align*}
where
\begin{align*}
\lambda_i &:= \omega_i + \epsilon &&\text{for $1\le i \le n-2$ with $i$ odd}; \\
\lambda_j &:= \omega_j &&\text{for $1\le j \le n$ with $j$ even}; \\
\mu &:= \omega_{n-1} - \omega_{n} + \epsilon.
\end{align*}
These modules are not spherical for $G'$  because $\wg \cap \<\om_n, \epsilon\>_{\ZZ} \neq 0$.  Moreover, for the same reason, $\wm$ is not $G$-saturated for any intermediate group $G$ for which $W$ is spherical.

\begin{prop} \label{prop:case16} Suppose $n\ge 4$. 
The $\Tad$-module $T_{X_0}\tM^G_{\wm}$ is multiplicity-free and has $\Tad$-weight set
\begin{align}
&\{\alpha_i+\alpha_{i+1} \colon 1\le i\le n-2\} &&\text{when $n$ is even};\\
& \{\alpha_i+\alpha_{i+1} \colon 1\le i\le n-3\} \cup \{\alpha_{n-1}\}  &&\text{when $n$ is odd}. \label{eq:case16nodd}
\end{align}
In particular, $\dim T_{X_0}\tM^G_{\wm} = d_W$. 
\end{prop}
\begin{proof}
When $n$ is even, we are done by Proposition~\ref{prop:case16neven}, because $\dim \Vggp = d_W$. 
On the other hand, when $n$ is odd, let $J$ be the set (\ref{eq:case16nodd}) and put $\beta = \alpha_{n-2}+\alpha_{n-1}$. We prove in Proposition~\ref{prop:case16nodd} that $\Vggp $ is multiplicity-free, and that its $\Tad$-weight set is a subset of $J \cup \{\beta\}$ and contains $\beta$.  In particular, $\dim \Vggp \le d_W + 1$.  When $\beta$ is not a $\Tad$-weight of $\Vgg$, it follows that $\dim \Vgg \le d_W$ and we are done. We show in Proposition~\ref{prop:exclbeta16} that even when $\beta$ is a $\Tad$-weight of $\Vgg$, the corresponding section  in $H^0(G\cdot x_0, \shN_{X_0})^G$ does not extend to $X_0$. Consequently $\dim T_{X_0}\tM^G_{\wm} \le d_W$ and the proposition follows. 
\end{proof}

 \begin{prop} \label{prop:case16neven}
Suppose $n\ge 4$ is even. Then $\Vggp$ is a multiplicity-free $\Tad$-module with $\Tad$-weight set
\[\{\alpha_i+\alpha_{i+1} \colon 1\le i\le n-2\}.\]
In particular, $\dim \Vggp = d_W$. 
\end{prop}
\begin{proof}
Consider the $G$-submodule $V'$  of $V$ defined as
\[V' := V(\lambda_1) \oplus V(\lambda_2) \oplus \ldots \oplus V(\lambda_{n-2}) \oplus V(\mu).\]
Note that as a $G'$-module, $V'$ is the direct sum of the fundamental representations. Furthermore, $V = V' \oplus V(\lambda_n)$ and $V(\lambda_n)$ is one-dimensional. The rest of the proof is identical to that of Proposition~\ref{prop:case15}. 
\end{proof}

When $n$ is odd, determining $\Vggp$ requires a little more care, because $V(\lambda_{n-1}) \isom V(\mu)$ as $G'$-modules.  
\begin{prop}\label{prop:case16nodd}
Suppose $n\ge 5$ is odd. Then $\Vggp$ is a multiplicity-free $\Tad$-module. Its $\Tad$-weight set is a subset of
\begin{equation} \label{eq:weights16nodd}
\{\alpha_i+\alpha_{i+1} \colon 1\le i\le n-2\} \cup \{\alpha_{n-1}\}. 
\end{equation}
The weight $\beta=\alpha_{n-2} + \alpha_{n-1}$ occurs and its eigenspace  is spanned by the vector 
\[
[X_{-\beta} v_{\lambda_{n-2}}] = - [X_{-\beta} (v_{\lambda_{n-1}} + v_{\mu})].
\]
\end{prop}
\begin{proof}
Let $V'$ be the following $G'$-submodule of $V$:
\[V' := V(\lambda_1) \oplus V(\lambda_2) \oplus \ldots \oplus V(\lambda_{n-2}) \oplus V_{n-1}\]
where
\(V_{n-1} := \<G'\cdot (v_{\lambda_{n-1}}+v_{\mu})\>_{\CC}.\)
Then
\begin{align}
V &= V' \oplus Z_{n-1} \label{eq:VVprZ},\\
\intertext{where $Z_{n-1} := \<G'\cdot (v_{\lambda_{n-1}}-v_{\mu})\>_{\CC}$, and }
\fg \cdot x_0 &= \fg' \cdot x_0 \oplus \CC(v_{\lambda_{n-1}}-v_{\mu}) \label{eq:gx0gprx0}
\end{align}
Moreover, we have an inclusion of $G'_{x_0} \rtimes \Tad$-modules
\[\fg \cdot x_0 \inn V' \oplus \CC(v_{\lambda_{n-1}}-v_{\mu}) \inn V\]
and so an exact sequence
\[
0 \longrightarrow \frac{V' \oplus \CC(v_{\lambda_{n-1}}-v_{\mu})}{\fg \cdot x_0} \longrightarrow \Vg \longrightarrow \frac{V}{V' \oplus \CC(v_{\lambda_{n-1}}-v_{\mu})} \longrightarrow 0.
\]
Taking $G'_{x_0}$-invariants, we obtain an exact sequence of $\Tad$-modules
\begin{equation} \label{eq:seqcase16nodd}
0 \longrightarrow \Bigl(\frac{V' \oplus \CC(v_{\lambda_{n-1}}-v_{\mu})}{\fg \cdot x_0}\Bigr)^{G'_{x_0}} \longrightarrow \Vggp \longrightarrow \Bigl(\frac{V}{V' \oplus \CC(v_{\lambda_{n-1}}-v_{\mu})}\Bigr)^{G'_{x_0}} 
\end{equation}
From (\ref{eq:gx0gprx0}) we have that
\[\frac{V' \oplus \CC(v_{\lambda_{n-1}}-v_{\mu})}{\fg \cdot x_0} \isom \frac{V'}{\fg'\cdot x_0}\]
as $G'_{x_0} \rtimes \Tad$-modules. 
Clearly, as a $G'$-module, $V'$ is the direct sum of the fundamental representations, and $\fg'\cdot x_0$ is the tangent space to the orbit of the sum of the highest weight vectors in $V'$. Therefore \cite[Cor 3.9 and Thm 3.10] {bravi&cupit} tells us that $\Bigl(\frac{V' \oplus \CC(v_{\lambda_{n-1}}-v_{\mu})}{\fg \cdot x_0}\Bigr)^{G'_{x_0}}$ is a multiplicity-free $\Tad$-module with weight set
\(\{\alpha_1 + \alpha_2, \alpha_2 + \alpha_3, \ldots, \alpha_{n-2}+
\alpha_{n-1} \}.\) 
On the other hand, (\ref{eq:VVprZ}) tells us that 
\[\frac{V}{V' \oplus \CC(v_{\lambda_{n-1}}-v_{\mu})} \isom \frac{Z_{n-1}}{\CC(v_{\lambda_{n-1}}-v_{\mu})}.\]
Furthermore, we claim that \begin{equation} \label{eq:case16znmo}
\Bigl(\frac{Z_{n-1}}{\CC(v_{\lambda_{n-1}}-v_{\mu})}\Bigr)^{G'_{x_0}} = 
\CC[X_{-\alpha_{n-1}}(v_{\lambda_{n-1}}-v_{\mu})].\end{equation} Indeed, if $[v]$ is a nonzero $\Tad$-eigenvector in  
\(\Bigl(\frac{Z_{n-1}}{\CC(v_{\lambda_{n-1}}-v_{\mu})}\Bigr)^{G'_{x_0}}\)
then there exists a simple root $\alpha$ so that $X_{\alpha}v \neq 0$
(because $v$ is not a highest weight vector) and $X_{\alpha}v \in
\CC(v_{\lambda_{n-1}}-v_{\mu}) = Z_{n-1}^U$. Hence $X_{\alpha}v$ has
trivial $\Tad$-weight and therefore $v$ has weight $\alpha$. Since
$Z_{n-1} \isom V(\omega_{n-1})$, this implies that
$\alpha=\alpha_{n-1}$ and the claim (\ref{eq:case16znmo}).  

From the sequence (\ref{eq:seqcase16nodd}) and the description of its first and third term above, we know that the $\Tad$-module $\Vggp$ is multiplicity-free, and that its $\Tad$-weight  set is a subset of (\ref{eq:weights16nodd}) and contains all its weights except possibly $\alpha_{n-1}$. 
The assertion about the eigenspace of weight $\beta$ merely needs a straightforward verification. 
\end{proof}

\begin{remark} \label{rem:vggp16nodd}
In fact, the $\Tad$-weight set of $\Vggp$ in Proposition~\ref{prop:case16nodd} \emph{equals} the set  (\ref{eq:weights16nodd}). Indeed, 
the proof of Proposition~\ref{prop:case16nodd} shows that the $\Tad$-weight set of $\Vggp$ contains all elements of (\ref{eq:weights16nodd}) except possibly $\alpha_{n-1}$. 
Moreover, $\alpha_{n-1}$ belongs to the $\Tad$-weight set because  $[X_{-\alpha_{n-1}} v_{\lambda_{n-1}}] = - [X_{-\alpha_{n-1}} v_{\mu}] \in \Vggp$ by a straightforward verification (or because $s_{n-1\,n}$ is a `simple reflection' in Knop's List; or \emph{a posteriori} by Proposition~\ref{prop:case16} because $T_{X_0}\tM^G_{\wm} \inn \Vggp$).
\end{remark}

The next lemma determines for which groups $G$ the weight $\beta = \alpha_{n-2}+\alpha_{n-1}$ is a $\Tad$-weight of $\Vgg$. 
\begin{lemma} \label{lem:betainlam16}
Suppose $n\ge 5$ is odd and let $\beta$ be defined as in Proposition~\ref{prop:case16nodd}. Then the following are equivalent  (recall that, by assumption, $(G,W)$ is spherical)
\begin{enumerate}
\item  $\beta$ is a $\Tad$-weight of $\Vgg$; \label{betainlam16one}
\item  $\beta \in \wg$;  \label{betainlam16two}
\item $\ft = \ker [(a+1)\omega_n - (a-1)\epsilon] \inn \Lie(\overline{T})$ for some integer $a$. \label{betainlam16three}
\end{enumerate}
For every integer $a$ we have the following equality in $X(\overline{T})$:
\begin{equation} \label{eq:beta16up}
\beta +[(a+1)\omega_n - (a-1)\epsilon] = 
\lambda_{n-2} + (a+1)\lambda_{n-1} - a\mu - \lambda_{n-3}.
\end{equation}
Consequently, if $\ft = \ker [(a+1)\omega_n - (a-1)\epsilon]$ for some integer $a$, restricting (\ref{eq:beta16up}) to $T$ yields the following equality in $\wg$:
\begin{equation} \label{eq:beta16}
\beta = \lambda_{n-2} + (a+1)\lambda_{n-1} - a\mu - \lambda_{n-3}.\end{equation}
\end{lemma}
\begin{remark}
We use $\ft = \Lie(T)$ in Lemma~\ref{lem:betainlam16} instead of $T$ because $\ker [(a+1)\omega_n - (a-1)\epsilon] \inn \overline{T}$ is not necessarily connected (for example, it is disconnected when $a=1$). 
\end{remark}
\begin{proof}
Since $\beta$ is a $\Tad$-weight of $\Vggp$ by Proposition~\ref{prop:case16neven}, the fact that (\ref{betainlam16one}) and (\ref{betainlam16two}) are equivalent follows from Lemma~\ref{lem:Tadweights}(\ref{liealgandtorinvs}). 
We now prove that (\ref{betainlam16two}) and (\ref{betainlam16three}) are equivalent. 
Recall that $r:X(\overline{T}) \onto X(T')$ and $q: X(\overline{T}) \onto X(T)$ are the restriction maps. Recall further that $\wg= q(\overline{\wg})$ and note that $\ker q \inn \ker r = \<\omega_n, \epsilon\>_{\ZZ}.$  Now $\beta = -\omega_{n-3} + \omega_{n-2} + \omega_{n-1} - \omega_{n} \in X(\overline{T})$. So $q(\beta) \in \wg$ if and only if $q(\beta + \lambda_{n-3} - \lambda_{n-2} -\lambda_{n-1}) = q(-\omega_n -\epsilon) \in \wg$. 
In other words, $q(\beta) \in \wg$ if and only if there exists $\gamma \in \overline{\wg}$ so that $q(-\omega_n-\epsilon) = q(\gamma)$, that is, so that $\gamma +\omega_n+\epsilon \in \ker q$. Since $\omega_n + \epsilon \in \ker r$ this is equivalent to the existence of $\gamma \in \overline{\wg} \cap \ker r$ so that $q(\gamma+\omega_n + \epsilon)=0$. 

Next we claim that $\overline{\wg}\cap \ker r = \<\omega_n - \epsilon\>$. The inclusion `$\supseteq$' is immediate: $\omega_n - \epsilon = \lambda_{n-1}-\mu$. The other inclusion follows from a direct calculation, or from Knop's List which tells us that\footnote{In the notation of Knop's List,  $\fa^* \cap \fz^*$ is used for $\<\overline{\wg}\>_{\CC} \cap \<\ker r\>_{\CC}.$} $\<\overline{\wg}\>_{\CC} \cap \<\ker r\>_{\CC} = \<\omega_n - \epsilon\>_{\CC}$ as subspaces of $\Lie(\overline{T})^*$ .

Consequently, $q(\beta) \in \wg$ if and only if there exists an integer $a$ so that 
\[a(\omega_n - \epsilon) + \omega_n + \epsilon = (a+1)\omega_n -(a-1)\epsilon \]
belongs to $\ker q$. Equivalently, $T \inn \ker[(a+1)\omega_n -(a-1)\epsilon]$, or (since $T$ is connected)
\begin{equation}
\ft \inn \ker[(a+1)\omega_n -(a-1)\epsilon]. \label{eq:case16betaint}
\end{equation}
On the other hand, \cite[Theorem 5.1]{knop-rmks} tells us that $W$ is spherical as a $G$-module if and only if 
\begin{equation}\ft \not\inn \ker (\omega_n - \epsilon). \label{eq:case16laspher}
\end{equation} 
Because $\ft' = \<\omega_n, \epsilon\>_{\CC}^{\perp}$ is of codimension $2$ in $\Lie(\overline{T})$,  and  for every integer $a$, the two vectors $(a+1)\omega_n -(a-1)\epsilon$ and $\omega_n-\epsilon$ in $\Lie(\overline{T})^*$ are linearly independent, $\ft$ satisfies (\ref{eq:case16laspher}) and (\ref{eq:case16betaint}) for some integer $a$ if and only if $
\ft = \ker [(a+1)\omega_n - (a-1)\epsilon]$. The equivalence of (\ref{betainlam16two}) and (\ref{betainlam16three}) follows.
The straightforward verification of (\ref{eq:beta16up}) is left to the reader.
\end{proof}

\begin{prop} \label{prop:exclbeta16}
Suppose $n\ge 5$ is odd and let $\beta$ be defined as in Proposition~\ref{prop:case16nodd}. 
Let $a$ be an integer and suppose that the maximal torus $T$ of $G$ satisfies
$\ft = \ker [(a+1)\omega_n - (a-1)\epsilon]$. 
Then the section  $s \in H^0(G \cdot x_0, \shN_{X_0})^G$ defined\footnote{The fact that this formula defines a section of $H^0(G \cdot x_0, \shN_{X_0})^G \isom \Vgg$ uses Lemma~\ref{lem:betainlam16}.} by 
\[s(x_0) = [X_{-\beta} v_{\lambda_{n-2}}] = -[X_{-\beta}(v_{\lambda_{n-1}}+v_{\mu})] \in \Vgg \]
does not extend to $X_0$. 
\end{prop}
\begin{proof}
We consider two cases: $a<0$ and $a\ge0$. 

(i) If $a<0$, we apply Proposition~\ref{prop:excl} with $\lambda = \mu$
and \(v = X_{-\beta} v_{\lambda_{n-2}}\). We check the four conditions: (\ESo) follows from equation (\ref{eq:beta16}); (\ESt) is clear from the description of $v$ given above; (\ESth) follows from the equalities $\mu= \omega_{n-1} - \omega_{n} + \epsilon$ and $\<\lambda_{n-1}, \alpha_{n-1}^{\vee}\> = 1$; for (\ESf) take $\delta = \lambda_{n-1}$.

(ii) If $a\ge 0$, we apply Proposition~\ref{prop:excl} with $\lambda = \lambda_{n-1}$
and the same $v$. We check the four conditions: (\ESo) follows from equation (\ref{eq:beta16});
(\ESt) is clear from the description of $v$ given above;
(\ESth) follows from the equalities $\lambda_{n-1}= \omega_{n-1}$ and $\<\mu, \alpha_{n-1}^{\vee}\> = 1$;
for (\ESf) take $\delta = \mu$.
\end{proof}

\begin{remark} \label{rem:vgg16}
We now obtain a description of the module $\Tad$-module $\Vgg$. For $n$ even, this is done in Proposition~\ref{prop:case16neven} since $\Vgg = \Vggp$ (because $\Vggp$ has dimension $d_W$). 
For $n$ odd, the $\Tad$-module $\Vggp$ is described in Remark~\ref{rem:vggp16nodd}. Call its $\Tad$-weight set $F$.  
Now $\Vggo$ is the $\Tad$-submodule of $\Vggp$ with $\Tad$-weight set $F \setminus \{\beta\}$, where $\beta  = \alpha_{n-2} + \alpha_{n-1}$. Indeed, $\beta$ does not belong to the $\Tad$-weight set of $\Vggo$ by Lemma~\ref{lem:betainlam16}, whereas $F\setminus\{\beta\}$ does, as one can prove in at least three ways:  (i) direct verification that $F \setminus \{\beta\} \inn \overline{\wg}$; or (ii) use Knop's information about the little Weyl group of $W^*$ (see Remark~\ref{rem:lwg}); or (iii) note \emph{a posteriori} that by Proposition~\ref{prop:case16} the subspace $T_{X_0}\tM^{\overline{G}}_{\overline{\wm}}$ of $\Vggo$ has dimension $d_W=|F|-1$.
Since $\Vggo \inn \Vgg \inn \Vggp$,  the $\Tad$-module $\Vgg$ is completely determined by our characterization in Lemma~\ref{lem:betainlam16} of those intermediate groups $G$ for which $\beta$ is  a $\Tad$-weight of $\Vgg$. 
\end{remark}

\subsection{The modules $(\GL(m) \times \GL(n), (\CC^m \otimes \CC^n) \oplus \CC^n)$ with $1\leq m, 2\leq n$} \label{subsec:case17}
We begin with some notation. Put
\begin{align*}
K &= \min(m+1, n) \\
L &= \min(m,n).
\end{align*}
Note that $L=K-1$ (when $m+1 \le n$) or $L=K$ (otherwise).
We will also use the following notation:
\begin{align}
\lambda_i &= \omega_{i-1} + \omega_i'  &&i\in\{1, \ldots, K\} \text{ (with $\omega_0 = 0$)}
\label{eq:case17wts1}\\
\lambda'_i &= \omega_i + \omega_i' &&i\in\{1,\ldots, L\} \label{eq:case17wts2}
\end{align}
For the modules under consideration,
\begin{align*}
&E = \{\lambda_i \colon 1\le i \le K\} \cup \{\lambda'_i \colon 1\le i \le L\};\\
&d_W = K+L -2 = \min(2m+1, 2n) -2. 
\end{align*}
These modules are not spherical for $G'$  because $\wg \cap \<\om_m, \om'_n\>_{\ZZ} \neq 0$.  Moreover, for the same reason, $\wm$ is not $G$-saturated for any intermediate group $G$ for which $W$ is spherical.

In this section we will prove the following proposition.
\begin{prop} \label{prop:case17}
The $\Tad$-module $T_{X_0}\tM^G_{\wm}$ is multiplicity-free and has $\Tad$-weight set
\begin{equation} \label{eq:Tadweightscase17}
\{\alpha_i \colon 1\le i \le L-1\} \cup \{\alpha'_j \colon 1\le j \le K-1\}. \end{equation}
In particular, $\dim T_{X_0}\tM^G_{\wm} = d_W$.  
\end{prop}
\begin{proof}
Call $F$ the set (\ref{eq:Tadweightscase17}) and let the sets $J_0$ and $J_1$ be defined  as in Proposition~\ref{prop:vgg}. Put
\begin{equation*}
J:= \begin{cases}
J_1 &\text{if $n=m-1$ and $\alpha_{m-2}+\alpha_{m-1} \in \wg$}; \\
J_1 &\text{if $m=n-2$ and $\alpha'_{n-2} + \alpha'_{n-1} \in \wg$};\\
J_0 &\text{otherwise}.
\end{cases}
\end{equation*}
Corollary~\ref{cor:vgg17} says that $\Vgg$ is a multiplicity-free $\Tad$-module, that its $\Tad$-weight set $D$ contains $J$ and that $D \inn J \cup F$. Lemmas~\ref{lem:exclbetar}, \ref{lem:exclbetaprj}, \ref{lem:exclbetam1} and \ref{lem:exclbetan2} prove that the sections of $H^{0}(G\cdot x_0, \shN_{X_0})^G \isom \Vgg$ corresponding to the $\Tad$-weights in $J$ do not extend to $X_0$. This implies that the $\Tad$-weight set of $T_{X_0}\tM^G_{\wm}$ is a subset of $F$. Equality follows, as always, from Corollary~\ref{cor:apriori}.  
\end{proof}

\begin{remark} \label{rem:propcase17}
As the Proposition~\ref{prop:case17} and Corollary~\ref{cor:vgg17} show, except for a few small values of $m$ and $n$, the inclusion $T_{X_0}\tM^G_{\wm} \inn (\Vg)^{\overline{G}_{x_0}}$ is strict. Moreover, for $n=m-1$ and for $m=n-2$ there exist groups $G \inn \overline{G}$, containing $\overline{G}'$, for which $W$ is spherical and for which the inclusion $(\Vg)^{\overline{G}_{x_0}} \inn \Vgg$ is strict (see  Corollary~\ref{cor:vgg17}, Lemmas~\ref{lem:betam1inlambda} and \ref{lem:betan1inlambda}).
\end{remark}

\begin{prop} \label{prop:vgg}
Suppose $m \ge 1, n\ge 2$. 
Let $F$ be the set (\ref{eq:Tadweightscase17}) and put 
\begin{align*}
J_0&:= \{\alpha_{r-1} + \alpha_{r} \colon 2\le r \le L-1\} \cup \{\alpha'_{s-1} + \alpha'_{s} \colon 2\le s \le K-1\}; \\
J_1&:= \begin{cases} J_0 &\text{if $n\neq m-1$ and $m\neq n-2$}; \\
J_0 \cup \{\alpha_{m-2} + \alpha_{m-1}\} &\text{if $n = m-1$};\\
J_0 \cup \{\alpha'_{n-2} + \alpha'_{n-1}\} & \text{if $m=n-2$}. \end{cases}
\end{align*}

The  $\Tad$-module $\Vggp$ is multiplicity-free; its $\Tad$-weight set contains $J_1$ and is a subset of $F \cup J_1$. 

For the $\Tad$-weights in $J_0$,  basis vectors for the corresponding eigenspaces in $\Vggp$ are given in the following table:
\newline
\begin{tabular}{|c|c|}
\hline
$\Tad$-weight & eigenvector in $\Vggp$\\
\hline
$\beta_r:=\alpha_{r-1} + \alpha_{r}$ & $[X_{-\beta_r}( v_{\lambda_{r}}  + v_{\lambda'_{r-1}})] = -[X_{-\beta_r} (v_{\lambda_{r+1}}  + v_{\lambda'_{r}})]$\\
\hline
$\beta'_s:=\alpha'_{s-1} + \alpha'_{s}$ & $[ X_{-\beta'_s}(v_{\lambda_{s-1}} + v_{\lambda'_{s-1}})] = -  
[ X_{-\beta'_s}(v_{\lambda_{s}}+v_{\lambda'_s})] $\\
\hline
\end{tabular}
\newline with  $2\le r \le L-1$, $2\le s \le K-1$.

If $n=m-1$  then  the $\Tad$-eigenspace of weight $\alpha_{m-2}+\alpha_{m-1}$ is spanned by the following eigenvector:
\newline
\begin{tabular}{|c|c|}
\hline
$\Tad$-weight & eigenvector in $\Vggp$ \\
\hline
$\beta_{m-1}:=\alpha_{m-2} + \alpha_{m-1}$ & $[X_{-\beta_{m-1}}  (v_{\lambda_{m-1}}+v_{\lambda'_{m-2}})] = -[X_{-\beta_{m-1}} v_{\lambda'_{m-1}}]$ \\
\hline
\end{tabular}

If $m=n-2$  then  the $\Tad$-eigenspace of weight $\alpha'_{n-2}+\alpha'_{n-1}$ is spanned by the following eigenvector:
\newline
\begin{tabular}{|c|c|}
\hline
$\Tad$-weight & eigenvector in $\Vggp$ \\
\hline
$\beta'_{n-1}:=\alpha'_{n-2} + \alpha'_{n-1}$ & 
$[ X_{-\beta'_{n-1}} (v_{\lambda_{n-2}}+v_{\lambda'_{n-2}})]= -[ X_{-\beta'_{n-1}} v_{\lambda_{n-1}}]$ \\
\hline
\end{tabular}
\end{prop}

\begin{remark} \label{rem:vggp17}
We use the notation of Proposition~\ref{prop:vgg}. The following somewhat stronger statement holds, but we do not need it in what follows. The $\Tad$-weight set of $\Vggp$ is equal to $F \cup J_1$ and the $\Tad$-eigenspaces with weight in $F$ are spanned by the following eigenvectors:
\newline
\begin{tabular}{|c|c|}
\hline
$\Tad$-weight & eigenvector in $\Vggp$ \\
\hline
$\alpha_i$ & $[X_{-\alpha_i} v_{\lambda_{i+1}}] = -[X_{-\alpha_i} v_{\lambda_{i}'}]$ \\
\hline
$\alpha'_j$ & $[ X_{-\alpha'_j}v_{\lambda_j} ] = -[X_{-\alpha'_j}v_{\lambda_j'}]$ \\
\hline
\end{tabular}
\newline
with $1\le i \le L-1$ and $1\le j \le K-1$. The argument runs as follows.
It is a straightforward matter, using properties of root operators and the fact that $F \inn p(\wg)$, to verify that the eigenvectors listed in this remark belong to $\Vggp$. Alternatively, the fact that $F$ belongs to the $\Tad$-weight set of $\Vggp$ is a consequence of the fact that it belongs to the $\Tad$-weight set of $\Vggo$, which, in turn, follows from Proposition~\ref{prop:case17} or from Knop's computation of the little Weyl group of $W^*$ (see Remark~\ref{rem:lwg}).  
\end{remark}

\begin{cor}  \label{cor:vgg17}
We continue to use the notation of Proposition~\ref{prop:vgg}.  For all $m\ge 1, n \ge 2$, we have that $\Vggo$ is the subspace of $\Vggp$ spanned by the eigenvectors with $\Tad$-weights in $F \cup J_0$. 

Depending on $m$ and $n$ we have the following description of $\Vgg$:
\begin{enumerate}
\item If $n\neq m-1$ and $m\neq n-2$, then $\Vgg = \Vggp = \Vggo$;
\item If $n=m-1$ then $\Vgg = \Vggp$ if and only if $\beta_{m-1} \in \wg$. If $\beta_{m-1} \notin \wg$ then $\Vgg = \Vggo$;
\item If $m=n-2$ then $\Vgg = \Vggp$ if and only if $\beta'_{n-1} \in \wg$. If $\beta'_{n-1} \notin \wg$, then $\Vgg = (\Vg)^{\overline{G}_{x_0}}$.
\end{enumerate}
\end{cor}
\begin{proof}
From Lemma~\ref{lem:Tadweights}(\ref{liealgandtorinvs})  we know that a $\Tad$-eigenvector in $\Vggp$ belongs to $\Vgg$ if and only if its $\Tad$-weight belongs to $\wg$. For all indices $i,j$ such that $1\le i \le L-1$ and $1\le j \le K-1$, we have
\begin{align*}
&\alpha_i = \lambda_{i+1} + \lambda'_i - \lambda_i - \lambda'_{i+1};\\
&\alpha'_j = \lambda_j + \lambda'_j - \lambda_{j+1} - \lambda'_{j-1} 
\end{align*}
(where $\lambda_0 = 0$ when it occurs) and so $\alpha_i, \alpha'_j \in \overline{\wg}$. Consequently 
$\beta_r, \beta'_s \in \overline{\wg}$ when $2\le r \le L-1$ and $2\le s \le K-1$. This implies that $\alpha_i, \alpha'_j, \beta_r, \beta'_s \in \wg = q(\overline{\wg})$.  On the other hand, straightforward verifications (or Lemmas~\ref{lem:betam1inlambda} and \ref{lem:betan1inlambda}) show that  
$\beta_{m-1} \notin \overline{\wg}$ when $n=m-1$ and that $\beta'_{n-1} \notin \overline{\wg}$ when $m=n-2$. 
All the assertions follow.
\end{proof}

\begin{remark} \label{rem:vgg17}
\begin{enumerate}[(1)]
\item Using Remark~\ref{rem:vggp17}, the first assertion in Corollary~\ref{cor:vgg17} can be improved to the statement that $\Vggo$ is a multiplicity-free $\Tad$-module with $\Tad$-weight set equal to $F\cup J_0$. 
\item When $n=m-1$, Lemma~\ref{lem:betam1inlambda} below tells us for which groups $G$ the eigenvector with weight $\beta_{m-1}$ belongs to $\Vgg$.  For example, $\beta_{m-1} \in \wg=\Lambda_{(G,W^*)}$ for $G=\SL(m) \times \GL(m-1)$, but not for $G=\overline{G}=\GL(m) \times \GL(m-1)$. When $m=n-2$, Lemma~\ref{lem:betan1inlambda} does the same for $\beta'_{n-1}$. 
\end{enumerate}
\end{remark}

We will break the proof of Proposition~\ref{prop:vgg} up into several lemmas (see page~\pageref{proofpropvgg} for the actual proof). We first set up some notation. \label{notation17} First note that $G' = G^1 \times G^2$ with $G^1 = \SL(m)$ and $G^2 = \SL(n)$. Let $T^1$ and $T^2$ be the projection of $T'\inn G^1 \times G^2$ to $G^1$ and to $G^2$, respectively. Then $T'=T^1 \times T^2$. Let $\Tad^1$ be the adjoint torus of $G^1$ and $\Tad^2$ be the adjoint torus of $G^2$. We write $\Lambda_R^1$ and $\Lambda_R^2$ for the corresponding root lattices, and $R^{+}_{1}$ and $R^{+}_{2}$ for the sets of positive roots with respect to the Borel subgroups $B^1 \inn G^1$ and $B^2\inn G^2$ whose product $B^1 \times B^2$ is $B'$.  We will use $U^1$ and $U^2$ for the unipotent radical of $B^1$ and $B^2$, respectively. 
Note that the root lattice of $G$, which is the character group of $\Tad$, is $\Lambda_R = \Lambda_R^1 \oplus \Lambda_R^2$ and that $R^{+} = R^{+}_{1} \cup R^{+}_{2}$. 

The next lemma says that there are no `mixed' $\Tad$-weights in $\Vggp$.
\begin{lemma} \label{lem:nomix}
The $\Tad$-weights occurring in $\Vggp$ belong to $(\Lambda_R^1 \oplus 0) \cup (0 \oplus \Lambda_R^2)$.
\end{lemma}
\begin{proof}
Suppose that $v$ is a $\Tad$-eigenvector in $V$ so that  $[v] \in \Vggp$ contradicts the assertion. Since $v$ has nonzero weight for $\Tad^1$ and for $\Tad^2$, $v \notin V^{\Tad^1} = V^{U^1}$ and $v\notin V^{\Tad^2} = V^{U^2}$.  Therefore, there is a simple root $\alpha_i$ of $G^1$ and a simple root $\alpha_j'$ of $G^2$ such that $X_{\alpha_i} v \neq 0$ and $X_{\alpha_j'}v \neq 0$.
Moreover $X_{\alpha_i} v, X_{\alpha_j'} v \in \fg \cdot x_0$ and by
Lemma~\ref{lem:tangentcrit}(\ref{AAA}) this implies, using that $[v]$
contradicts the assertion and that $R^+=R_1^+\cup R_2^+$, that the $\Tad$-weight of  $X_{\alpha_i} v$ belongs to $R^{+}_{2}$ and that of $X_{\alpha_j'} v$ to $R^{+}_1$. It follows that $v$ has $\Tad$-weight $\alpha_i + \alpha_j'$. Consequently 
\(X_{\alpha'_j} v \in \CC(X_{-\alpha_i}x_0). \)
As long as $i+1 \le K$ (which implies that $i\le L$), 
\[X_{-\alpha_i}x_0 =   X_{-\alpha_i}(v_{\lambda_{i+1}} + v_{\lambda'_{i}}) = X_{-\alpha_i}v_{\omega_i}\otimes(v_{\omega'_{i+1}} + v_{\omega'_i})\]
and so $X_{\alpha'_j} v \in \CC(X_{-\alpha_i}x_0)$ implies that  there exist $u_1 \in V(\omega'_{i+1})$ and $u_2 \in V(\omega'_i)$ so that $X_{\alpha'_j}u_1 = v_{\omega'_{i+1}}$ and $X_{\alpha'_j}u_2 = v_{\omega'_{i}}$. This is impossible because the fundamental representations $V(\omega_{i+1}')$ and $V(\omega_{i}')$ cannot both contain a $\Tad^2$-eigenvector of weight the simple root $\alpha_j'$.

On the other hand, if $i+1 > K$, we still have that $i\le L$ because $X_{-\alpha_i}x_0 \neq 0$ (as $X_{\alpha'_j}v$ is a nonzero element of the line it spans). Therefore $i=K=L=n$. Then 
\[X_{-\alpha_i}x_0 = X_{-\alpha_n}x_0  = X_{-\alpha_n}v_{\lambda'_n} = (X_{-\alpha_n} v_{\omega_n})\otimes v_{\omega'_n}.\]
But now, $X_{\alpha'_j} v \in \CC(X_{-\alpha_i}x_0)$ is impossible, since $V(\omega'_n)$ contains no $\Tad^2$-eigenvectors of nonzero weight. 
\end{proof}

For the next lemma, recall that $\wgo$ stands for $\Lambda_{(\overline{G},W^*)} = \<E\>_{\ZZ}$, 
and that $p(\wg) =  \Lambda_{(G',W^*)}$.
\begin{lemma} \label{lem:basislambda}
We have
\begin{equation} \label{eq:lambdaovg}
\wgo = \<\omega_1,\ldots,\omega_L,\omega_1',\ldots,\omega_K'\>_{\ZZ} \inn X(\overline{T})
\end{equation}
Moreover, for $i=1,\ldots,K $ and $j=1, \ldots, L$ we have the following equalities in $X(\overline{T})$
\begin{align} 
\omega'_i &= \lambda_i -\sum_{k=1}^{i-1} (\lambda'_k - \lambda_k) \label{eq:basisinlambda1}
\\
\omega_j&= \sum_{k=1}^{j}(\lambda'_k-\lambda_k). \label{eq:basisinlambda2}
\end{align}

\end{lemma}
\begin{proof}
We first prove equation (\ref{eq:lambdaovg}).
Consider the matrix $F$ whose columns are the coefficients of
\[ \lambda_1, \lambda_1', \lambda_2, \lambda_2', \ldots, \lambda_{K-1}, \lambda'_{K-1}, \lambda_K, (\lambda'_L) 
\] 
in the basis
\[ \omega'_1, \omega_1, \omega'_2, \omega_2, \ldots,  \omega_{K-1}, \omega'_{K-1}, \omega_K, (\omega'_L). 
\]
The brackets in $(\lambda'_L)$ and $(\omega'_L)$ indicate that these weights might not occur: $L = K-1$ or $L=K$, depending on $m$ and $n$. We have that $F$ is a $(K+L) \times (K+L)$ upper triangular matrix with $1$ on the diagonal. So $F$ is invertible over $\ZZ$ which proves (\ref{eq:lambdaovg}). 

Equations (\ref{eq:basisinlambda1}) and (\ref{eq:basisinlambda2}) are obtained by inverting the matrix $F$ or by a straightforward recursive argument. 
\end{proof}

The following lemma will prove useful too. It is a slight generalization of \cite[Corollary 3.9]{bravi&cupit}.
\begin{lemma} \label{lem:consecfund}
Suppose $m \ge 2$ is an integer and suppose $k \le m-1$ is another positive integer. Define the following $\SL(m)$-module:
\[ M := V(\omega_1) \oplus V(\omega_2) \oplus \ldots \oplus V(\omega_k)\]
Furthermore, call the sum of highest weight vectors $m_0$:
\[ m_0 = v_{\omega_1} + v_{\omega_2} + \ldots + v_{\omega_k} \]

Then $(M/\sl(m)\cdot m_0)^{\SL(m)_{m_0}}$ is the multiplicity-free $\Tad$-module with $\Tad$-weight set
\begin{equation} \label{eq:wsm}
\{ \alpha_1+ \alpha_2, \alpha_2+\alpha_3, \ldots , \alpha_{p-1}+ \alpha_{p}\}
\end{equation}
where
\[p=\begin{cases} 
k-1& \text{if $k < m-1$},\\ 
k & \text{if $k=m-1$}. 
\end{cases} 
\]
\end{lemma}
\begin{proof}
The monoid $\<\omega_1, \omega_2, \ldots, \omega_k\>_{\NN}$ is $\SL(m)$-saturated, whence Theorem~\ref{thm:bcf} tells us that $(M/\fg\cdot m_0)^{\SL(m)_{m_0}}$ is a multiplicity-free $\Tad$-module of which the $\Tad$-weights belong to the set $D$ consisting of the following elements of $\Lambda_R$ (see Remark~\ref{rem:srA}):
\begin{itemize}
\item[(\SRo)] $\alpha_i + \alpha_j$ with $1\le i \le m-3$ and $j-i \ge 2$; 
\item[(\SRt)] $2\alpha_i$ with $1\le i \le m-1$;
\item[(\SRth)] $\alpha_{i+1} +\alpha_{i+2} + \ldots + \alpha_{i+r}$ with $0\le i \le m-3$ and  $2 \le r \le m-i-1$;
\item[(\SRf)] $\alpha_{i} + 2\alpha_{i+1} + \alpha_{i+2}$ with $1\le i \le m-3$.
\end{itemize}  
Using the argument of the proof of \cite[Corollary 3.9]{bravi&cupit} we obtain the $\Tad$-weight set $F$ of $(M/\fg\cdot m_0)^{\SL(m)_{m_0}}$: we first exclude `enough' $\Tad$-weights in $D$ from belonging to $F$ and then show that the remaining elements of $D$ belong to $F$. 

Weights of type (\SRo) and (\SRt)  cannot occur in $F$ because the fundamental representations of $\SL(m)$ do not contain such $\Tad$-weights. Next suppose $\gamma = \alpha_{i} + 2\alpha_{i+1} + \alpha_{i+2}$ is a weight of type (\SRf). If $i<k$, then \cite[Proposition 3.4]{bravi&cupit} with $\delta = \alpha_i$ tells us that $\gamma$ does not belong to $F$. If $i \ge k$, then $\<\gamma, \alpha_{i+1}^{\vee}\> = 2$ implies that $\gamma$ does not belong
to  $ \<\omega_1, \ldots \omega_k\>_{\ZZ}$ and a fortiori not $F$.

Now suppose $\gamma$ is a root of type (\SRth) with $r\ge 3$. If $i+2 \le k$, then \cite[Proposition 3.4]{bravi&cupit} with $\delta = \alpha_{i+2}$ tells us that $\gamma$ is not in $F$. If $i+2 > k$, then $i+r > k$ and since $\<\gamma, \alpha_{i+r}^{\vee}\> = 1$ this tells us that $\gamma \notin \<\omega_1, \ldots \omega_k\>_{\ZZ}$ and so again $\gamma$ is not in $F$.

The final type of $\Tad$-weight in $D$ to rule out from $F$ are those of type (\SRth) with $r=2$ and  $i+2 > p$. Then $p=k-1 < m-1$ and therefore  $i+2 > k-1$. If $i+2 >k$, then $\<\gamma,\alpha_{i+2}^{\vee}\> = 1$ tells us that $\gamma \notin \<\omega_1, \ldots \omega_k\>_{\ZZ}$.  If $i+2 = k$, then the equality $\<\gamma, \alpha_{i+3}^{\vee}\> =-1$ yields the same conclusion. Yet again, $\gamma$ does not belong to $F$.

Finally, that $F$ contains the weights listed in (\ref{eq:wsm}) follows
like in the proof of \cite[Corollary 3.9]{bravi&cupit}: for $i\in\{1,\ldots,p-1\}$, the vector
\[[X_{-\alpha_{i+1}}X_{-\alpha_{i}}m_0] = [X_{-\alpha_{i}}X_{-\alpha_{i+1}}m_0] \in M/\sl(m)\cdot m_0\]
has $\Tad$-weight $\alpha_i+\alpha_{i+1}$ and, as a straightforward verification shows, is fixed by $\SL(m)_{x_0}$. 
\end{proof}

Thanks to  Lemma~\ref{lem:nomix}, the $\Tad$-weight set of $(\Vg)^{G'_{x_0}}$ breaks up into two disjoint sets: its intersection with $\Lambda_R^1$ on the one hand, and it intersection with $\Lambda_R^2$ on the other. We will bound these two sets and show that each eigenspace has dimension one. 
We introduce some more  notation. If $M$ is a representation of $\Tad = \Tad^1 \times \Tad^2$, then we denote $M_{\Lambda^1_R}$ (respectively $M_{\Lambda^2_R}$) the subspace of $M$ spanned by eigenvectors with $\Tad$-weight in $\Lambda^1_R \inn \Lambda_R$ (respectively in $\Lambda^2_R \inn \Lambda_R$). Equivalently, $M_{\Lambda^1_R} = M^{\Tad^2}$ and $M_{\Lambda^2_R} = M^{\Tad^1}$.  
 
Note that by Lemma~\ref{lem:nomix} and because $\Vggp$ has no vectors of $\Tad$-weight $0$,  we have the following decomposition of $\Tad$-modules:
\begin{equation} \label{eq:Vgginto1and2}
\Vggp = \Vggp_{\Lambda_R^1} \oplus  \Vggp_{\Lambda_R^2} .
\end{equation}

Put $A:=V^{\Tad^2}=V^{U^2}=V_{\Lambda^1_R}$. Explicitly,
\begin{multline}
A =\CC(v_{\omega_0} \otimes v_{\omega_1'}) \oplus V(\omega_1)\otimes \CC v_{\omega_2'} \oplus \ldots \oplus V(\omega_{K-1}) \otimes \CC v_{\omega'_K} \oplus \\
V(\omega_1) \otimes \CC v_{\omega'_1} \oplus V(\omega_2) \otimes \CC v_{\omega'_2} \oplus \ldots
\oplus V(\omega_L) \otimes \CC v_{\omega'_L}
\end{multline}

\begin{lemma} \label{lem:Aquotisinv}
The inclusion $A \into V$ induces an isomorphism of $\Tad$-modules 
\[ 
\Bigl(\frac{A}{A\cap \fg\cdot x_0}\Bigr)^{G^1_{x_0}} \isom (\Vg)^{G^1_{x_0}}_{\Lambda_R^1}
\]
where $(\Vg)^{G^1_{x_0}}_{\Lambda_R^1}$ is the subspace of $(\Vg)^{G^1_{x_0}}$ spanned by $\Tad$-eigenvectors with weight in $\Lambda_R^1$.
\end{lemma}
\begin{proof}
Consider the exact sequence of $G'_{x_0} \rtimes \Tad$-modules
\begin{equation} \label{eq:exactVgg}
0 \longrightarrow \fg \cdot x_0 \longrightarrow V \longrightarrow V/\fg \cdot x_0 \longrightarrow 0
\end{equation}
We can view this as an exact sequence of $G^1_{x_0} \times T^2_{\ad}$-modules because the direct product $G^1_{x_0} \times T^2_{\ad}$ is a subgroup of $G'_{x_0} \rtimes \Tad$: the action of $T^2_{\ad} (\inn T^1_{\ad} \times T^2_{\ad})$ on $G^1_{x_0} (\inn G^1 \times G^2)$ by conjugation is trivial.
Taking invariants of the exact sequence by the reductive group $T^2_{\ad}$ yields an isomorphism of $G^1_{x_0}$-modules
\[\Bigl(\frac{A}{A\cap \fg\cdot x_0}\Bigr) \isom (\Vg)_{\Lambda_R^1}, \]
since $(\Vg)_{\Lambda_R^1} = (\Vg)^{T^2_{\ad}}$, $V^{T^2_{\ad}}=A$ and $(\fg\cdot x_0)^{T^2_{\ad}}= A \cap \fg \cdot x_0$. Taking $G^1_{x_0}$-invariants yields the claim. 
\end{proof}

\begin{lemma} \label{lem:Vggsplit1}
We have
 that $(\Vg)^{G^1_{x_0}}_{\Lambda_R^1} = \Vggp_{\Lambda_R^1}$
\end{lemma}
\begin{proof}
Since $G^1_{x_0} \inn G'_{x_0}$, we have $\Vggp \inn (\Vg)^{G^1_{x_0}}$ and therefore, by taking $\Tad^2$-invariants, that $\Vggp_{\Lambda_R^1} \inn (\Vg)^{G^1_{x_0}}_{\Lambda_R^1}$. For the other inclusion, 
\begin{equation} \Vggp_{\Lambda_R^1} \supseteq (\Vg)^{G^1_{x_0}}_{\Lambda_R^1}
\end{equation}
it suffices to prove that 
\begin{equation}
(\Vg)^{G^1_{x_0}}_{\Lambda_R^1} \inn \Vggp. \label{eq:Vgg1x0invggp}
\end{equation}
We first note
 that $A = V^{G^2_{x_0}}$ because $G^2_{x_0}$ fixes every highest weight vector for $G^2$  with weight in the image of $E$ under the restriction map $X(\overline{T}) \to X(T') \to X(T^2)$.
Since the quotient map $\varphi:V \to \Vg$ is a map of $G^2_{x_0}$-modules this implies that 
\begin{equation} \label{eq:AVg2x0}
\varphi(A) \inn (\Vg)^{G^2_{x_0}}.\end{equation}
Also, taking $T^2_{\ad}$-invariants of the exact sequence (\ref{eq:exactVgg}) and remembering that $A= V^{T^2_{\ad}}$ we see that 
\begin{equation} \label{eq:varphiA}
\varphi(A) = (\Vg)^{\Tad^2} = (\Vg)_{\Lambda_R^1}.\end{equation} 
Furthermore, the inclusion
\begin{equation} \label{eq:Vgtad2incl}
(\Vg)^{\Tad^2} \inn (\Vg)^{G^2_{x_0}}\end{equation}
we obtain from (\ref{eq:AVg2x0}) and (\ref{eq:varphiA}) is $G^1_{x_0}$-equivariant, because $G^1_{x_0}$ commutes with $G^2_{x_0}$ and with $\Tad^2$. 
Next we claim that $G'_{x_0} = G^1_{x_0} \times G^2_{x_0}$.  Indeed, by Lemma~\ref{lem:basislambda}, $T'_{x_0} = T^1_{x_0} \times T^2_{x_0}$. Moreover,  by Lemma~\ref{lem:isotropygroup}(\ref{iso3}) it follows that $\fg'_{x_0} = \fg^1_{x_0} \oplus \fg^2_{x_0}$ and so by Lemma~\ref{lem:isotropygroup}(\ref{iso1}) we obtain the claim. It implies that 
\[ \bigl((\Vg)^{G^2_{x_0}}\bigr)^{G^1_{x_0}} = \Vggp.\]   
We can now conclude: taking $G^1_{x_0}$-invariants in (\ref{eq:Vgtad2incl}) gives us the desired inclusion (\ref{eq:Vgg1x0invggp}).
\end{proof}
\begin{remark}
In Lemma~\ref{lem:Vggsplit1}, the inclusion $\Vggp_{\Lambda_R^1} \inn (\Vg)^{G^1_{x_0}}_{\Lambda_R^1}$ is immediate and would be sufficient for proving Theorem~\ref{thm:cbc}, since the goal is to bound the dimension of $T_{X_0}\tM^G_{\wm}$. 
The extra information in the lemma allows us to determine $\Vggp$ and $\Vgg \isom H^0(G\cdot x_0,\shN_{X_0})$. 
\end{remark}

The next step in the proof of Proposition~\ref{prop:vgg} is to bound
the $\Tad^1$-weight set of $\Bigl(\frac{A}{A\cap \fg\cdot
  x_0}\Bigr)^{G^1_{x_0}}$. We also describe some of the associated eigenspaces.

\begin{lemma} \label{lem:Ainvars}
Suppose $m\ge 1, n\ge 2$ and put 
\begin{align*}
K&:= \{\alpha_i \colon 1\le i \le L-1\};\\
K_0&:= \{\alpha_{r-1} + \alpha_{r} \colon 2\le r \le L-1\};\\
K_1&:= \begin{cases} K_0 &\text{if $n\neq m-1$}; \\
K_0 \cup \{\alpha_{m-2} + \alpha_{m-1}\} &\text{if $n = m-1$}.
\end{cases}
\end{align*}
The $\Tad^1$-module $\Bigl(\frac{A}{A\cap \fg\cdot x_0}\Bigr)^{G^1_{x_0}}$ is multiplicity-free; its $\Tad^1$-weight set contains $K_1$ and is a subset of $K \cup K_1$. 

For the $\Tad^1$-weights in $K_0$, basis vectors for the corresponding eigenspaces in $\Bigl(\frac{A}{A\cap \fg\cdot x_0}\Bigr)^{G^1_{x_0}}$ are given in the following table:
\newline
\begin{tabular}{|c|c|}
\hline
$\Tad$-weight & eigenvector \\
\hline
$\beta_r:=\alpha_{r-1} + \alpha_{r}$ &$[X_{-\beta_r}( v_{\lambda_{r}}  + v_{\lambda'_{r-1}})] = -[X_{-\beta_r} (v_{\lambda_{r+1}}  + v_{\lambda'_{r}})]$\\
\hline
\end{tabular}
\newline
with $2\le r \le L-1$.

If $n=m-1$  then  the $\Tad^1$-eigenspace of weight $\alpha_{m-2}+\alpha_{m-1}$ is spanned by the following eigenvector:
\newline
\begin{tabular}{|c|c|}
\hline
$\Tad$-weight & eigenvector in $\Vggp$ \\
\hline
$\beta_{m-1}:=\alpha_{m-2} + \alpha_{m-1}$ & $[X_{-\beta_{m-1}}  (v_{\lambda_{m-1}}+v_{\lambda'_{m-2}})] = -[X_{-\beta_{m-1}} v_{\lambda'_{m-1}}]$ \\
\hline
\end{tabular}
\end{lemma}
\begin{proof}
We begin by defining $G^1$-submodules of $A$ which are isomorphic to fundamental representations.
For $i=1,\ldots,L-1$, put 
\[
Z_i := \text{the simple $G^1$-submodule of $A$ with highest weight vector
 $z_i:=v_{\lambda_{i+1}} + v_{\lambda'_i}$ }
\]
Next, put
\[ 
Z_L := \begin{cases}
 \CC v_{\lambda_{L+1}}  \oplus  \CC v_{\lambda'_L}  &\text{if $L=K-1 (=m)$} \\
\text{the simple $G^1$-module with highest weight vector $z_L:= v_{\lambda'_L}$}  &\text{if $L=K$}
\end{cases}
\]
We also define the following
trivial
$G^1_{x_0} \rtimes T^1_{\ad}$-module (the action of $G^1_{x_0}$ is trivial since $G^1_{x_0} = \cap_{\lambda \in E}G^1_{v_{\lambda}}$): 
\[ Z_0 := \CC v_{\lambda_1} \oplus \CC(v_{\lambda_2}-v_{\lambda'_1}) + \ldots + \CC(v_{\lambda_L}-v_{\lambda'_{L-1}}). \]
Next, we define the $G^1_{x_0}\rtimes T^1_{\ad}$-module
\[ Z:= Z_0 \oplus Z_1 \oplus \ldots \oplus Z_L \]
Then
 $A \cap \fg \cdot x_0 \inn Z \inn A$ and we obtain the following exact sequence of $G^1_{x_0} \rtimes T^1_{\ad}$-modules:
\begin{equation} \label{eq:exactseqA}
0 \longrightarrow \frac{Z}{A\cap \fg\cdot x_0} \longrightarrow \frac{A}{A \cap \fg \cdot x_0} \longrightarrow \frac{A}{Z} \longrightarrow 0.
\end{equation}
Taking $G^1_{x_0}$-invariants we obtain the exact sequence of $T^1_{\ad}$-modules
\begin{equation} \label{eq:seqAinv}
 0 \longrightarrow \Bigl(\frac{Z}{A\cap \fg\cdot x_0}\Bigr)^{G^1_{x_0}} \longrightarrow \Bigl(\frac{A}{A \cap \fg \cdot x_0}\Bigr)^{G^1_{x_0}} \longrightarrow \Bigl(\frac{A}{Z}\Bigr)^{G^1_{x_0}}.
\end{equation}

We now want to determine the $\Tad^1$-modules $\Bigl(\frac{Z}{A\cap \fg\cdot x_0}\Bigr)^{G^1_{x_0}}$ and $\Bigl(\frac{A}{Z}\Bigr)^{G^1_{x_0}}$. We begin with the first. To do so, put
\begin{align*}
y_0 &:= \begin{cases}
z_1 + z_2 +\ldots z_L &\text{if $L < m$};\\
z_1 + z_2 + \ldots z_{L-1} &\text{if $L=m$};
\end{cases}\\
\overline{Z_0}&:=
\begin{cases}
 Z_0 &\text{if $L<m$}' \\
 Z_0 \oplus Z_L &\text{if $L=m$};
\end{cases}\\
\overline{Z} &:= 
\begin{cases}
Z_1 \oplus \ldots \oplus Z_L &\text{if $L<m$};\\
Z_1 \oplus \ldots \oplus Z_{L-1} &\text{if $L=m$}.
\end{cases}
\end{align*}
Then we have the following decompositions as $G^1_{x_0} \rtimes T^1_{\ad}$-modules:
\begin{align*}
Z &= \overline{Z} \oplus \overline{Z_0}; \\
A\cap \fg\cdot x_0 & = \fg^1 \cdot y_0 \oplus \overline{Z_0}.
\end{align*}
It follows that as $G^1_{x_0} \rtimes T^1_{\ad}$-modules
\[
\frac{Z}{A\cap \fg \cdot x_0} \isom \frac{\overline{Z}}{\fg^1\cdot y_0}
\]
After remarking that 
 $G^1_{y_0} = G^1_{x_0}$ and that $\overline{Z}$ is a sum of
 consecutive fundamental representations of $G^1$,
 Lemma~\ref{lem:consecfund} describes  $\Bigl(\frac{Z}{A\cap \fg\cdot
   x_0}\Bigr)^{G^1_{x_0}}$ as a $\Tad^1$-module. It is multiplicity-free and its $\Tad^1$-weight set is
\begin{align} 
&\{\alpha_1 + \alpha_2, \alpha_2 + \alpha_3,\ldots, \alpha_{L-2}+\alpha_{L-1}\} &&\text{if $L < m-1$ or $L=m$} \label{eq:T1modwe1}
\\
&\{\alpha_1 + \alpha_2, \alpha_2 + \alpha_3, \ldots, \alpha_{m-2}+\alpha_{m-1}\} &&\text{if $L=m-1 (=n)$} \label{eq:T1modwe2}
\end{align} 

We now turn to $\Bigl(\frac{A}{Z}\Bigr)^{G^1_{x_0}}$. We define 
\[
A_i := \text{the simple $G^1$-module with highest weight vector $(v_{\lambda_{i+1}}-v_{\lambda_i'})$}\] 
for $i=1,\ldots, L-1$
and then
\[
\overline{A} :=
\begin{cases}
\CC v_{\lambda_1} \oplus A_1 \oplus \ldots \oplus A_{L-1} &\text{if $L<m$;} \\
\CC v_{\lambda_1} \oplus A_1 \oplus \ldots \oplus A_{L-1} \oplus Z_L &\text{if $ L = m$.}
\end{cases}
\]
Then we have the following decomposition as $G^1_{x_0} \rtimes \Tad^1$-modules: 
\[A = \overline{A} \oplus \overline{Z}\]
and therefore the isomorphism of $G^1_{x_0}\rtimes \Tad^1$-modules
\[\frac{A}{Z} \isom \frac{\overline{A}}{\overline{Z_0}} \]
Now note that for $i=1,\ldots L-1$, $A_i \isom V(\omega_i)$ as a $G^1$-module, and that
\begin{equation*}
\overline{Z_0} = \overline{A}^{U_1}=
\begin{cases} \CC v_{\lambda_1} \oplus
(A_1\oplus \ldots \oplus A_{L-1})^{U^1} &\text{if $L <m$;} \\
\CC v_{\lambda_1} \oplus
(A_1\oplus \ldots \oplus A_{L-1})^{U^1} \oplus Z_L &\text{if $L =m$.}
\end{cases}
\end{equation*}
It follows that we have the following isomorphism of $G^1_{x_0} \rtimes T^1_{\ad}$-modules:
\[\frac{\overline{A}}{\overline{Z_0}} \isom \frac{A_1 \oplus \ldots \oplus A_{L-1}}{(A_1 \oplus \ldots \oplus A_{L-1})^{U_1}}\]
Therefore 
\[\Bigl(\frac{A}{Z}\Bigr)^{G^1_{x_0}} \isom \Bigl(\frac{A_1 \oplus \ldots \oplus A_{L-1}}{(A_1 \oplus \ldots \oplus A_{L-1})^{U_1}}\Bigr)^{G^1_{x_0}} \inn  \Bigl(\frac{A_1 \oplus \ldots \oplus A_{L-1}}{(A_1 \oplus \ldots \oplus A_{L-1})^{U_1}}\Bigr)^{\fu_1}\]
as $T^1_{\ad}$-modules. 
Let $v$ be a $T^1_{\ad}$-eigenvector in $ A_1 \oplus \ldots \oplus A_{L-1}$ such that $[v] \in  \Bigl(\frac{A_1 \oplus \ldots \oplus A_{L-1}}{(A_1 \oplus \ldots \oplus A_{L-1})^{U_1}}\Bigr)^{\fu_1}$ is nonzero. Then $v$ is not a highest weight vector and so there is a simple root $\alpha_i$ of $G^1$ such that $X_{\alpha_i} v \neq 0$ and $X_{\alpha_i} v \in (A_1 \oplus \ldots \oplus A_{L-1})^{U_1}$. It follows that $v$ has $T^1_{\ad}$ weight $\alpha_i$ and, since all the $A_k$ are fundamental representations,  that $v \in \CC[ X_{-\alpha_i} (v_{\lambda_{i+1}} - v_{\lambda'_i})]$.

We have proved the lemma's claim that   $\Bigl(\frac{A}{A \cap
  \fg \cdot x_0}\Bigr)^{G^1_{x_0}}$ is multiplicity-free and its claim
about the module's $\Tad^1$-weight set. What remains is to prove that
the listed eigenvectors belong to $\Bigl(\frac{A}{A \cap
  \fg \cdot x_0}\Bigr)^{G^1_{x_0}}$.  This is straightforward.
\end{proof}

We will proceed in exactly the same way to determine the $\Tad^2$-module $\Vggp_{\Lambda_R^2}$.
We now put $C := V^{U^1} = V^{T^1_{\ad}}=V_{\Lambda^2_R}$. 
With the same proofs
 we obtain the following analogs of Lemmas \ref{lem:Aquotisinv} and \ref{lem:Vggsplit1}.

\begin{lemma} \label{lem:Cquotisinv}
The inclusion $C \into V$ induces an isomorphism of $\Tad$-modules 
\[ 
\Bigl(\frac{C}{C\cap \fg\cdot x_0}\Bigr)^{G^2_{x_0}} \isom (\Vg)^{G^2_{x_0}}_{\Lambda_R^2}
\]
where $(\Vg)^{G^2_{x_0}}_{\Lambda_R^2}$ is the subspace of $(\Vg)^{G^2_{x_0}}$ spanned by $\Tad$-eigenvectors with weight in $\Lambda_R^2$.
\end{lemma}

\begin{lemma} \label{lem:Vggsplit2}
We have that $(\Vg)^{G^2_{x_0}}_{\Lambda_R^2} = \Vggp_{\Lambda_R^2}$
\end{lemma}

Here is an analogue to Lemma~\ref{lem:Ainvars}.

\begin{lemma}  \label{lem:Cinvars}
Suppose $m\ge 1, n\ge 2$ and put 
\begin{align*}
K&:= \{\alpha'_j \colon 1\le j \le K-1\};\\
K_0&:= \{\alpha'_{s-1} + \alpha'_{s} \colon 2\le s \le K-1\};\\
K_1&:= \begin{cases} K_0 &\text{if $m\neq n-2$}; \\
K_0 \cup \{\alpha'_{n-2} + \alpha'_{n-1}\} &\text{if $m = n-2$}.
\end{cases}
\end{align*}
The $\Tad^2$-module $\Bigl(\frac{C}{C\cap \fg\cdot x_0}\Bigr)^{G^2_{x_0}}$ is multiplicity-free; its $\Tad^2$-weight set contains $K_1$ and is a subset of $K \cup K_1$. 

For the $\Tad^2$-weights in $K_0$, basis vectors for the corresponding
eigenspaces in $\Bigl(\frac{C}{C\cap \fg\cdot x_0}\Bigr)^{G^2_{x_0}}$
are given in the following table:
\newline
\begin{tabular}{|c|c|}
\hline
$\Tad$-weight & eigenvector \\
\hline
$\beta'_s:=\alpha'_{s-1} + \alpha'_{s}$ & $[X_{-\beta'_s}(v_{\lambda_{s-1}}+v_{\lambda'_{s-1}})] = -  
[X_{-\beta'_s}(v_{\lambda_{s}}+v_{\lambda'_{s}})] $\\
\hline
\end{tabular}
\newline
with $2\le s \le K-1$.

If $m = n-2$ then the $T^2_{\ad}$-eigenspace of weight $\alpha'_{n-2}+\alpha'_{n-1}$ in
$\Bigl(\frac{C}{C\cap \fg\cdot x_0}\Bigr)^{G^2_{x_0}}$is spanned by
the following eigenvector: 
\newline
\begin{tabular}{|c|c|}
\hline
$\Tad$-weight & eigenvector \\
\hline
$\beta'_{n-1}:=\alpha'_{n-2} + \alpha'_{n-1}$ & 
$[X_{-\beta'_{n-1}}(v_{\lambda_{n-2}}+v_{\lambda'_{n-2}})]= -[X_{-\beta'_{n-1}} v_{\lambda_{n-1}}]$\\
\hline
\end{tabular}
\end{lemma}
\begin{proof}
We will proceed as in the proof of Lemma~\ref{lem:Ainvars}. We begin by defining $G^2$-submodules of $C$ which are isomorphic to fundamental representations.
For $i=1,\ldots,K-1$, put
\[
Z_i := \text{the simple $G^2$-submodule of $C$ with highest weight vector
 $z_i:=v_{\lambda_{i}} + v_{\lambda'_i}$.}
\]
Next, put
\[ 
Z_K:= \begin{cases}
 \CC v_{\lambda_{K}}  \oplus  \CC v_{\lambda'_K}  &\text{if $K=L (=n)$} \\
\text{the simple $G^2$-module with highest weight vector $z_K:= v_{\lambda_K}$}  &\text{if $K=L+1$}
\end{cases}
\]
We also define a trivial $G^2_{x_0}\rtimes T^2_{\ad}$-module $Z_0$ to
`account' for the highest weight vectors in $C$ missing from $Z_1 \oplus \ldots \oplus Z_K$:
\[ Z_0 := \CC (v_{\lambda_1}-v_{\lambda'_1}) \oplus \CC(v_{\lambda_2}-v_{\lambda'_2}) + \ldots + \CC(v_{\lambda_{K-1}}-v_{\lambda'_{K-1}}) \]
Next, we define the $G^2_{x_0} \rtimes T^2_{\ad}$-module 
\[Z:= Z_0 \oplus Z_1 \oplus \ldots \oplus Z_K.\]
Then $C \cap \fg \cdot x_0 \inn Z \inn C$ and we obtain the following exact sequence of $G^2_{x_0} \rtimes T^2_{\ad}$-modules:
 \[ 0 \longrightarrow \frac{Z}{C\cap \fg\cdot x_0} \longrightarrow \frac{C}{C \cap \fg \cdot x_0} \longrightarrow \frac{C}{Z} \longrightarrow 0
\]
Taking $G^2_{x_0}$-invariants we obtain the following exact sequence of $T^2_{\ad}$-modules
\begin{equation} \label{eq:seqCinv}
 0 \longrightarrow \Bigl(\frac{Z}{C\cap \fg\cdot x_0}\Bigr)^{G^2_{x_0}} \longrightarrow \Bigl(\frac{C}{C \cap \fg \cdot x_0}\Bigr)^{G^2_{x_0}} \longrightarrow \Bigl(\frac{C}{Z}\Bigr)^{G^2_{x_0}}
\end{equation}

We now want to determine $\Bigl(\frac{Z}{C\cap \fg\cdot x_0}\Bigr)^{G^2_{x_0}}$ and $\Bigl(\frac{C}{Z}\Bigr)^{G^2_{x_0}}$. We begin with the first. To do so, put
\begin{align*}
y_0 &:= \begin{cases}
z_1 + z_2 +\ldots z_K &\text{if $K < n$};\\
z_1 + z_2 + \ldots z_{K-1} &\text{if $K=n$};
\end{cases}\\
\overline{Z_0}&:=
\begin{cases}
 Z_0 &\text{if $K<n$}; \\
 Z_0 \oplus Z_K &\text{if $K=n$};
\end{cases}\\
\overline{Z} &:= 
\begin{cases}
Z_1 \oplus \ldots \oplus Z_K &\text{if $K<n$};\\
Z_1 \oplus \ldots \oplus Z_{K-1} &\text{if $K=n$}.
\end{cases}
\end{align*}
Then we have the following decompositions as $G^2_{x_0} \rtimes T^2_{\ad}$-modules:
\begin{align*}
Z &= \overline{Z} \oplus \overline{Z_0}; \\
C\cap \fg\cdot x_0 & = \fg_2 \cdot y_0 \oplus \overline{Z_0}.
\end{align*}
It follows that as $G^2_{x_0} \rtimes T^2_{\ad}$-modules
\[
\frac{Z}{C\cap \fg \cdot x_0} \isom \frac{\overline{Z}}{\fg^2\cdot y_0}.
\]
After remarking that $G^2_{y_0} = G^2_{x_0}$  and that
$\overline{Z}$ the a sum of consecutive fundamental representations of
$G^2$, Lemma~\ref{lem:consecfund} describes  $\Bigl(\frac{Z}{C\cap
  \fg\cdot x_0}\Bigr)^{G^2_{x_0}}$ as a $\Tad^2$-module. It is
multiplicity-free and its $\Tad^2$-weight set is
\begin{align} 
&\{\alpha'_1 + \alpha'_2, \alpha'_2 + \alpha'_3,\ldots, \alpha'_{K-2}+\alpha'_{K-1}\} &&\text{if $K < n-1$ or $K=n$}; \label{eq:T2modwe1}
\\
&\{\alpha'_1 + \alpha'_2, \alpha'_2 + \alpha'_3, \ldots, \alpha'_{K-1}+\alpha'_{K}\} &&\text{if $K=n-1 (\Leftrightarrow m=n-2)$}. \label{eq:T2modwe2}
\end{align} 

We now turn to $\Bigl(\frac{C}{Z}\Bigr)^{G^2_{x_0}}$. We define 
\[
C_i := \text{the simple $G^2$-module with highest weight vector $(v_{\lambda_{i}}-v_{\lambda_i'})$}\] 
for $i=1,\ldots, K-1$
and then
\[
\overline{C} :=
\begin{cases}
 C_1 \oplus \ldots \oplus C_{K-1} &\text{if $K<n$;} \\
C_1 \oplus \ldots \oplus C_{K-1} \oplus Z_K &\text{if $ K = n$.}
\end{cases}
\]
Then we have the following decomposition as $G^2_{x_0} \rtimes \Tad^2$-modules: 
\[C = \overline{C} \oplus \overline{Z}\]
and therefore an isomorphism of $G^2_{x_0}\rtimes \Tad^2$-modules
\[\frac{C}{Z} \isom \frac{\overline{C}}{\overline{Z_0}}. \]
Now note that for $i=1,\ldots K-1$, $C_i \isom V(\omega'_i)$ as a $G_2$-module, and that
\begin{equation*}
\overline{Z_0}  = \overline{C}^{U_2}=
\begin{cases} 
(C_1\oplus \ldots \oplus C_{K-1})^{U_2} &\text{if $K <n$}; \\
(C_1\oplus \ldots \oplus C_{K-1})^{U_2} \oplus Z_K &\text{if $K =n$}.
\end{cases}
\end{equation*}
It follows that we have the following isomorphism of $G^2_{x_0} \rtimes T^2_{\ad}$-modules:
\[\frac{\overline{C}}{\overline{Z_0}} \isom \frac{C_1 \oplus \ldots \oplus C_{K-1}}{(C_1 \oplus \ldots \oplus C_{K-1})^{U_2}}\]

The rest of the proof runs exactly as in the proof of Lemma~\ref{lem:Ainvars}.
\end{proof}

Here is the formal proof of Proposition~\ref{prop:vgg}.
\begin{proof}[Proof of Proposition~\ref{prop:vgg}]
\label{proofpropvgg}
The proof consists of equation~(\ref{eq:Vgginto1and2}); Lemma~\ref{lem:Vggsplit1} and Lemma~\ref{lem:Vggsplit2}; Lemma~\ref{lem:Aquotisinv} and Lemma~\ref{lem:Cquotisinv}; Lemma~\ref{lem:Ainvars} and Lemma~\ref{lem:Cinvars}.
\end{proof}

We next prove that the sections in $H^0(G \cdot x_0, \shN_{X_0})^G$
corresponding to the invariants in $\Vgg$ with $\Tad$-weights
belonging to the set
$J_1$  of Proposition~\ref{prop:vgg} do \emph{not} belong to $H^0(X_0,
\shN_{X_0})^G \isom T_{X_0}\tM^G_{\wm}$. We begin by expressing the $\Tad$-weights in terms of the basis $E$ of $\wgo$.

\begin{lemma} \label{lem:betainlambda}
Using the notation introduced in equations (\ref{eq:case17wts1}) and
(\ref{eq:case17wts2}) on page~\pageref{eq:case17wts1} and that of
Proposition~\ref{prop:vgg} we have the following equalities in $\wgo$:
\begin{align*}
\beta_r &= \lambda'_{r-1} + \lambda_{r+1} - \lambda_{r-1} - \lambda'_{r+1} \quad\text{ when $2\le r \le L-1$}; \\
\beta_s' &= \lambda_{s-1} + \lambda'_s - \lambda'_{s-2} - \lambda_{s+1} \quad\text{ when  $2 \le s \le K-1$},
\end{align*}
where $\lambda'_0 = 0$. 
\end{lemma}
\begin{proof}
Straightforward verification.
\end{proof}

\begin{lemma} \label{lem:betam1inlambda} We use the notation of Proposition~\ref{prop:vgg} and suppose $n=m-1$. Then the following are equivalent (recall that, by assumption, $(G,W)$ is spherical):
\begin{enumerate}
\item  $\beta_{m-1} \in \wg$;
 \item $T = \ker(\omega_m - b\omega'_{n}) \inn \overline{T}$ for some integer $b$. 
  \end{enumerate}
For every integer $b$ we have the following equality in $X(\overline{T})$:
\begin{multline} \label{eq:betaminlambdaup}
\beta_{m-1} +(\omega_m - b\omega'_n) = \lambda'_{m-1} +(-1-b)\lambda_{m-1} +\\ (b+2)(\lambda'_{m-2}-\lambda_{m-2})
+  (b+1) \sum_{k=1}^{m-3}(\lambda'_k - \lambda_k) 
\end{multline}  
 Consequently, if $T=\ker(\omega_m - b \omega'_n)$ for some integer $b$, restricting (\ref{eq:betaminlambdaup}) to $T$ yields the following equality in
 $\wg$:
\begin{equation} \label{eq:betaminlambda}
\beta_{m-1} = \lambda'_{m-1} +(-1-b)\lambda_{m-1} +(b+2)(\lambda'_{m-2}-\lambda_{m-2})
+  (b+1) \sum_{k=1}^{m-3}(\lambda'_k - \lambda_k).
\end{equation}
\end{lemma}
\begin{proof}
 We consider $\beta_{m-1}$ as an element of $X(\overline{T})$ and  first determine when $q(\beta_{m-1}) \in \wg$. 
Recall that  $p\colon X(T) \onto X(T')$, $q\colon X(\overline{T}) \onto X(T)$ and $r\colon X(\overline{T}) \onto X(T')$ are the restriction maps.  Since $r = p \circ q$,
\[
\ker q \inn \ker r = \<\omega_m, \omega_n'\>_{\ZZ}.\]
Next, note that 
\begin{equation} \label{eq:betamm1}
\beta_{m-1} = \alpha_{m-2} + \alpha_{m-1} = -\omega_{m-3} + \omega_{m-2} +\omega_{m-1} -\omega_m \in X(\overline{T}).
\end{equation}
where $\omega_{m-3} = 0$ if $m=3$. 
Since by equation (\ref{eq:lambdaovg})
\[ \overline{\wg} = \<\omega_1,\ldots,\omega_{m-1},\omega_1',\ldots,\omega_n'\>_{\ZZ}\]
it follows that $q(\beta_{m-1}) \in \wg = q(\overline{\wg})$ if and only if $q(\omega_m) \in  \wg$. 
This means there exists $\gamma \in  \overline{\wg}$ such that $q(\omega_m) = q(\gamma)$, that is $\omega_m - \gamma \in \ker q$. We claim $\gamma$ belongs to $\ZZ \om_n'$. Indeed, $\omega_m - \gamma \in \ker r$ and therefore $\gamma \in \ker r$. Using the linear independence of the set $\{\omega_1,\ldots, \omega_m, \omega_1',\ldots,\omega_n'\}$ in $X(\overline{T})$ we have that $\overline{\wg} \cap \ker r = \ZZ\omega_n'$. This proves the claim, and we have proved (identifying $\beta_{m-1}=q(\beta_{m-1})$ since we have identified the root lattices of $G$ and $\overline{G}$) that $\beta_{m-1} \in  \wg$ if and only if there exists an integer $b$ such that $\omega_m - b \omega'_n \in \ker q$. 

Now $W$ is spherical as a $G$-module if and only if the restriction of $q$ to $\overline{\wg}$ is injective. 
That is, if and only if $\ker q \cap \overline{\wg} = 0$. Since $\ker r \cap   \overline{\wg} = \ZZ\omega_n'$, this is equivalent to $\ker q \cap \ZZ\omega_n' = 0$. 
Using that $\ker q \inn \ker r = \<\omega_{m}, \omega_n'\>_{\ZZ}$, it follows
that ($(G,W)$ is spherical and) $\beta_{m-1} \in \wg$ if and only if there exists an integer $b$ such that $\ker q = \<\omega_m - b \omega'_n\>_{\ZZ}$. This is equivalent to the first assertion. 
 
The straightforward verification of (\ref{eq:betaminlambdaup}) is left to the reader.
\end{proof}

\begin{lemma} \label{lem:betan1inlambda}
We use the notation of Proposition~\ref{prop:vgg} and suppose $m=n-2$. Then 
the following are equivalent (recall that, by assumption, $(G,W)$ is spherical):
\begin{enumerate}
\item   $\beta'_{n-1} \in \wg$;
\item $T = \ker(a\omega_m - \omega'_{n}) \inn \overline{T}$ for some integer $a$. 
\end{enumerate}
For every integer $a$ we have the following equality in $X(\overline{T})$:
\begin{multline} \label{eq:betanm2up}
\beta'_{n-1} -(a\omega_m-\omega'_n)  = \lambda_{n-1} + (2+a)\lambda_{n-2} + (-2-a) \lambda'_{n-3}  + \\
(1+a)\lambda_{n-3} +  (-1-a) \lambda'_{n-2} + (1+a) \sum_{k=1}^{n-4}(\lambda_k- \lambda'_k)
\end{multline}
Consequently, if $T=\ker(a\omega_m-\omega'_n)$ for some integer $a$, restricting (\ref{eq:betanm2up}) to $T$ yields the following equality in $\wg$: 
\begin{multline} \label{eq:betanm2}
\beta'_{n-1} = \lambda_{n-1} + (2+a)\lambda_{n-2} + (-2-a) \lambda'_{n-3}  + \\
(1+a)\lambda_{n-3} +  (-1-a) \lambda'_{n-2} + (1+a) \sum_{k=1}^{n-4}(\lambda_k- \lambda'_k).
\end{multline}
\end{lemma}
\begin{proof}
This proof is very similar to that of Lemma~\ref{lem:betam1inlambda}. For the first assertion, 
the arguments are identical except that now equation~(\ref{eq:lambdaovg}) gives
\[\overline{\wg} = \<\omega_1,\ldots,\omega_m,\omega_1',\ldots,\omega_{n-1}'\>_{\ZZ}\]
and that
\[\beta_{n-1}' = \alpha_{n-2}' + \alpha_{n-1}'= -\omega_{n-3}'+\omega_{n-2}' + \omega_{n-1}' -\omega_{n}' \in X(\overline{T}).\]

The straightforward verification of equation~(\ref{eq:betanm2up}) is left to the reader.
\end{proof}

We now apply Proposition~\ref{prop:excl} a few times to exclude the sections in
$H^0(G \cdot x_0, \shN_{X_0})^G \isom \Vgg$ with $\Tad$-weight in
$J_1$ (of Proposition~\ref{prop:vgg}) from belonging to $H^0(X_0,
\shN_{X_0})^G$.  We begin with the weights in $J_0 \inn J_1$.
  
\begin{lemma} \label{lem:exclbetar}
Suppose $2\le r \le L-1$. The section $s \in H^0(G \cdot x_0, \shN_{X_0})^G$ defined by
\[s(x_0) =  [X_{-\beta_r} (v_{\lambda_{r}} + v_{\lambda'_{r-1}})] = -[X_{-\beta_r} (v_{\lambda_{r+1}}+v_{\lambda'_r})] \in \Vgg \]
does not extend to $X_0$.
\end{lemma}
\begin{proof}
We apply Proposition~\ref{prop:excl} with $\lambda = \lambda_{r+1}$ and 
\begin{equation}v=  X_{-\beta_r}(v_{\lambda_r} + v_{\lambda'_{r-1}}) \in V. \label{eq:vexclbetar}
\end{equation}
Recall that $\beta_r = \alpha_{r-1} + \alpha_r$. 
We check the four conditions of Proposition~\ref{prop:excl}:
(\ESo) follows from Lemma~\ref{lem:betainlambda};
(\ESt) is clear from (\ref{eq:vexclbetar});
(\ESth) follows from the equalities $\lambda_{r+1} = \omega_r + \omega_{r+1}'$, $\<\lambda'_{r}, \alpha_r^{\vee}\> =1$ and $\<\lambda'_{r+1},(\alpha'_{r+1})^{\vee}\>=1$;
for (\ESf) take $\delta = \lambda'_r = \omega_r + \omega'_r$.
\end{proof}

\begin{lemma} \label{lem:exclbetaprj}
Suppose $2\le j \le K-1$. The section $s \in H^0(G \cdot x_0, \shN_{X_0})^G$ defined by
\[s(x_0) =  [X_{-\beta'_j}(v_{\lambda_{j-1}}+v_{\lambda'_{j-1}})] = -  
[ X_{-\beta'_j}(v_{\lambda_{j}}+v_{\lambda'_j})] \in \Vgg \]
does not extend to $X_0$.
\end{lemma}
\begin{proof}
We again check the conditions in  Proposition~\ref{prop:excl}, this time with $\lambda= \lambda'_{j}$, 
 and 
\begin{equation} \label{eq:vexclbetaprj}
v= X_{-\beta_j'}(v_{\lambda_{j-1}} + v_{\lambda'_{j-1}}) \in V.\end{equation}
Recall that $\beta'_j = \alpha'_{j-1} + \alpha'_j$. 
(\ESo) follows from Lemma~\ref{lem:betainlambda};
(\ESt) is clear from (\ref{eq:vexclbetaprj});
(\ESth) follows from the equalities $\lambda'_{j} = \omega_j + \omega_{j}'$, $\<\lambda_{j+1}, \alpha_j^{\vee}\> =1$ and $\<\lambda_{j},(\alpha'_{j})^{\vee}\>=1$;
for (\ESf) take $\delta = \lambda_j= \omega_{j-1} + \omega'_j$.
\end{proof}

We now deal with the $\Tad$-weight in $J_1\setminus J_0$, first when
$n=m-1$, then when $m=n-2$. Note that the eigenspace of $H^0(G\cdot
x_0, \shN_{X_0})^G$ with this weight is only nontrivial for certain
intermediate subgroups $G$, see  Lemma~\ref{lem:betam1inlambda} and Lemma~\ref{lem:betan1inlambda}, which also prove that the formula for $s$ in Lemmas \ref{lem:exclbetam1} and \ref{lem:exclbetan2} actually defines a section of $H^0(G \cdot x_0, \shN_{X_0})^G \isom \Vgg$.

\begin{lemma} \label{lem:exclbetam1}
Suppose $n=m-1$ and let $b$ be an integer. Suppose that the maximal torus $T$ of $G$ satisfies
$T=\ker(\omega_m - b \omega'_n)$. Then the section  $s \in H^0(G \cdot x_0, \shN_{X_0})^G$ defined by 
\[s(x_0) = [X_{-\beta_{m-1}}(v_{\lambda_{m-1}}+v_{\lambda'_{m-2}})] = -[X_{-\beta_{m-1}} v_{\lambda'_{m-1}}] \in \Vgg \]
does not extend to $X_0$. 
\end{lemma}
\begin{proof}
We consider two cases: $b\le -2$ and $b> -2$. 

(i)If $b\le -2$, then $(-1-b) \ge 1$. We apply Proposition~\ref{prop:excl} with $\lambda = \lambda_{m-1}$
and \[v = X_{-\beta_{m-1}} v_{\lambda'_{m-1}} \in V.\] We check the four conditions:
(\ESo) follows from equation (\ref{eq:betaminlambda});
(\ESt) is clear from the description of $v$ given above;
(\ESth) follows from the equalities $\lambda_{m-1} = \omega_{m-2} + \omega'_{n}$ and $\<\lambda'_{m-2}, \alpha_{m-2}^{\vee}\> = 1$;
for (\ESf) take $\delta = \lambda'_{m-2} = \omega_{m-2} + \omega'_{m-2}$.

(ii)If $b>-2$, then we apply Proposition~\ref{prop:excl} with $\lambda = \lambda'_{m-2}$ 
and the same $v$ as in part (i). We again check the four conditions:
(\ESo) follows from equation (\ref{eq:betaminlambda});
(\ESt) is clear from the description of $v$ given above;
(\ESth) follows from the equalities $\lambda'_{m-2}= \omega_{m-2} + \omega'_{m-2}$,  $\<\lambda_{m-1}, \alpha_{m-2}^{\vee}\>=1$ and $\<\lambda_{m-2}, (\alpha'_{m-2})^{\vee}\>=1$;
for (\ESf) take $\delta = \lambda_{m-1} = \omega_{m-2} + \omega'_{m-1}$.
\end{proof}

\begin{lemma} \label{lem:exclbetan2}
Suppose $m=n-2$ and let $a$ be an integer. Suppose that the maximal torus $T$ of $G$ satisfies
$T=\ker(a\omega_m - \omega'_n)$. Then the section  $s \in H^0(G \cdot x_0, \shN_{X_0})^G$ defined by 
\[s(x_0) = [X_{-\beta'_{n-1}} (v_{\lambda_{n-2}}+v_{\lambda'_{n-2}})]= -[X_{-\beta'_{n-1}} v_{\lambda_{n-1}}] \in \Vgg \]
does not extend to $X_0$. 
\end{lemma}
\begin{proof}
We break the proof up into two cases ($a\le -2$ and $a>-2$), each of which is treated by an application of   
Proposition~\ref{prop:excl} with \(v = X_{-\beta'_{n-1}}v_{\lambda_{n-1}} \in V,\) but the dominant weight $\lambda \in E$ depends on $a$. 

(i)When $a\le -2$, we put $\lambda = \lambda'_{n-2}$. We check the
four conditions of Proposition~\ref{prop:excl}:
(\ESo) follows from equation~(\ref{eq:betanm2});
(\ESt) is clear from the description of $v$ given above;
(\ESth) follows from the equalities $\lambda'_{n-2}= \omega_{n-2} + \omega'_{n-2}$,  $\<\lambda_{n-1}, \alpha_{n-2}^{\vee}\>=1$ and $\<\lambda_{n-2}, (\alpha'_{n-2})^{\vee}\>=1$;
for (\ESf) take $\delta = \lambda_{n-2} = \omega_{n-3} + \omega'_{n-2}$.

(ii)When $a>-2$, we apply Proposition~\ref{prop:excl} with $\lambda = \lambda_{n-2}$. We check the four conditions: 
(\ESo) follows from equation~(\ref{eq:betanm2});
(\ESt) is clear from the description of $v$ given above;
(\ESth) follows from the equalities $\lambda_{n-2}= \omega_{n-3} + \omega'_{n-2}$,  $\<\lambda'_{n-3}, \alpha_{n-3}^{\vee}\>=1$ (if $n > 4$, otherwise $\lambda_{n-2} = \omega'_{n-2}$) and $\<\lambda'_{n-2}, (\alpha'_{n-2})^{\vee}\>=1$;
for (\ESf) take $\delta = \lambda'_{n-2} = \omega_{n-2} + \omega'_{n-2}$.
\end{proof}

\subsection{The modules $(\GL(m) \times \GL(n), (\CC^m \otimes \CC^n) \oplus (\CC^n)^*)$ with $1\leq m, 2\leq n$} \label{subsec:case18}

We begin with some notation. Put 
\begin{align*}
K &= \min(m, n-1) \\
L &= \min(m,n).
\end{align*}
Note that $K=L-1$ (when $m>n-1$) or $K=L$ (otherwise). We will also use the following notation:
\begin{align*}
\lambda_i &= \omega_i + \omega_{i-1}'  &&\text{for }i\in\{1, \ldots, K\} \text{ (with $\omega'_0 = 0$)}
\\ 
\mu&=\omega'_{n-1} - \omega'_n \\ 
\lambda'_i &= \omega_i + \omega_i' &&\text{for }i\in\{1,\ldots, L\}. 
\end{align*}
For the modules under consideration,
\begin{align*}
&E=\{\lambda_i\colon 1\le i \le K\} \cup \{\lambda'_i\colon 1\le i \le L\} \cup \{\mu\} ;\\
&d_W=K+L-1= \min(2m+1,2n)-2.
\end{align*}
These modules are not spherical for $G'$  because $\wg \cap \<\om_m, \om'_n\>_{\ZZ} \neq 0$.  Moreover, for the same reason, $\wm$ is not $G$-saturated for any intermediate group $G$ for which $W$ is spherical.

In this section we will prove the following proposition.
\begin{prop} \label{prop:case18}
The $\Tad$-module $T_{X_0}\tM^G_{\wm}$ is multiplicity-free. Its $\Tad$-weight set is
\begin{equation} \label{eq:Tadweightscase18}
\{\alpha_i \colon 1\le i \le L-1\} \cup \{\alpha'_j \colon 1\le j \le K-1\} \cup \{\alpha'_K+\alpha'_{K+1}+\ldots+\alpha'_{n-1}\}.
\end{equation}
In particular, $\dim T_{X_0}\tM^G_{\wm} = d_W$.  
\end{prop}
\begin{proof}
Call $F$ the set (\ref{eq:Tadweightscase18}). Let $J_1$ be the set defined in Proposition~\ref{prop:vgg18} and $J_2$ the set defined in Corollary~\ref{cor:vgg18}. Now, put
\[J:=
\begin{cases}
J_1 &\text{if $n=m-1$ and $\alpha_{m-2}+\alpha_{m-1} \in \wg$};\\
J_1 & \text{if $m=n-2>1$ and $\alpha'_{n-3} + \alpha'_{n-2} \in \wg$};\\
J_2 &\text{otherwise}.
\end{cases}
\]
Corollary~\ref{cor:vgg18} proves that $\Vgg$ is a multiplicity-free $\Tad$-module, that its $\Tad$-weight set $D$ contains $J$ and that $D \inn J \cup F$. 
Applying Proposition~\ref{prop:excl} with $v$ and $\lambda$ given in the table below, one then proves that the sections of $H^{0}(G\cdot x_0, \shN_{X_0})^G \isom \Vgg$ corresponding to the $\Tad$-weights in $J$ do not extend to $X_0$. We omit the straightforward verifications that the four conditions of Proposition~\ref{prop:excl} are met in every case (they are similar to the proofs of 
Lemmas~\ref{lem:exclbetar}, \ref{lem:exclbetaprj}, \ref{lem:exclbetam1} and \ref{lem:exclbetan2}).  
In this table, the integers $b$ and $a$ are as in Lemmas \ref{lem:betam1inlambda18} and \ref{lem:betan1inlambda18}.  

\noindent
\begin{tabular}{|c|c|c|c|}
\hline
conditions&$\Tad$-weight & $v$ & $\lambda$ \\
\hline
$2\le r \le K-1$&$\beta_r:= \alpha_{r-1}+\alpha_r$ & $X_{-\beta_r}( v_{\lambda_{r}}  + v_{\lambda'_{r}})$ &$\lambda_{r-1}$ \\
\hline
$3\le n\le m$&$\beta_{n-1}:=\alpha_{n-2}+\alpha_{n-1}$ & $X_{-\beta_{n-1}}( v_{\lambda_{n-1}}  + v_{\lambda'_{n-1}})$ &$\lambda_{n-2}$ \\
\hline
$2\le s\le K-1$&$\beta'_s:=\alpha'_{s-1}+\alpha'_s$ & $X_{-\beta'_s}(v_{\lambda_{s+1}}+v_{\lambda'_s})$ & $\lambda'_{s-1}$\\
\hline
$1<n-1\le m$&$\beta'_{n-1}:= \alpha'_{n-2}+\alpha'_{n-1}$ & $X_{-\beta'_{n-1}}(v_{\lambda_{n-1}}+v_{\lambda'_{n-2}})$ & $\mu$\\
\hline
$n=m-1\ge 2$&$\beta_{m-1}:=\alpha_{m-2}+\alpha_{m-1}$ & $X_{-\beta_{m-1}}v_{\lambda'_{m-1}} $ & $\begin{cases}
\lambda_{m-2} \text{ if $b \ge -1$} \\
\lambda'_{m-2} \text{ if $b<-1$}
\end{cases}$\\
\hline
$m=n-2>1$&$\beta'_{n-2}:=\alpha'_{n-3}+\alpha'_{n-2}$ & $X_{-\beta'_{n-2}}v_{\lambda'_{n-2}} $ & $\begin{cases}
\lambda'_{n-3} \text{ if $a \ge -1$} \\
\lambda_{n-2} \text{ if $a<-1$}
\end{cases}$\\
\hline
\end{tabular}

This shows that the $\Tad$-weight set of $T_{X_0}\tM^G_{\wm}$ is a subset of $F$. Equality follows, as always, from Corollary~\ref{cor:apriori}.
\end{proof}

As the arguments in this section are adaptations of those of Section~\ref{subsec:case17}, we do not provide all the proofs.  

\begin{prop} \label{prop:vgg18}
Suppose $m\ge 1, n\ge 2$. Let $F$ be the set (\ref{eq:Tadweightscase18}) and put
\begin{align*}
J_0&:=\{\alpha_{r-1}+\alpha_r \colon 2 \le r \le L-1 \} \cup \{\alpha'_{s-1} + \alpha'_s \colon 2 \le s \le K-1\}; \\
J_1&:= \begin{cases} 
J_0 \cup \{\alpha'_{n-2}+\alpha'_{n-1}\} &\text{if $m\ge n-1>1$ and $n \neq m-1$};\\
J_0 \cup \{\alpha_{m-2}+\alpha_{m-1}\} \cup \{\alpha'_{n-2}+\alpha'_{n-1} \} &\text{if $n=m-1>2$}; \\
J_0 \cup \{\alpha_1 + \alpha_2\} &\text{if $n=m-1=2$}; \\
J_0 \cup \{\alpha'_{n-3}+\alpha'_{n-2}\} &\text{if $m=n-2 > 1$}; \\
J_0 &\text{otherwise}.
\end{cases}
\end{align*}
The $\Tad$-module $\Vggp$ is multiplicity-free; its $\Tad$-weight set contains $J_1$ and is a subset of $F \cup J_1$. 

For the $\Tad$-weights in $J_0$, basis vectors for the corresponding eigenspaces in $\Vggp$ are given in the following table:

\begin{tabular}{|c|c|}
\hline
$\Tad$-weight & eigenvector \\
\hline
$\beta_r:=\alpha_{r-1} + \alpha_{r}$ & $[X_{-\beta_r} (v_{\lambda_{r-1}}+ v_{\lambda'_{r-1}})] = -[X_{-\beta_r}( v_{\lambda_{r}}  + v_{\lambda'_{r}})]$\\
\hline
$\beta'_s:=\alpha'_{s-1} + \alpha'_{s}$ & $[X_{-\beta'_s}(v_{\lambda_{s}}+v_{\lambda'_{s-1}})] = -  
[X_{-\beta'_s}(v_{\lambda_{s+1}}+v_{\lambda'_s})] $\\
\hline
\end{tabular}

\noindent with $2\le r \le L-1$, $2\le s \le K-1$.

If $m\ge n-1>1$ then the $\Tad$-weight space of $\Vggp$ of weight  $\alpha'_{n-2}+\alpha'_{n-1}$ is spanned by the following eigenvector: 

\begin{tabular}{|c|c|}
\hline
$\Tad$-weight & eigenvector \\
\hline
$\beta'_{n-1}:=\alpha'_{n-2} + \alpha'_{n-1}$ & $[X_{-\beta'_{n-1}}(v_{\lambda_{n-1}}+v_{\lambda'_{n-2}})] = -  
[X_{-\beta'_{n-1}}(v_{\lambda'_{n-1}}+v_{\mu})] $\\
\hline
\end{tabular}

If $n = m-1$ then the $\Tad$-weight space of $\Vggp$ of weight $\alpha_{m-2} + \alpha_{m-1}$ is spanned by the following eigenvector:

\begin{tabular}{|c|c|}
\hline
$\Tad$-weight & eigenvector \\
\hline
$\beta_{m-1}:= \alpha_{m-2} + \alpha_{m-1}$ & $[X_{-\beta_{m-1}}(v_{\lambda_{m-2}} + v_{\lambda'_{m-2}}) ] = -[X_{-\beta_{m-1}}v_{\lambda'_{m-1}}]$ \\
\hline
\end{tabular}

If $m=n-2>1$ then the $\Tad$-weight space of $\Vggp$ of weight $\alpha'_{n-3}+\alpha'_{n-2}$ is spanned by the following eigenvector:

\begin{tabular}{|c|c|}
\hline
$\Tad$-weight & eigenvector \\
\hline
$\beta'_{n-2}:=\alpha'_{n-3} + \alpha'_{n-2}$ & $[X_{-\beta'_{n-2}} (v_{\lambda_{n-2}} + v_{\lambda'_{n-3}})] = -[X_{-\beta_{n-2}} v_{\lambda'_{n-2}}]$ \\
\hline
\end{tabular}

\end{prop}

\begin{remark} \label{rem:vggp18}
We use the notation of Proposition~\ref{prop:vgg18}. The following somewhat stronger statement holds, but we do not need it in what follows. The $\Tad$-weight set of $\Vggp$ is equal to $F \cup J_1$ and below are basis vectors for the eigenspaces with weight in $F$. The argument is the same as that of Remark~\ref{rem:vggp17}. 

\begin{tabular}{|c|c|c|}
\hline
conditions & $\Tad$-weight & eigenvector \\
\hline
$1\le i \le L-1$ &$\alpha_i$ & $[X_{-\alpha_i} v_{\lambda_i}] = -[X_{-\alpha_i} v_{\lambda'_i}]$ \\
\hline
$1\le j \le K-1$ &$\alpha'_j$ & $[X_{-\alpha'_j}v_{\lambda_{j+1}} ] = -[X_{-\alpha'_j}v_{\lambda_j'}]$ \\
\hline
$n-1 \le m$&$\alpha'_{n-1}$ & $[X_{-\alpha'_{n-1}}v_{\lambda'_{n-1}} ] = -[X_{-\alpha'_{n-1}}v_{\mu}]$ \\
\hline
 $m \le n-2$ &$\alpha'_{m}+ \alpha'_{m+1} + \ldots + \alpha'_{n-1}$ & $[X_{-\gamma} v_{\lambda'_m}] = -[X_{-\gamma}v_{\mu}]$\\
\hline
\end{tabular}
\end{remark}

With a proof like that of Corollary~\ref{cor:vgg17} we have the following consequence of Proposition~\ref{prop:vgg18}. 
\begin{cor}  \label{cor:vgg18} We use the notation of Proposition~\ref{prop:vgg18}. 
Put  
\begin{equation*}
J_2 := \begin{cases}
J_0 \cup \{\alpha'_{n-2} + \alpha'_{n-1}\} &\text{if $1<n-1\le m$};\\
J_0 &\text{otherwise}.
\end{cases}
\end{equation*}
For all $m\ge 1, n\ge 2$, we have that $\Vggo$ is the subspace of $\Vggp$ spanned by the eigenvectors with $\Tad$-weight in $F \cup J_2$. 
Depending on $m$ and $n$, we have the following description of $\Vgg$:
\begin{enumerate}
\item Unless $n=m-1$ or $m=n-2>1$, we have that $\Vgg= \Vggo = \Vggp$;
\item If $n = m-1$, then $\Vgg = \Vggp$ if and only if $\beta_{m-1} \in \wg$. If $\beta_{m-1} \notin \wg$ then $\Vgg = \Vggo$;
\item  If $m=n-2>1$ then $\Vgg = \Vggp$ if and only if $\beta'_{n-2} \in \wg$. If $\beta'_{n-2} \notin \wg$ then $\Vgg = \Vggo$.
\end{enumerate}
\end{cor}

\begin{remark} \begin{enumerate}[(1)]
\item Using Remark~\ref{rem:vggp18}, the first assertion of Corollary~\ref{cor:vgg18} can be improved to the statement that $\Vggo$ is a multiplicity-free $\Tad$-module with $\Tad$-weight set $F\cup J_2$. 
\item For $n=m-1$, Lemma~\ref{lem:betam1inlambda18} below tells us for which intermediate groups $G$ the eigenvector in $\Vggp$ with weight $\beta_{m-1}$ belongs to $\Vgg$. When $m=n-2>1$, Lemma~\ref{lem:betan1inlambda18} does the same for $\beta'_{n-2}$. 
\end{enumerate}
\end{remark}

Since the proof of Proposition~\ref{prop:vgg18} is very similar to that of Proposition~\ref{prop:vgg}, we will not provide all details. We begin with a few lemmas, and then outline the rest of the proof on page~\pageref{proofofvgg18}.
 We will make use of the notation introduced in Section~\ref{subsec:case17} on page~\pageref{notation17}. 

\begin{lemma} \label{lem:nomix18}
The $\Tad$-weights occurring in $\Vggp$ belong to $(\Lambda_R^1 \oplus 0) \cup (0 \oplus \Lambda_R^2)$.
\end{lemma}
\begin{proof}
As in the proof of Lemma~\ref{lem:nomix}, we have to rule out $\Tad$-eigenvectors in $\Vggp$ of weight $\alpha_i + \alpha'_j$ where $\alpha_i$ is a simple root of $G^1$ and $\alpha'_j$ is a simple root of $G^2$.  Suppose, by contradiction, that $[v] \in \Vggp$ is such an eigenvector. 
Then $X_{\alpha'_j}v \in \CC X_{-\alpha_i} x_0$. 
As long as $i \le K$, we have that $X_{-\alpha_i} x_0 = X_{-\alpha_i}(v_{\lambda_i} + v_{\lambda'_i})$, which yields a contradiction because the $G^2$-modules $V( \omega'_{i-1})$ and $V(\omega'_i)$ cannot both contain a nonzero $\Tad^2$-eigenvector of weight the simple root $\alpha'_j$.

When $i > K$ we still have $i \le L$ because $X_{-\alpha_i}x_0 \neq 0$. So $i > K$ implies that $K=n-1, L=n$ and $i=n$. Then $X_{-\alpha_i} x_0 = X_{-\alpha_n}v_{\lambda'_n}$ which again yields a contradiction: $V(\omega'_n)$ contains no $\Tad^2$-eigenvectors of nonzero weight. 
\end{proof}

\begin{lemma} \label{lem:basislambda18}
We have 
\begin{equation*}  
\wgo = \<\omega_1,\ldots,\omega_L,\omega_1',\ldots,\omega_K', \mu\>_{\ZZ}
\end{equation*}
or, equivalently,
\begin{equation*}
\wgo = 
\begin{cases}
\<\omega_1,\ldots,\omega_n,\omega_1',\ldots, \omega'_n\>_{\ZZ} &\text{if $m>n-1$};\\
\<\omega_1,\ldots,\omega_m,\omega_1',\ldots, \omega'_n\>_{\ZZ} &\text{if $m=n-1$};\\
\<\omega_1,\ldots,\omega_m,\omega_1',\ldots,\omega_m', \omega'_{n-1}-\omega'_n\>_{\ZZ} &\text{if $m<n-1$}.
\end{cases} 
\end{equation*}
Moreover, for $i\in\{1,\ldots,K\}$ we have the following equalities in $X(\overline{T})$:
\begin{align*} 
\omega_i &= \lambda_i -\sum_{k=1}^{i-1} (\lambda_k' - \lambda_k);
\\
\omega'_i&= \sum_{k=1}^{i}(\lambda'_k-\lambda_k).\\
\intertext{When $m \ge n-1$ we have}
\omega'_{n} &= \sum_{k=1}^{n-1}(\lambda'_k-\lambda_k) - \mu\\
\intertext{
as well,
and when $m>n-1$ there is also}
\omega_n &= \lambda'_n + \mu - \sum_{k=1}^{n-1}(\lambda'_k-\lambda_k).
\end{align*}
\end{lemma}

We will make use of Lemma~\ref{lem:consecfund} but also of the following variant. Again, its proof is an adaptation of that of \cite[Corollary 3.9]{bravi&cupit}.
\begin{lemma} \label{lem:consecfundgap}
Suppose $m \ge 4$ is an integer and suppose $k \le m-3$ is another positive integer. Define the following $\SL(m)$-module:
\[ M := V(\omega_1) \oplus V(\omega_2) \oplus \ldots + V(\omega_k) \oplus V(\omega_{m-1})\]
Furthermore, call the sum of highest weight vectors $m_0$:
\[ m_0: = v_{\omega_1} + v_{\omega_2} + \ldots + v_{\omega_k} + v_{\omega_{m-1}} \]

Then $(M/\fg\cdot m_0)^{\SL(m)_{m_0}}$ is the multiplicity-free $\Tad$-module with weight set
\begin{equation} \label{eq:wsmg}
\{ \alpha_1+ \alpha_2, \alpha_2+\alpha_3, \ldots , \alpha_{k-2}+ \alpha_{k-1}, \alpha_k+\alpha_{k+1} + \ldots +\alpha_{m-1}\}
\end{equation}
\end{lemma}
\begin{proof} 
First, note that the monoid $\<\omega_1, \omega_2, \ldots, \omega_k, \omega_{m-1}\>_{\NN}$ is $\SL(m)$-saturated so  that the assumptions of Theorem~\ref{thm:bcf} are satisfied. 
Theorem 3.10 in~\cite{bravi&cupit} tells us that $(M/\fg\cdot m_0)^{\SL(m)_{m_0}}$ is a multiplicity-free $\Tad$-module. Therefore, $(M/\sl(m)\cdot m_0)^{\SL(m)_{m_0}}$ is a multiplicity-free $\Tad$-module whose weight set $F$ is a subset of the set $D$ in the proof of Lemma~\ref{lem:consecfund}. Just like in that proof we use the argument of \cite[Corollary 3.9]{bravi&cupit} to show that $F$ is the set (\ref{eq:wsmg}). 

Weights of type (\SRo) and (\SRt) do not occur in $F$ because the fundamental representations of $\SL(m)$ do not contain such $\Tad$-weights. 

Next suppose $\gamma = \alpha_{i} + 2\alpha_{i+1} + \alpha_{i+2}$ is a weight of the type (\SRf). Then $\<\gamma, \alpha_{i+1}^{\vee}\> = 2$ and so when $i > k-1$, we have $\gamma \notin  \<\omega_1, \ldots \omega_k, \omega_{m-1}\>_{\ZZ}$. If $i\le k-1$, then \cite[Proposition 3.4]{bravi&cupit} with $\delta = \alpha_i$ tells us that $\gamma$ does not belong to $F$. 

Now suppose $\gamma$ is a root of type (\SRth):
$\gamma=\alpha_{i+1} +\alpha_{i+2} + \ldots + \alpha_{i+r}$. 
First, let us assume $ r=2$ and $i+1\ge k-1$.  Then $\<\gamma, \alpha_{i+2}^{\vee}\> = 1$ and so $\gamma \notin  \<\omega_1, \ldots \omega_k, \omega_{m-1}\>_{\ZZ}$ for $k-1<i+1< m-2$. When $i+1 = k-1$ we can use that $\<\gamma, \alpha_{i+3}^{\vee}\> = -1$ to reach the same conclusion. When $i+1 = m-2$, the fact that $\<\gamma, \alpha_{i+1}^{\vee}\>=1$ does the trick.

Next we assume $r \ge 3$. If $k < i+r < m-1$, then $\<\gamma, \alpha_{i+r}^{\vee}\> = 1$ tells us that $\gamma \notin  \<\omega_1, \ldots \omega_k, \omega_{m-1}\>_{\ZZ}$. When $i+r \le k$, then \cite[Proposition 3.4]{bravi&cupit} with $\delta = \alpha_{i+r-1}$ tells us that $\gamma$ is not $F$. When $i+r = m-1$ and $i+1>k$, then $\<\gamma, \alpha_{i+1}^{\vee}\> = 1$ implies that $\gamma \notin  \<\omega_1, \ldots \omega_k, \omega_{m-1}\>_{\ZZ}$. When  $i+r = m-1$ and $i+1<k$ then  \cite[Proposition 3.4]{bravi&cupit} with $\delta = \alpha_{i+2}$ tells us that $\gamma$ is not in $F$.

Finally, that the weight set $F$ contains the weights of the form $\alpha_i + \alpha_{i+1}$  listed in (\ref{eq:wsmg}) follows exactly like in the proof of \cite[Corollary 3.9]{bravi&cupit}. For the weight $\gamma=\alpha_k+\ldots \alpha_{m-1}$, a weight vector is $[X_{-\alpha_k}X_{-\gamma+\alpha_{k}} m_0] =  [X_{-\gamma+\alpha_{k}}X_{-\alpha_k} m_0] 
\in M/\fg\cdot m_0$.
\end{proof}

\begin{proof}[Outline of proof of Proposition~\ref{prop:vgg18}] \label{proofofvgg18}
By the same arguments as in Section~\ref{subsec:case17} we have a decomposition
\begin{equation}
\Vggp = \Vggp_{\Lambda^1_R} \oplus \Vggp_{\Lambda^2_R}
\end{equation}
compatible with the action of $\Tad = \Tad^1 \times \Tad^2$, that the injection $A:=V^{U^2} \into V$ induces an isomorphism of $\Tad^1$-modules
\begin{align*}
\Bigl(\frac{A}{A\cap \fg\cdot x_0}\Bigr)^{G^1_{x_0}} &\isom \Vggp_{\Lambda^1_R}\\
\intertext{and that the injection $C:=V^{U^1}  \into V$ induces an isomorphism of $\Tad^2$-modules}
\Bigl(\frac{C}{C\cap \fg\cdot x_0}\Bigr)^{G^2_{x_0}} &\isom \Vggp_{\Lambda^2_R}.
\end{align*}
We therefore first determine the $\Tad^1$-module $\bigl(\frac{A}{A\cap \fg\cdot x_0}\bigr)^{G^1_{x_0}}$ and then the $\Tad^2$-module $\bigl(\frac{C}{C\cap \fg\cdot x_0}\bigr)^{G^2_{x_0}}$.

To do so, we begin by introducing certain $G^1$-submodules of $A$: for $i\in\{1,\ldots, K\}$, put
\[Z_i := \text{the simple $G^1$-submodule of $A$ with highest weight vector $z_i:=v_{\lambda_i} + v_{\lambda'_i}$}.\]
When $L = K+1$, that is, if $m\ge n$, also put
\[Z_L:= \text{the simple $G^1$-submodule of $A$ with highest weight vector $z_L:=v_{\lambda'_L}$}.\]
We also define the following trivial $G^1_{x_0} \rtimes \Tad^1$-submodule of $A$:
\[Z_0 := \CC(v_{\lambda_1}-v_{\lambda'_1}) \oplus \ldots \oplus \CC(v_{\lambda_K}-v_{\lambda'_K}) + \CC v_{\mu} \]
and the $G^1_{x_0} \rtimes \Tad^1$-submodule of $A$:
\[Z:=Z_0 \oplus Z_1 \oplus \ldots \oplus Z_L.\]
Then $A \cap \fg \cdot x_0 \inn Z \inn A$ and therefore, we again obtain an exact sequence of $G^1_{x_0} \rtimes \Tad^1$-modules like (\ref{eq:exactseqA}) on page~\pageref{eq:exactseqA}, and consequently an exact sequence of $\Tad^1$-modules like (\ref{eq:seqAinv}). 

To determine  the $\Tad^1$-module $\Bigl(\frac{Z}{A\cap \fg\cdot x_0}\Bigr)^{G^1_{x_0}}$ we introduce
\begin{align*}
y_0&:=\begin{cases}
z_1 + \ldots + z_L &\text{if $L<m$ (i.e.~$n<m$)};\\
z_1 + \ldots + z_{L-1} &\text{if $L=m$ (i.e.~$n\ge m$)};
\end{cases}\\
\overline{Z_0}&: =\begin{cases}
Z_0 &\text{if $L<m$ (i.e.~$n<m$)}; \\
Z_0 + Z_L &\text{if $L=m$ (i.e.~$n\ge m$)};
\end{cases}\\
\overline{Z}&:=\begin{cases}
Z_1 \oplus \ldots \oplus Z_L &\text{if $L<m$ (i.e.~$n<m$)};\\
Z_1 \oplus \ldots \oplus Z_{L-1} &\text{if $L=m$ (i.e.~$n\ge m$)}.
\end{cases}
\end{align*}
We then obtain that $Z=\overline{Z} \oplus \overline{Z_0}$ and $A\cap \fg\cdot x_0 = \fg^1\cdot y_0 \oplus \overline{Z_0}$ and so the inclusion $\overline{Z} \into Z$ induces an isomorphism of $\Tad^1$-modules
\[\Bigl(\frac{\overline{Z}}{\fg^1\cdot y_0}\Bigr)^{G^1_{x_0}} \isom \Bigl(\frac{Z}{A\cap \fg\cdot x_0}\Bigr)^{G^1_{x_0}}.\]
The $\Tad^1$-module on the left is determined by Lemma~\ref{lem:consecfund}. It is multiplicity-free and its $\Tad^1$-weights are
\begin{align*}
&\alpha_1+\alpha_2, \alpha_2+\alpha_3,\ldots,\alpha_{L-2} + \alpha_{L-1} &&\text{if $L \neq m-1$};\\
&\alpha_1+\alpha_2, \alpha_2+\alpha_3,\ldots,\alpha_{m-2} + \alpha_{m-1} &&\text{if $L= m-1$ (i.e.~ if $n=m-1)$}.
\end{align*}

Next we determine the $\Tad^1$-module $\Bigl(\frac{A}{Z}\Bigr)^{G^1_{x_0}}$. To do so, we introduce
\begin{align*}
A_i&:= \<G^1\cdot(v_{\lambda_i} -v_{\lambda'_i})\>_{\CC} \quad \quad \text{for $i\in\{1,2,\ldots,L-1\}$};\\
\overline{A}&:= \begin{cases}
A_1 \oplus\ldots\oplus A_{L-1} &\text{if $n<m$};\\
A_1\oplus\ldots\oplus A_{L-1} \oplus Z_m &\text{if $n=m$};\\
A_1 \oplus\ldots\oplus A_{L-1} \oplus \<G^1\cdot v_{\lambda_m}\>_{\CC} \oplus  \<G^1\cdot v_{\lambda'_m}\>_{\CC} &\text{if $n>m$}.
\end{cases}
\end{align*}
Then $A = \overline{A} \oplus \overline{Z}$ and $\overline{Z_0} = \overline{A}^{U^1}$. It follows that 
\[\frac{A}{Z} \isom  \frac{\overline{A}}{\overline{Z_0}} \isom \frac{A_1 \oplus \ldots \oplus A_{L-1}}{(A_1 \oplus \ldots \oplus A_{L-1})^{U^1}}\]
and so $\Bigl(\frac{A}{Z}\Bigr)^{G^1_{x_0}}$ is a multiplicity-free $\Tad^1$-module whose weights  are 
\[\alpha_1, \alpha_2, \ldots, \alpha_{L-1}.\]

We now move to the  the $\Tad^2$-module $\bigl(\frac{C}{C\cap \fg\cdot x_0}\bigr)^{G^2_{x_0}}$,
where $C=V^{U^1}$.  We put
\begin{align*}
z_i &:= v_{\lambda_{i+1}}+ v_{\lambda_i} \quad \quad \text{for $i=1,\ldots, K-1$};\\
z_K &:= \begin{cases}v_{\lambda'_K} + v_{\mu} &\text{if $K=n-1$ (i.e.~$n-1 \le m$)};\\
v_{\lambda'_k} &\text{if $K<n-1$ (i.e.~$n-1>m$)};\end{cases}\\
Z_i &:= \<G^2 \cdot z_i\>_{\CC}  \quad \quad \text{for $i=1,\ldots, K$};\\
Z_0 &:=\begin{cases}
\CC v_{\lambda_1} \oplus \CC(v_{\lambda_2} - v_{\lambda_1}) \oplus \ldots \oplus \CC(v_{\lambda_K}-v_{\lambda'_{K-1}}) \oplus  \\ \quad \quad \oplus \CC(v_{\lambda'_K} - v_{\mu}) \oplus \CC v_{\lambda'_{K+1}}  &\text{if $n \le m$};\\
\CC v_{\lambda_1} \oplus \CC(v_{\lambda_2} - v_{\lambda_1}) \oplus \ldots \oplus \CC(v_{\lambda_K}-v_{\lambda'_{K-1}}) \oplus \CC(v_{\lambda'_K} - v_{\mu}) &\text{if $n=m-1$};\\
\CC v_{\lambda_1} \oplus \CC(v_{\lambda_2} - v_{\lambda_1}) \oplus \ldots \oplus \CC(v_{\lambda_K}-v_{\lambda'_{K-1}}) &\text{if $m<n-1$};
\end{cases}\\
Z &:=\begin{cases}
Z_0 \oplus Z_1 \oplus \ldots \oplus Z_K &\text{if $K=n-1$};\\
Z_0 \oplus Z_1 \oplus \ldots \oplus Z_K \oplus V(\mu) &\text{if $K<n-1$}.
\end{cases}
\end{align*}
Because $C \cap \fg \cdot x_0 \inn Z \inn C$, we again obtain an exact sequence of $\Tad^2$-modules like (\ref{eq:seqCinv}) on page~\pageref{eq:seqCinv}, and so we determine the $\Tad^2$-modules $\Bigl(\frac{Z}{C\cap \fg\cdot x_0}\Bigr)^{G^2_{x_0}}$ and $\Bigl(\frac{C}{Z}\Bigr)^{G^2_{x_0}}$. 

For the first, put 
\begin{align*}
y_0&:= \begin{cases}
z_1 + z_2 + \ldots + z_K &\text{if $K=n-1$};\\
z_1 + z_2 + \ldots +z_K + v_{\mu} &\text{if $K<n-1$};
\end{cases}\\
\overline{Z} &:= \begin{cases}
Z_1 \oplus \ldots \oplus Z_K &\text{if $K=n-1$}; \\
Z_1 \oplus \ldots \oplus Z_K \oplus V(\mu) &\text{if $K<n-1$}.
\end{cases}
\end{align*}
Then 
\[\Bigl(\frac{Z}{C\cap \fg\cdot x_0}\Bigr)^{G^2_{x_0}} \isom \Bigl(\frac{\overline{Z}}{\fg^2\cdot y_0}\Bigr)^{G^2_{y_0}}\]
and the latter $\Tad^2$-module is described by Lemma~\ref{lem:consecfund} or Lemma~\ref{lem:consecfundgap} depending on $K$. It is multiplicity-free and its $\Tad^2$-weights are
\begin{align*}
&\alpha'_1+\alpha'_2,\alpha'_2 +\alpha'_3, \ldots, \alpha'_{n-2}+\alpha'_{n-1} &&\text{if $K=n-1$ or $K=n-2$};\\
&\alpha'_1+\alpha'_2, \ldots, \alpha'_{K-2}+\alpha'_{K-1}, \alpha'_K+\alpha'_{K+1}+\ldots+\alpha'_{n-1}
&&\text{if $K<n-2$}.
\end{align*}
Finally, for $\Bigl(\frac{C}{Z}\Bigr)^{G^2_{x_0}}$ we put
\begin{align*}
C_i &:= \<G^2 \cdot(v_{\lambda_{i+1}} - v_{\lambda'_i})\>_{\CC} &&\text{for $i\in\{1,\ldots,K-1\}$};\\
C_{n-1}&:=\<G^2 \cdot (v_{\lambda'_{n-1}}-v_{\mu})\>_{\CC} &&\text{if $K=n-1$ (i.e. $n-1 \le m$)}
\end{align*}
and 
\begin{align*}
\overline{C}&:= \<G^2 \cdot Z_0\>_{\CC};\\
\widetilde{C} &:= \begin{cases}
C_1 \oplus C_2 \oplus \ldots \oplus C_{L-1} &\text{if $m \neq n-1$};\\
C_1 \oplus C_2 \oplus \ldots \oplus C_{L-1} \oplus C_{n-1} &\text{if $m=n-1$ (and $L=n-1$)}.
\end{cases}
\end{align*}
Then $C = \overline{C} \oplus \overline{Z}$, and because $Z = Z_0 \oplus \overline{Z}$
\[\frac{C}{Z} \isom \frac{\overline{C}}{Z_0} \text{ as $G^2_{x_0} \rtimes \Tad^2$-modules.}\]
Therefore 
\[\frac{C}{Z} \isom \frac{\overline{C}}{Z_0} \isom \frac{\widetilde{C}}{\widetilde{C}^{U^2}}\]
and so the $\Tad^2$-weights in the multiplicity-free $\Tad^2$-module
\[\Bigl(\frac{C}{Z}\Bigr)^{G^2_{x_0}} \isom \Bigl(\frac{\widetilde{C}}{\widetilde{C}^{U^2}}\Bigr)^{G^2_{x_0}}\]
are
\begin{align*}
&\alpha'_1, \ldots, \alpha'_{L-1} &&\text{if $m\neq n-1$}; \\
&\alpha'_1, \ldots, \alpha'_{L-1}, \alpha'_{n-1} &&\text{if $m=n-1$ (then $L=n-1$)}. 
\end{align*}
It is a straightforward matter to verify that the vectors listed in the proposition indeed belong to the eigenspaces. 
\end{proof}

Finally, in the next three lemmas, we express the $\Tad$-weights to which we apply Proposition~\ref{prop:excl} (in the proof of Proposition~\ref{prop:case18}) in terms of the basis $E$ of $\wgo$. The proofs are omitted as they are very similar to the proofs of Lemmas \ref{lem:betainlambda}, \ref{lem:betam1inlambda} and \ref{lem:betan1inlambda} in Section~\ref{subsec:case17}.
\begin{lemma} Using the notation introduced in Proposition~\ref{prop:vgg18}, we have that
\begin{align*}
\beta_r &= -\lambda'_{r-2} + \lambda_{r-1} + \lambda'_r - \lambda_{r+1} \quad \text{when $2\le r \le K-1$};
\label{eq:betarcase18}\\
\beta'_s &=  -\lambda_{s-1} + \lambda'_{s-1} + \lambda_{s+1} - \lambda'_{s+1} \quad \text{when $2\le s\le K-1$},
\end{align*}
where $\lambda'_0=0$. 
If $3\le n\le m$ (then $L=n>K$) and
\begin{equation*} \label{eq:betanm1case18}
\beta_{L-1}= \beta_{n-1} = \lambda_{n-2} + \lambda'_{n-1} - \mu - \lambda'_n - \lambda'_{n-3}
\end{equation*}
where $\lambda'_0 := 0$ if $n=3$. 
If $n-1 \le m$ and $n\neq 2$, then
\begin{equation*}
\beta'_{n-1} = -\lambda_{n-2} + \lambda'_{n-2} + \mu.
\end{equation*}
\end{lemma}

We now come to the two $\Tad$-weights that only occur in $\Vgg$ for certain groups $G$ between $G'$ and $\overline{G}$. 
\begin{lemma} \label{lem:betam1inlambda18} We use the notation of Proposition~\ref{prop:vgg18} and suppose $n=m-1$. Then the following are equivalent (recall that, by assumption, $(G,W)$ is spherical):
\begin{enumerate}
\item $\beta_{m-1} \in \wg$;
 \item 
  $T = \ker(\omega_m - b\omega'_{n}) \inn \overline{T}$ for some integer $b$. 
  \end{enumerate}
For every integer $b$ we have the following equality in $X(\overline{T})$
 \begin{multline} \label{eq:betaminlambda18up}
\beta_{m-1} +(\omega_m-b\omega'_n)= \lambda'_{m-1} +(1+b)\mu +(b+2)(\lambda_{m-2}-\lambda'_{m-3}) \\
- (b+1)[\lambda'_{m-2}-\lambda_{m-3} + \sum_{k=1}^{m-4}(\lambda'_k - \lambda_k) ],
\end{multline} 
where $\lambda'_{m-3} = \lambda_{m-3} = 0$ when $m=3$. 
 Consequently, if $T=\ker(\omega_m - b \omega'_n)$ for some integer $b$, restricting (\ref{eq:betaminlambda18up}) to $T$ yields the following equality in $\wg$:
\begin{multline*}
\beta_{m-1} = \lambda'_{m-1} +(1+b)\mu +(b+2)(\lambda_{m-2}-\lambda'_{m-3}) \\
- (b+1)[\lambda'_{m-2}-\lambda_{m-3} + \sum_{k=1}^{m-4}(\lambda'_k - \lambda_k) ],
\end{multline*}
where $\lambda'_{m-3} = \lambda_{m-3} = 0$ when $m=3$. 
\end{lemma}

\begin{lemma} \label{lem:betan1inlambda18}
We use the notation of Proposition~\ref{prop:vgg18} and suppose $m=n-2>1$. Then 
the following are equivalent (recall that, by assumption, $(G,W)$ is spherical):
\begin{enumerate}
\item $\beta'_{n-2} \in \Lambda_{(G,W)}$;
\item $T = \ker(a\omega_m - \omega'_{n}) \inn \overline{T}$ for some integer $a$. 
\end{enumerate}
For every integer $a$ we have the following equality in $X(\overline{T})$:
\begin{multline} \label{eq:betanm218up}
\beta'_{n-2} -(a\omega_m-\omega'_n) = \lambda'_{n-2}-\mu -(1+a) [\lambda_{n-2}-\sum_{k=1}^{n-4}(\lambda'_k-\lambda_k)]\\
+ (2+a)(\lambda'_{n-3}-\lambda_{n-3}).
\end{multline}
Consequently, if $T=\ker(a\omega_m-\omega'_n)$ for some integer $a$, restricting (\ref{eq:betanm218up}) to $T$ yields the following equality in $\wg$: 
\begin{equation*} \label{eq:betanm218}
\beta'_{n-2} = \lambda'_{n-2}-\mu -(1+a) [\lambda_{n-2}-\sum_{k=1}^{n-4}(\lambda'_k-\lambda_k)]
+ (2+a)(\lambda'_{n-3}-\lambda_{n-3}).
\end{equation*}
\end{lemma}

\subsection{The modules $(\GL(m) \times \SL(2) \times \GL(n), (\CC^m \otimes \CC^2)\oplus (\CC^2 \otimes \CC^n))$ with 
$2\leq m \leq n$}
Here 
\begin{align*}
&E = \{\omega_1 + \omega', \omega' + \omega''_1, \omega_1 + \omega''_1, \omega_2, \omega''_2\};\\
&d_W =3 .
\end{align*} 
In this case $\wm$ is not $G$-saturated for any group $G$ for which $W$ is spherical as one easily checks using Lemma~\ref{lem:bcfcritpsat}. The module $W$ is spherical for $G'$ if and only if $m>2$.  

In this section, we prove the following proposition.

\begin{prop}\label{prop:case21}
The $\Tad$-module $\Vggp$ is multiplicity-free and its $\Tad$-weight set is $\{\alpha_1, \alpha', \alpha_1''\}$.  In particular, $\dim \Vggp = d_W$. Consequently, $\dim T_{X_0}\tM^G_{\wm} = d_W$.
\end{prop}
The proof will be given after a few lemmas we need.
We introduce some notation. $G^1 := \SL(m)$, $G^2 := \SL(2)$ and $G^3:= \SL(n)$, so that $G'=G^1 \times G^2 \times G^3$ is the semisimple part of $\overline{G}$. For $i=1,2,3$ we denote  $T^i$ the projection of  the maximal torus $T' = T\cap G$ of  $G'$ to $G^i$. Then $T^i$ is a maximal torus of $G^i$. Note that $\Tad = \Tad^1 \times \Tad^2 \times \Tad^3$ for the adjoint tori $\Tad^i$ of $G_i$ ($i=1,2,3$). Let  $\Lambda^i_R$ be the root lattice of $G_i$, and note that $\Lambda_R = \Lambda_R^1 \oplus \Lambda_R^2 \oplus \Lambda_R^3$. 

\begin{lemma} \label{lem:nomix21}
The $\Tad$-weights occurring in $\Vggp$ belong to $(\Lambda_R^1 \oplus 0 \oplus 0) \cup (0 \oplus \Lambda_R^2 \oplus 0) \cup (0 \oplus 0 \oplus \Lambda^2_R)$.
\end{lemma}
\begin{proof}
By the same argument as in Lemma~\ref{lem:nomix}, if $[v] \in \Vggp$ is a $\Tad$-eigenvector contradicting the lemma, then its weight is $\sigma + \sigma'$, where $\sigma$ is a simple root for $G_i$ and $\sigma'$ is a simple root of $G_j$ with $1\le i< j \le 3$ and moreover 
$X_{-\sigma} v$ is a nonzero element of  the line spanned by $X_{\sigma'} x_0$. 

Looking at the set $E$, the only simple roots $\sigma'$ so that $X_{\sigma'} x_0 \neq 0$ are $\alpha_1, \alpha_2, \alpha', \alpha_1'', \alpha_2''$. 
We can immediately rule out $\sigma'= \alpha_2$ and $\sigma'=\alpha_2''$ since all the $\Tad$-weights in $V(\omega_2)$ belong to $\Lambda^1_R$ and those in $V(\omega''_2)$ belong to $\Lambda^2_R$,
We can also assume $\sigma' \neq \alpha_1$ because $j>1$. 
 
Next, $\sigma' = \alpha'$ also leads to a contradiction. Indeed, 
\[X_{\alpha'} x_0 =  v_{\omega_1} \otimes (X_{\alpha'} v_{\omega'}) + (X_{\alpha'} v_{\omega'})  \otimes v_{\omega_1''}\]
and there is no simple root (which would be $\sigma$) that occurs as a $\Tad$-weight in both $V(\omega_1)$ and $V(\omega_1'')$. An analogous argument excludes $\sigma' = \alpha_1''$.
\end{proof}

\begin{lemma} \label{lem:noomone}
None of $\omega_1, \omega'$ and $\omega_1''$ belong to
\[ \<\omega_1 + \omega', \omega'+\omega_1'', \omega_1+\omega_1''\>_{\ZZ} \inn X(T').\]
\end{lemma}
\begin{proof}
Put $\lambda_1 := \omega_1 + \omega'$,  $\lambda_2 := \omega'+\omega_1''$, $\lambda_3 :=  \omega_1+\omega_1''$ and $\Gamma := \<\lambda_1, \lambda_2, \lambda_3\>_{\ZZ}$.

Put \[
A:= \begin{pmatrix}1&0&1 \\
1&1&0 \\
0&1&1
  \end{pmatrix}\]
so that $\begin{pmatrix}\lambda_1 & \lambda_2 & \lambda_3 \end{pmatrix} = 
\begin{pmatrix} 
\omega_1 & \omega' & \omega_1''
\end{pmatrix}
\cdot A.
$
Since $\det(A)=2$, $\Gamma$ is a strict subgroup of $\<\omega_1, \omega', \omega_2''\>_{\ZZ}$.

Now, if $\omega_1$ were an element of $\Gamma$, then so would $\omega' = \lambda_2-\omega_1$ and $\omega_1''= \lambda_3-\omega_1$, contradicting that the inclusion $\Gamma \inn \<\omega_1,\omega', \omega_1''\>_{\ZZ}$ is strict. By the same argument, $\omega'$ and $\omega_1''$ do not belong to $\Gamma$. 
\end{proof}

\begin{proof}[Proof of proposition~\ref{prop:case21}]
We have that
\[p(\wg) = \<\omega_1+\omega', \omega'+\omega_1'', \omega_1+\omega_1'', \omega_2, \omega_2''\>_{\ZZ} \inn X(T'),\]
where $\omega_2 = 0$ if $m=2$ and $\omega_2''=0$ if $n=2$. 
By Lemma~\ref{lem:nomix21} we know that the $\Tad$-weights in $\Vggp$ belong to $\Lambda^1_R$, $\Lambda^2_R$ or $\Lambda^3_R$. 

We start by considering the $\Tad$-weights in $V$ that belong to $\Lambda^1_R$. They are
\begin{align*}
&\alpha_1, \alpha_1+\alpha_2, \ldots, \alpha_1 + \alpha_2 + \ldots \alpha_{m-1},\\
&\alpha_2+\alpha_3,\alpha_2+ \alpha_3+\alpha_4, \ldots ,\alpha_2 + \ldots + \alpha_{m-1}
\end{align*}
Using that the image of $p(\wg)$ under $X(T') \onto X(T^1)$ is a subgroup of $\<\omega_1, \omega_2\>_{\ZZ}$ we see that among these $\Tad^1$-weights in $V$ only the following \emph{can} belong to $p(\wg)$:
\begin{align*}
&\alpha_1 && \text{if $m\neq 3$};\\
&\alpha_1, \alpha_1+ \alpha_2 &&\text{if $m=3$}. 
\end{align*}
Furthermore, even when $m=3$, $\alpha_1 + \alpha_2 \notin p(\wg)$. Indeed, since
\[\alpha_1 = (\omega_1+\omega') + (\omega_1+\omega_2'') - (\omega'+\omega_1'')-\omega_2\] 
belongs to $p(\wg)$, since $\omega_2 \in p(\wg)$, and since we know from Lemma~\ref{lem:noomone} that $\omega_1 \notin p(\wg)$ and so 
 $\alpha_2=-\omega_1+2\omega_2 \notin p(\wg)$, 
it follows that $\alpha_1+ \alpha_2 \notin p(\wg)$. 
This proves that for all $m$ and $n$, the only possible $\Tad^1$-weight in $\Vggp$ is $\alpha_1$. Since the eigenspace of $V$ of weight $\alpha_1$ has dimension $2$ and the eigenspace of $\fg\cdot x_0$ of that weight has dimension $1$, $\alpha_1$ occurs with multiplicity at most $1$ in $\Vggp$. 

The argument for the $\Tad$-weights in $\Lambda_R^3$ is identical. For those in $\Lambda_R^2$ it is even simpler. This proves that the $\Tad$-module $\Vggp$ is multiplicity-free and that its $\Tad$-weight set is a subset of $\{\alpha_1, \alpha', \alpha_1''\}$. As always, equality follows from  
Corollary~\ref{cor:apriori}. 
\end{proof}
\begin{remark}
The $\Tad$-eigenspace of $\Vggp$ of weight $\alpha_1$ is spanned by the vector
\[[X_{-\alpha_1}v_{\omega_1+\omega'}] = -[X_{-\alpha_1}v_{\omega_1+\omega_1''}] \in \Vg.\]
Indeed, it is not hard to verify that the vector is fixed by $\fg'_{x_0}$. Clearly, it has $\Tad$-weight $\alpha_1$. The other two eigenspaces have similar descriptions.  
\end{remark}

\section{Acknowledgements}
The authors thank S\'ebastien Jansou for introducing them to this beautiful subject, and Morgan Sherman for useful discussions at the start of this project. We thank Michel Brion for  several informative conversations and in particular for suggesting the strategy to prove that certain sections of the normal sheaf do not extend (Section~\ref{sec:Brionstrat}). Finally, we thank the referee for several valuable comments and suggestions. We also benefited from experiments with the computer algebra program {\it Macaulay2}~\cite{M2}. 

S. P. was supported by the Portuguese Funda\c{c}\~ao para a Ci\^encia e a
Tecno\-lo\-gia through Grant
SFRH/BPD/22846/2005 of POCI2010/FEDER and through Project 
\newline PTDC/MAT/099275/2008.

B.~V.~S. received support from the Portuguese Funda\c{c}\~ao para a Ci\^encia e a
Tecno\-lo\-gia through Grant SFRH/BPD/21923/2005 and through Project POCTI/FEDER, as well as from The City University of New York PSC-CUNY Research Award Program.


\def\cprime{$'$}
\providecommand{\bysame}{\leavevmode\hbox to3em{\hrulefill}\thinspace}
\providecommand{\MR}{\relax\ifhmode\unskip\space\fi MR }
\providecommand{\MRhref}[2]{%
  \href{http://www.ams.org/mathscinet-getitem?mr=#1}{#2}
}
\providecommand{\href}[2]{#2}

\end{document}